\documentclass[11pt,oneside]{article}
\usepackage[letterpaper]{geometry}
\usepackage{amsmath}
\usepackage{amsthm}
\usepackage{amssymb}
\usepackage{mathrsfs}
\usepackage{dsfont}
\usepackage{bbm}
\usepackage{lineno}
\usepackage{xcolor}
\usepackage[normalem]{ulem}
\usepackage{charter}
\usepackage{amsfonts}
\usepackage{graphicx}
\usepackage{natbib}

\usepackage{caption}
\usepackage{subcaption}

\makeatletter
\theoremstyle{plain}
\newtheorem{thm}{\protect\theoremname}[section]
\newtheorem{lem}[thm]{\protect\lemmaname}
\newtheorem{cor}[thm]{\protect\corollaryname}
\newtheorem{prop}[thm]{Proposition}

\theoremstyle{definition}
\newtheorem{defn}[thm]{\protect\definitionname}
\newtheorem{assumption}[thm]{\protect\assumptionname}
\newtheorem{example}[thm]{\protect\examplename}

\newtheorem{axiom}{Axiom}

\theoremstyle{remark}
\newtheorem{rem}[thm]{\protect\remarkname}

\theoremstyle{plain}
\newtheorem*{question*}{\protect\questionname}

\providecommand{\assumptionname}{Assumption}
\providecommand{\examplename}{Example}
\providecommand{\definitionname}{Definition}
\providecommand{\corollaryname}{Corollary}
\providecommand{\lemmaname}{Lemma}
\providecommand{\questionname}{Question}
\providecommand{\remarkname}{Remark}
\providecommand{\theoremname}{Theorem}

\renewcommand{\Re}{\mathbb{R}}

\newcommand{\indpart}[1]{\smallskip\noindent\emph{#1.}\ }

\title{History-Dependent Recursive Preferences in Markov Decision Processes}
\author{William B. Haskell}
\date{\today}

\begin{document}

\maketitle

\begin{abstract}
In finite horizon dynamic programming with history-dependent preferences, the relevant state may be the entire realized history, even when the physical state is Markov. This paper develops a behavioral state-reduction theory for such Markov decision processes.
Under behavioral axioms and a certainty-equivalent richness condition, the full-history problem admits a recursive representation composed of time and risk aggregators.
We then derive a canonical preference-augmented (PA) state by quotienting histories that have the same current physical Markov state, are indifferent under every common continuation plan, and remain equivalent after every common one-step extension.
This canonical PA state is minimal among reachable recursive factorizations of the underlying preferences. Under Markov feasibility and standard dynamic-programming regularity, a PA Bellman selector induces an optimal full-history policy.
With additional rectangularity and exhaustiveness conditions, we reparameterize the preference memory into distinct belief and taste coordinates, and obtain a separated representation and Bellman recursion.
We give a taxonomy of examples to illustrate the scope of our framework.
\end{abstract}

\section{Introduction}

We study finite-horizon Markov decision processes (MDPs) where the decision maker's (DM's) current time preferences and risk attitudes depend on the history of past states and consumption.
For example, DM's preferences may change as his wealth or habit stock changes. Preference models based only on the current physical Markov state ignore such behavior. Under unrestricted history dependence, the state dimension grows with the horizon and dynamic programming (DP) becomes increasingly difficult.

We develop a class of dynamic preferences between history independence and unrestricted full-history dependence.
We first identify when DM's preferences endogenously determine a preference-augmented (PA) state that may be substantially smaller than the entire history.
Under additional rectangularity and exhaustiveness assumptions on the preference-memory component, we then obtain a separated preference-augmented (SPA) state with distinct belief and taste coordinates.
Our framework is also constructive: a PA or SPA state and recursive structure induce full-history preferences.

\subsection{Contributions}

\textit{Axiomatic foundation.} Under behavioral axioms and a certainty-equivalent richness condition, we derive a recursive representation for full-history preferences (Theorem~\ref{thm:recursive_utility}).
We identify time and risk aggregators on compact effective domains and prove that these effective aggregators admit jointly continuous monotone global extensions, which constitute the recursive representation. This representation underpins the rest of our development.

\textit{Endogenous memory augmentation.} We derive the preference-memory component of the PA state, rather than choosing it exogenously (e.g., wealth, a habit stock, a posterior belief). The endogenous quotient construction aggregates histories with the same current physical state that are behaviorally indistinguishable.
It also yields compact metrizable PA-state spaces, continuous transition maps, and a recursive factorization of DM's preferences (Theorem~\ref{thm:recursive_memory}). Furthermore, the canonical quotient is the minimal preference-memory augmentation over the physical Markov state among reachable recursive factorizations (Proposition~\ref{prop:minimality}).

\textit{Dynamic programming.} When feasibility factors through a PA state and additional continuity and compactness conditions on the consumption constraints hold, the optimal value functions satisfy a Bellman recursion on this PA state (Theorem~\ref{thm:DP_Bellman}). A PA Bellman selector with respect to this recursion induces a full-history optimal policy (Theorem~\ref{thm:DP_verification}).

\textit{Belief and taste separation.} DM's beliefs and tastes may be entangled because the time and risk aggregators can both depend on the PA state.
Under additional rectangularity and exhaustiveness conditions, we separate the preference-memory component into distinct belief and taste coordinates, yielding separated preference-augmented (SPA) states.
The belief coordinate contains the memory relevant for subjective prediction, risk evaluation, and ambiguity evaluation, while the taste coordinate contains the memory relevant for current utility and time preference.
We derive a separated recursive representation (Theorem~\ref{thm:memory_recursive_separated}) and Bellman recursion (Corollary~\ref{cor:DP_separated}).

\textit{Taxonomy.} Our PA and SPA framework is modular: it accommodates general time and risk aggregators and a variety of memory structures.
Our taxonomy covers five groups of dynamic preference models: history-independent preferences; SPA preferences with only one active memory coordinate (belief or taste but not both); SPA preferences with decoupled transitions; SPA preferences with coupled transitions; and entangled PA preferences, where the same preference-memory state enters both the time and risk aggregators.

\subsection{Related Literature}

This paper draws from four literature streams: (i) recursive preferences; (ii) history-dependent preferences; (iii) risk-aware and robust DP; and (iv) information states in stochastic control.

\emph{Recursive preferences.}
Recursive utility models decompose intertemporal preferences into current consumption and future continuation utility.
Koopmans~\cite{koopmans1960stationary} initiated the study for deterministic consumption streams, and Kreps and Porteus~\cite{kreps1978temporal} extended it to temporal payoff lotteries.
Subsequent work studies further separations of risk, time, and ambiguity attitudes, including \cite{epstein1989substitution,chew1991recursive,hayashi2005intertemporal,maccheroni2006dynamic,klibanoff2009recursive,strzalecki2011axiomatic,sarver2018dynamic}.
Ren and Stachurski~\cite{ren2018dynamic} develop DP methods for recursive preferences, and Stachurski and Zhang~\cite{stachurski2021dynamic} treat state-dependent discounting.
Bommier et al.~\cite{bommier2017monotone} characterize recursive preferences that are monotone with respect to statewise dominance. Notably, they show that imposing this kind of monotonicity greatly restricts admissible certainty equivalents.
Our modular Bellman recursion with general monotone aggregators fits the monotone abstract DP framework~\cite{bertsekas2012dynamic,sargent2023completely}.

Several recent works on recursive preferences emphasize ambiguity attitudes.
Dynamic ambiguity-averse preferences include the multiple-priors model \cite{epstein2003recursive} and dynamically consistent choice under ambiguity~\cite{siniscalchi2011dynamic}. Strzalecki~\cite{strzalecki2013temporal} shows how ambiguity aversion interacts with preferences for the temporal resolution of uncertainty. Cerreia-Vioglio et al.~\cite{cerreia2022ambiguity} study wealth-dependent ambiguity attitudes.
Marinacci et al.~\cite{marinacci2026recursive} relate recursive and ex-ante representations of ambiguity preferences through generalized rectangularity conditions. These works all take the underlying state space as a primitive; we take a complementary approach and derive it endogenously from preferences.

\emph{History dependence.}
Recursive preferences may depend, for example, on past consumption~\citep{campbell1999force,rozen2010foundations} or private information~\citep{fernandes2000recursive,marcet2019recursive}. Schroder and Skiadas~\cite{schroder2002isomorphism} and Backus et al.~\cite{backus2004exotic} convert history-dependent asset-pricing problems into recursive state-variable models.
Dillenberger and Rozen~\cite{dillenberger2015history} axiomatize risk attitudes that depend on the history of past disappointment and elation relative to endogenously formed targets. Tserenjigmid~\cite{tserenjigmid2019history} studies history-dependent risk aversion, and Frick et al.~\cite{frick2019dynamic} characterize history dependence in dynamic stochastic choice models.
These studies examine specific channels of history dependence; we ask how general history dependence compresses into a preference-memory state.

\emph{Risk-aware MDPs.}
Risk-sensitive control of MDPs originated in \cite{howard1972risk}.
The risk-aware MDP literature later studies dynamic risk mappings and their connection to robustness; see, e.g., \cite{ruszczynski2006conditional,ruszczynski2010risk,Iyengar_2005,Nilim_Ghaoui_2005,mannor2016robust,osogami2011iterated,chow2015risk,li2017quantile}, the survey of time consistency in \cite{bielecki2017survey}, the temporal decomposition of risk functionals in \cite{pflug2016time}, and the monographs \cite{shapiro2021lectures,dentcheva2024risk}.

Bäuerle and Jaśkiewicz \cite{bauerle2024riskSensitiveOverview} recently survey risk-sensitive MDPs based on optimized certainty equivalents, and solve them using state augmentation. Bäuerle and Ott~\cite{bauerle2011markov} augment the state with the cumulative reward; similar augmentations appear in \cite{chow2015risk,li2017quantile}. Our construction identifies the coarsest preference-memory augmentation required over the physical Markov state, among recursive factorizations.

\emph{Information states and state aggregation.}
The belief state of a partially observed MDP is the classical example of a derived state that restores the Markov property to a sequential optimization problem \cite{astrom1965optimal,smallwood1973optimal}. Subramanian et al.~\cite{subramanian2022approximate} develop a theory of information states as compressions of history that are sufficient for prediction and control. In the fully observed setting, state aggregation can be done by bisimulation \cite{givan2003equivalence}.
The PA state is effectively an information state for preference evaluation; the canonical PA state is the coarsest derived state that supports a recursive representation.

\subsection{Organization}

Section~\ref{sec:preliminaries} formalizes the objects of choice, DM's dynamic preferences, and the MDP model.
Section~\ref{sec:history} establishes the fully history-dependent recursive representation.
Section~\ref{sec:PA} derives the PA state and preference memory for a class of history-dependent preferences. Section~\ref{sec:DP} develops a Bellman recursion on the PA state.
In Section~\ref{sec:separation}, we separate beliefs and tastes and discuss the corresponding separated representation and Bellman recursion.
Section~\ref{sec:examples} proposes a taxonomy of PA and SPA preferences, and applies our framework to several examples.
The paper concludes in Section~\ref{sec:conclusion}, and all proofs are in the appendices.

\section{Preliminaries}
\label{sec:preliminaries}

For any integer $n \geq 1$, we let $[n] \triangleq \{0, 1, \ldots, n\}$ denote the running index starting from zero.
For integers $n_1 \leq n_2$, we let $[n_1 : n_2] \triangleq \{n_1, n_1 + 1, \ldots, n_2\}$ denote the running index from $n_1$ to $n_2$ (inclusive).

For a set ${\cal X}$, we let ${\rm id}_{\cal X} : {\cal X} \rightarrow {\cal X}$ denote the identity operator.
For a correspondence ${\cal A} : {\cal E} \rightrightarrows {\cal C}$ on ${\cal E} \subseteq {\cal X}$, we let ${\rm Sel}({\cal A})$ be the set of all selectors $g:{\cal E}\to{\cal C}$ with $g(e)\in{\cal A}(e)$ for all $e\in{\cal E}$.

\subsection{Plans}

The planning horizon is $[T] \triangleq \{0, 1, \ldots, T\}$, and terminal utility is evaluated at $T+1$.
The initial period $t=0$ state $\bar{s}_0$ is fixed; ${\cal S}$ is a finite set of exogenous states that can occur in periods $t \in [1,T] \triangleq \{1,2,\ldots,T\}$; and $\star$ is a dummy period $T+1$ state.
We let ${\cal S}_0 = \{\bar{s}_0\}$, ${\cal S}_t = {\cal S}$ for $t \in [1,T]$, and ${\cal S}_{T+1} = \{\star\}$ be the state spaces for all $t \in [T+1]$, and let $s_t \in {\cal S}_t$ be the state observed at the beginning of period $t \in [T+1]$.
For simplicity, we just write ${\cal S}$ where possible.
We suppose that $\{s_t\}_{t=0}^T$ is a Markov chain, and there is a reference transition kernel $\{q_t(\cdot \vert s_t)\}_{t=0}^{T-1}$ for the evolution of $\{s_t\}_{t=0}^T$.
Let ${\cal C} = [0, \bar{c}] \subset \mathbb{R}_{\geq 0}$ be a non-empty compact set of admissible nonnegative consumption levels, where $c_t \in {\cal C}$ is the consumption realized during period $t$.

Let ${\cal H}_0 = \{\bar{s}_0\}$ be the period $t=0$ history.
Let ${\cal H}_t = {\cal H}_{t-1} \times {\cal C} \times {\cal S}_t$ denote the set of all period $t \in [1, T]$ histories, with typical element $h_t = (\bar{s}_0, c_0, \ldots, s_{t-1}, c_{t-1}, s_t)$, which includes the current period $t$ state.
The set of terminal histories ${\cal H}_{T+1} = {\cal H}_T \times {\cal C}$ is the set of all histories $h_{T+1} = (\bar{s}_0, c_0, \ldots, s_T, c_T)$, and does not include the dummy terminal state $\star$.
The spaces of histories are equipped with the product metric inherited from $\Re$ (via ${\cal C}$) and the finite set ${\cal S}$.
Consequently, each ${\cal H}_t$ is compact and metrizable.
The one-step history-extension map $\iota_t:{\cal H}_t\times{\cal C}\times{\cal S}_{t+1}\to{\cal H}_{t+1}$ is defined by $\iota_t(h_t,c,s)=(h_t,c,s)$ for $t \in [T-1]$, and by $\iota_T(h_T,c,\star)=(h_T,c)$ for $t=T$. All $\{\iota_t\}_{t=0}^T$ are continuous.

The choice objects are consumption plans. The same analysis applies to other payoff or resource-allocation streams.
Plans are defined recursively. There is no consumption in period $T+1$ so ${\cal F}_{T+1} \triangleq \{\emptyset\}$, then ${\cal F}_t \triangleq {\cal C} \times ({\cal F}_{t+1})^{\cal S}$ for $t \in [T]$.
A continuation plan $f \in {\cal F}_t$ is then a sequence of decision rules $f = \{f_j\}_{j=0}^{T-t}$, where each $f_j : {\cal S}^j \rightarrow {\cal C}$ is a mapping from realized sequences of future states to consumption levels.
Since ${\cal S}^0$ is a singleton, $f_0 \in {\cal C}$ is the immediate consumption; $f_1(s')$ is the consumption choice in the next period after observing $s'$; $f_2(s', s'')$ is the consumption choice two periods ahead after observing $(s', s'')$; etc.
All ${\cal F}_t$ are compact by compactness of ${\cal C}$ and finiteness of ${\cal S}$.
We let $c_{t:T} = (c_t, \ldots, c_T) \in {\cal F}_t$ denote a deterministic consumption plan which is independent of the states.

For $t \in [T-1]$, we can decompose a plan $f \in {\cal F}_t$ into $f = (c, f_+)$ consisting of the head $c$ (current period consumption) and the tail $f_+ : {\cal S} \rightarrow {\cal F}_{t+1}$ (next-period continuation plan).
The tail is explicitly $f_+ = \{f_+(s)\}_{s \in {\cal S}}$, where $f_+(s) \in {\cal F}_{t+1}$ is the continuation plan for each possible state $s \in {\cal S}$.
We suppress the time dependence when writing $f = (c, f_+)$.

The grand domain of choice for each period $t \in [T]$ is ${\cal D}_t \triangleq {\cal H}_t \times {\cal F}_t$, equipped with the product topology, with typical element $(h_t, f) \in {\cal D}_t$. The domain ${\cal D}_t$ consists of {\it life paths} which include the history up to period $t$, and then the continuation plan in ${\cal F}_t$ from period $t$ forward.

\subsection{Preferences}

Let $\sigma_t^s : {\cal H}_t \rightarrow {\cal S}_t$ be the physical-state projection where $\sigma_t^s(h_t) = s_t$ for $h_t = (\bar{s}_0, c_0, \ldots, c_{t-1}, s_t) \in {\cal H}_t$ and $t \in [T]$.
For $t = T+1$, set $\sigma_{T+1}^s(h_{T+1})=\star$ for every $h_{T+1}\in{\cal H}_{T+1}$.
We formalize the evolution of the information in the full-history model as follows.

\begin{defn}[Full-history state system]
\label{defn:full_history_state_system}
A full-history state system is a tuple
\[
\mathfrak{H} = (\{{\cal H}_t\}_{t=0}^{T+1}, \{\sigma_t^s\}_{t=0}^{T+1}, \{\iota_t\}_{t=0}^{T}),
\]
where ${\cal H}_t$ is the space of period-$t$ histories, $\sigma_t^s:{\cal H}_t\to{\cal S}_t$ is the physical-state projection, and $\iota_t:{\cal H}_t\times{\cal C}\times{\cal S}_{t+1}
\to{\cal H}_{t+1}$ is the history concatenation map.
\end{defn}
\noindent
Since ${\cal C}$ is compact and ${\cal S}$ is finite, each ${\cal H}_t$ is nonempty, compact, and metrizable; each $\sigma_t^s$ and $\iota_t$ is continuous; and we have the identity
\[
\sigma_{t+1}^s(\iota_t(h_t,c,s'))=s',\, \forall h_t \in {\cal H}_t, c \in {\cal C}, s' \in {\cal S}_{t+1}, t \in [T].
\]

We now define the grand preference over the grand domain of choice. All other preference relations in this work are derived from it.
We specify that the grand preference is defined with respect to $\mathfrak{H}$ (a specific state system of evolving information). This sets our convention for this paper: preferences are defined with respect to an appropriate state system.

\begin{defn}[Grand preference]
\label{defn:preference}
A \emph{grand preference} on $\mathfrak{H}$ is a tuple ${\cal P} = (\{\succeq_{(t)}\}_{t=0}^T, V_{T+1})$ where:

(i) For each $t \in [T]$, $\succeq_{(t)}$ is a binary relation on ${\cal D}_t$.

(ii) $V_{T+1} : {\cal H}_{T+1} \rightarrow \Re$ is continuous.
\end{defn}
\noindent
The preference relations $\{\succeq_{(t)}\}_{t=0}^T$ allow us to compare counterfactuals on different histories. For instance, we can compare the same continuation plan from a traumatic history versus an elating one.
The terminal value function $V_{T+1}$ is a primitive which anchors the backward recursion.
Our goal is to optimize an underlying MDP with respect to ${\cal P}$.

We obtain history-dependent conditional preferences from ${\cal P}$, which are used to evaluate a plan's utility on a specific history.

\begin{defn}[Conditional preferences]
\label{defn:conditional_preference}
For each $h_t \in {\cal H}_t$, the conditional preference $\succeq_{h_t}$ is defined on ${\cal F}_t$ by $f \succeq_{h_t} g$ if and only if $(h_t, f) \succeq_{(t)} (h_t, g)$, for all $f, g \in {\cal F}_t$.
The conditional preferences are the collection $\{\succeq_{h_t}\}$.
\end{defn}
\noindent
The conditional preferences $\{\succeq_{h_t}\}$ are local restrictions of $\{\succeq_{(t)}\}$.
If we had started with $\{\succeq_{h_t}\}$ as the primitive, we would not be able to easily compare plans across different histories without additional structure. Each $\succeq_{h_t}$ could potentially be arbitrary at different nodes $h_t$, and there would be no continuous way to stitch them together.

\begin{defn}[Compatible utility system]
\label{defn:utility_system}
A utility system is a family ${\bf U}=\{U_t\}_{t=0}^{T+1}$, where $U_t:{\cal H}_t\times{\cal F}_t\to\Re$ is continuous for each $t\in[T]$, and $U_{T+1}:{\cal H}_{T+1}\to\Re$ is continuous.
The utility system ${\bf U}$ is compatible with ${\cal P}$ if $U_{T+1}=V_{T+1}$; if $U_T(h_T,c)=U_{T+1}(\iota_T(h_T,c,\star))$ for all $h_T\in{\cal H}_T$ and $c\in{\cal C}$; and if $U_t(h_t,f)\geq U_t(h_t',g)$ if and only if $(h_t,f)\succeq_{(t)}(h_t',g)$ for all $h_t, h_t' \in {\cal H}_t$, $f, g \in {\cal F}_t$, and $t\in[T]$.
\end{defn}
\noindent
A compatible utility system represents ${\cal P}$.
It also induces a representation for $\{\succeq_{h_t}\}$ given by $U_{h_t}(\cdot) \equiv U_t(h_t, \cdot)$.

Wherever it appears below, ${\bf U}=\{U_t\}_{t=0}^{T+1}$ denotes a fixed compatible utility system for ${\cal P}$. All objects defined from its cardinal scale---including attainable utility sets, effective aggregators, quotient utilities, and recursive factorizations---are relative to this fixed choice.

\subsection{Markov Decision Processes}

DM's preferences are expressed by a compatible utility system, which gives the objective for the full-history MDP.
Feasibility is determined by history-dependent consumption constraints through the correspondence ${\cal A}_t : {\cal H}_t \rightrightarrows {\cal C}$, where ${\cal A}_t(h_t) \subset {\cal C}$ is the set of feasible immediate consumption levels on history $h_t$.
At each history $h_t$, the policy must select a feasible current consumption level from ${\cal A}_t(h_t)$.

\begin{defn}
\label{defn:policy}
A history-dependent policy $\varphi = (\varphi_t)_{t=0}^T$ is a collection of mappings $\varphi_t : {\cal H}_t \rightarrow {\cal C}$ such that $\varphi_t(h_t) \in {\cal A}_t(h_t)$ for all $h_t \in {\cal H}_t$. Let $\Phi = \prod_{t=0}^T {\rm Sel}({\cal A}_t)$ denote the set of all feasible history-dependent policies.
\end{defn}

Let $\varphi=(\varphi_\tau)_{\tau=t}^T$ be a feasible policy from period $t$ onward, where $\varphi_\tau(h_\tau)\in{\cal A}_\tau(h_\tau)$ is the feasible consumption level dictated by the policy for each $\tau \in [t, T]$.
Given a history $h_t$, let $f_t^\varphi[h_t]\in{\cal F}_t$ denote the consumption plan induced by following $\varphi$ starting from $h_t$ in period $t \in [T]$.

Given a compatible utility system ${\bf U}$ that represents DM's preferences, define $U_t^\varphi(h_t)\triangleq U_t(h_t,f_t^\varphi[h_t])$ to be the utility of $f_t^\varphi[h_t]$ starting from period $t$.
We then introduce the optimal value functions $J_t : {\cal H}_t \rightarrow \Re$ for all $t \in [T]$, which are defined by:
\begin{align*}
    J_T(h_T) \triangleq & \sup_{c \in {\cal A}_T(h_T)} U_{T+1}(\iota_T(h_T,c,\star)),\, \forall h_T \in {\cal H}_T,\\
    J_t(h_t) \triangleq & \sup_{\varphi \in \Phi} U_t^{\varphi}(h_t),\, \forall h_t \in {\cal H}_t,\, t \in [T-1].
\end{align*}
There is no continuation policy at the terminal period, so it just optimizes over current consumption.
The core difficulty with evaluating $\{J_t\}_{t=0}^T$ is that the state spaces are ${\cal H}_t$ for history-dependent preferences, and the dimension of ${\cal H}_t$ grows with time.
The central preference-reduction question of this paper is: what is the coarsest augmentation of the physical Markov state that supports recursive evaluation of preferences? When feasibility factors through this augmented state, it supports a Bellman recursion under regularity conditions.

\section{Full-History Recursive Benchmark}
\label{sec:history}

This section establishes a recursive representation for the full-history problem.
It gives conditions under which the utility system admits a recursive representation through history-dependent time and risk aggregators.
We first describe behavioral axioms for DM's preferences, then characterize the desired recursive structure, and conclude with our recursive representation theorem.
This representation result is the foundation for the state reduction developed subsequently.

\subsection{Axioms}

Our desired recursive representation is characterized by several behavioral axioms. The following two are standard; they establish the existence of utility functions which represent $\{\succeq_{(t)}\}$.

\begin{axiom}
\label{axiom:weak_order}
(Weak order) For each $t \in [T]$, $\succeq_{(t)}$ is complete and transitive on ${\cal D}_t$.
\end{axiom}

\begin{axiom}
\label{axiom:continuity}
(Continuity) For each $t \in [T]$ and $(h_t, f) \in {\cal H}_t \times {\cal F}_t$, the sets $\{(h_t', g) : (h_t', g) \succeq_{(t)} (h_t, f)\}$ and $\{(h_t', g) : (h_t', g) \preceq_{(t)} (h_t, f)\}$ are closed in ${\cal D}_t$.
\end{axiom}
\noindent
Axiom~\ref{axiom:continuity} is a joint continuity axiom over entire life paths in the grand domain ${\cal D}_t$. This requirement is stronger than the usual continuity axiom for history-independent preferences, which only requires continuity over continuation plans ${\cal F}_t$ for a fixed $h_t$.
A compatible utility system exists under Axiom~\ref{axiom:weak_order}, Axiom~\ref{axiom:continuity}, and Axiom~\ref{axiom:terminal_compatibility} (see Lemma~\ref{lem:existence}), but it is generally not unique.

The next axiom guarantees coherence of preferences over time (i.e., DM always agrees with his future self after the next-period state is realized).

\begin{axiom}
\label{axiom:dynamic_consistency}
(Dynamic consistency) For all $t \in [T-1]$, $h_t \in {\cal H}_t$, and $c \in {\cal C}$, if
\[
(\iota_t(h_t, c, s), f_+(s)) \succeq_{(t+1)} (\iota_t(h_t, c, s), g_+(s)),\, \forall s \in {\cal S},
\]
then $(h_t, (c, f_+)) \succeq_{(t)} (h_t, (c, g_+))$ for all $f = (c, f_+), g = (c, g_+) \in {\cal F}_t$.
\end{axiom}
\noindent
Applying Axiom~\ref{axiom:dynamic_consistency} in both directions gives the indifference version of dynamic consistency.
For any $t \in [T-1]$, $h_t \in {\cal H}_t$, and $f, g \in {\cal F}_t$ with $f = (c, f_+)$ and $g = (c, g_+)$, if $(\iota_t(h_t, c, s), f_+(s)) \sim_{(t+1)} (\iota_t(h_t, c, s), g_+(s))$ for all $s \in {\cal S}$, then $(h_t, f) \sim_{(t)} (h_t, g)$.

Axiom~\ref{axiom:dynamic_consistency} is only defined up to period $T-1$, since there is no continuation plan in period $T$. Period $T$ is handled by the next axiom which connects preferences $\succeq_{(T)}$ to the exogenous terminal value function $V_{T+1}$.

\begin{axiom}
\label{axiom:terminal_compatibility}
(Terminal compatibility) For all $h_T, h_T' \in {\cal H}_T$ and $c, c' \in {\cal C}$, $(h_T, c) \succeq_{(T)} (h_T', c')$ if and only if $V_{T+1}((h_T, c)) \geq V_{T+1}((h_T', c'))$.
\end{axiom}

The next axiom concerns monotonicity in current consumption. It isolates the direct value of higher current consumption, for fixed continuation utilities.
For example, in a habit model, raising current consumption may change the next-period habit stock, so the same future consumption rule need not deliver the same continuation utility.

\begin{axiom}
\label{axiom:monotonicity}
(Compensated consumption monotonicity) 
For $t\in[T-1]$, $h_t\in{\cal H}_t$, and $c>c'$, let  $f_+,g_+\in({\cal F}_{t+1})^{\cal S}$. 
If $(\iota_t(h_t,c,s),f_+(s)) \sim_{(t+1)} (\iota_t(h_t,c',s),g_+(s))$ for all $s\in{\cal S}$, then $(h_t,(c,f_+)) \succ_{(t)} (h_t,(c',g_+))$.
\end{axiom}

The following axiom separates current consumption from the ranking of risky continuation-utility vectors. It ensures that the induced risk comparison does not depend on the current consumption level at which the continuation vectors are realized.

\begin{axiom}
\label{axiom:weak_separability}
(Weak separability) For $t\in[T-1]$, $h_t\in{\cal H}_t$, and $c,c'\in{\cal C}$, let $f_+,g_+,f_+',g_+'\in({\cal F}_{t+1})^{\cal S}$. If $(\iota_t(h_t,c,s),f_+(s))\sim_{(t+1)}(\iota_t(h_t,c',s),f_+'(s))$ and $(\iota_t(h_t,c,s),g_+(s))\sim_{(t+1)}(\iota_t(h_t,c',s),g_+'(s))$ for all $s\in{\cal S}$, then $(h_t,(c,f_+))\succeq_{(t)}(h_t,(c,g_+))$ if and only if $(h_t,(c',f_+'))\succeq_{(t)}(h_t,(c',g_+'))$.
\end{axiom}

\subsection{Recursive Structure}

Fix a compatible utility system. We ask whether it admits recursive structure.
We start by defining the components of the full-history recursive structure, starting with the time aggregator.

\begin{defn}[Time aggregator]
\label{defn:time_aggregator}
A \emph{time aggregator} for period $t\in[T-1]$ is a function $W_t:{\cal H}_t\times{\cal C}\times\Re\to\Re$ such that $v\to W_t(h_t,c,v)$ is nondecreasing for every $(h_t,c)\in{\cal H}_t\times{\cal C}$.
\end{defn}
\noindent
The time aggregator evaluates the period $t$ utility of current consumption $c$, followed by next-period continuation utility $v$.

Next we introduce the risk aggregator.
Let ${\cal L}$ be the space of bounded random variables $\tilde{u} : {\cal S} \rightarrow \Re$ on ${\cal S}$.
This space models the vector of continuation utilities indexed by the next-period state.
Since ${\cal S}$ is finite, we have ${\cal L} \cong \Re^{\cal S}$.
We equip ${\cal L}$ with the supremum norm $\|\tilde u\|_\infty\triangleq\max_{s\in{\cal S}}|\tilde u(s)|$, and write $\tilde{u} \geq \tilde{v}$ when $\tilde{u}(s) \geq \tilde{v}(s)$ for all $s \in {\cal S}$ (i.e., pointwise dominance).
For $v\in\Re$, let ${\bf v}\in{\cal L}$ denote the constant random variable satisfying ${\bf v}(s)=v$ for all $s\in{\cal S}$.
Risk aggregators act on ${\cal L}$ and quantify DM's uncertainty about the continuation utility.

\begin{defn}[Risk aggregator]
\label{defn:risk_aggregator}
A {\it risk aggregator} for period $t \in [T-1]$ is a function $M_t : {\cal H}_t \times {\cal L} \rightarrow \Re$ such that: (i) if $\tilde{u} \leq \tilde{v}$, then $M_t(h_t, \tilde{u}) \leq M_t(h_t, \tilde{v})$ for all $h_t \in {\cal H}_t$; and (ii) $M_t(h_t, {\bf v}) = v$ for all $h_t \in {\cal H}_t$ and $v \in \Re$.
\end{defn}
\noindent
The risk aggregators $\{M_t\}_{t=0}^{T-1}$ can also embed DM's ambiguity attitudes, e.g., when there is ambiguity about the distribution of the next-period state.

We now combine these components into a full-history recursive structure.

\begin{defn}[Full-history recursive structure]
\label{defn:recursive_utility}
A {\it full-history recursive structure} on $\mathfrak{H}$ is a tuple ${\cal V} = (\{W_t\}_{t=0}^{T-1}, \{M_t\}_{t=0}^{T-1}, V_{T+1}^{\cal V})$ where
\begin{enumerate}
    \item $W_t : {\cal H}_t \times {\cal C} \times \Re \to \Re$ is a jointly continuous time aggregator for all $t\in[T-1]$.
    \item $M_t : {\cal H}_t \times {\cal L} \to \Re$ is a jointly continuous risk aggregator for all $t\in[T-1]$.
    \item $V_{T+1}^{\cal V} : {\cal H}_{T+1}\to\Re$ is continuous.
\end{enumerate}
The pair $(\mathfrak{H}, {\cal V})$ generates a utility system ${\bf U}^{\cal V} = \{U_t^{\cal V}\}_{t=0}^{T+1}$ according to $U_{T+1}^{\cal V} = V_{T+1}^{\cal V}$ and:
\begin{subequations}
\label{eq:recursive_utility}
\begin{align}
U_T^{\cal V}(h_T, c) = & U_{T+1}^{\cal V}(\iota_T(h_T,c,\star)),\, \forall h_T \in {\cal H}_T, c \in {\cal C},\\
U_t^{\cal V}(h_t, f) = & W_t(h_t, c, M_t(h_t, \tilde{u}_{t+1}^{\cal V}(h_t, f))),\, \forall h_t \in {\cal H}_t,\, f = (c, f_+) \in {\cal F}_t, t \in [T-1],
\end{align}
\end{subequations}
where $\tilde{u}_{t+1}^{\cal V}(h_t, f) \in {\cal L}$ is defined by $[\tilde{u}_{t+1}^{\cal V}(h_t, f)](s) \triangleq U_{t+1}^{\cal V}(\iota_t(h_t, c, s), f_+(s))$ for all $s \in {\cal S}$.
\end{defn}

We connect $\mathfrak{H}$ and a full-history recursive structure ${\cal V}$ back to the original grand preference ${\cal P}$, as defined below.

\begin{defn}[Full-history representation]
\label{defn:recursive_utility_representation}
The pair $(\mathfrak{H}, {\cal V})$ \emph{represents} ${\cal P}$ if its generated utility system ${\bf U}^{\cal V}$ is compatible with ${\cal P}$.
\end{defn}

\subsection{Main Result}

The next theorem establishes sufficient conditions for our desired history-dependent recursive representation.
We first describe the effective range of continuation utilities where risk preferences are behaviorally defined.

\begin{defn}[Attainable one-step utility vectors]
\label{defn:attainable_vectors}
For the fixed compatible utility system ${\bf U}$, define:

(i) For all $t\in[T-1]$, $h_t \in {\cal H}_t$, and $f=(c,f_+)\in{\cal F}_t$, the \emph{one-step continuation-utility vector} $\tilde{u}_{t+1}(h_t,f)\in{\cal L}$ is defined by $[\tilde{u}_{t+1}(h_t,f)](s) \triangleq U_{t+1}(\iota_t(h_t,c,s),f_+(s))$ for all $s \in {\cal S}$.

(ii) For all $t \in [T-1]$, $h_t \in {\cal H}_t$, and $c \in {\cal C}$, let $\Lambda_t^{\bf U}(h_t, c) \triangleq \{\tilde{u}_{t+1}(h_t,(c,f_+)) : f_+\in({\cal F}_{t+1})^{\cal S}\}$, $\Lambda_t^{\bf U}(h_t) \triangleq \bigcup_{c\in{\cal C}}\Lambda_t^{\bf U}(h_t, c)$, and $\Lambda_t^{\bf U} \triangleq \bigcup_{h_t \in {\cal H}_t} \Lambda_t^{\bf U}(h_t)$.
\end{defn}
\noindent
The sets $\Lambda_t^{\bf U}(h_t, c)$ of attainable continuation utilities defined above depend on the cardinal scale set by ${\bf U}$, so they are not preference primitives.
The effective domain of DM's risk preferences is then $D_t^M \triangleq \left\{ (h_t,\tilde{u}): \tilde{u}\in\Lambda_t^{\bf U}(h_t) \right\}$.
We assume there is a continuous, monotone, internal, certainty-equivalent completion on this effective domain.

\begin{assumption}[Certainty-equivalent richness]
\label{assu:CE_richness}
For the fixed compatible utility system ${\bf U}$, suppose that for every $t\in[T-1]$ there is a jointly continuous functional $M_t^0 : D_t^M \rightarrow \Re$, such that for each $h_t\in{\cal H}_t$:

(i) \textnormal{(Risk-ranking representation and constant normalization)} $M_t^0(h_t,\cdot)$ represents the conditional risk ranking---for plans $f=(c,f_+),\,g=(c,g_+)\in{\cal F}_t$ with the same current consumption $c$, $f \succeq_{h_t} g$ if and only if
\[
M_t^0(h_t, \tilde{u}_{t+1}(h_t,f)) \ge M_t^0(h_t, \tilde{u}_{t+1}(h_t,g))
\]
and $M_t^0(h_t,{\bf v})=v$ whenever ${\bf v} \in \Lambda_t^{\bf U}(h_t)$.

(ii) \textnormal{(Internality)} $\min_{s \in {\cal S}}\tilde{u}(s)\le M_t^0(h_t,\tilde{u})\le\max_{s \in {\cal S}}\tilde{u}(s)$ for every $\tilde{u} \in \Lambda_t^{\bf U}(h_t)$.

(iii) \textnormal{(Deterministic solvability)} For every $f=(c,f_+)\in{\cal F}_t$, let $v=M_t^0(h_t,\tilde{u}_{t+1}(h_t,f))$, then there is some $(c,g_+)\in{\cal F}_t$ such that ${\bf v} = \tilde{u}_{t+1}(h_t, (c, g_+))$.

(iv) \textnormal{(Monotone completion)}
If $\tilde{u},\tilde{v}\in\Lambda_t^{\bf U}(h_t)$ and $\tilde{u}\leq\tilde{v}$, then $M_t^0(h_t,\tilde{u})\leq M_t^0(h_t,\tilde{v})$.
\end{assumption}

Our full-history recursive representation result for ${\cal P}$ is next.

\begin{thm}[Full-history recursive representation]
\label{thm:recursive_utility}
Let ${\cal P}$ be a grand preference on $\mathfrak H$, and let ${\bf U}$ be a compatible utility system for ${\cal P}$.
Suppose that ${\cal P}$ satisfies
Axioms~\ref{axiom:weak_order}--\ref{axiom:weak_separability}
and that the fixed utility system ${\bf U}$ satisfies
Assumption~\ref{assu:CE_richness}. Then there exists a full-history
recursive structure ${\cal V}$ on $\mathfrak H$ whose generated utility
system equals ${\bf U}$. In particular,
$(\mathfrak H,{\cal V})$ represents ${\cal P}$.
\end{thm}
\noindent
The proof of Theorem~\ref{thm:recursive_utility} is based on the following chain of reasoning.
Dynamic consistency first reduces continuation plans to vectors of next-period continuation utilities. Weak separability, compensated monotonicity, and Assumption~\ref{assu:CE_richness} then give behaviorally identified time and risk aggregators on their effective domains. 
We then extend the effective aggregators to global continuous monotone aggregators, which constitute the recursive structure ${\cal V}$.
Finally, backward induction verifies that the resulting recursion reproduces the original utility system.
Proposition~\ref{prop:extension_invariance} shows that the choice of global extensions does not affect the evaluation of feasible plans.

The content of this section can be read in two directions. In the constructive direction, a full-history recursive structure ${\cal V}$ on $\mathfrak{H}$ satisfying the preceding definitions directly generates a recursive utility system and grand preference. Theorem~\ref{thm:recursive_utility} is the complementary direction: starting from a primitive grand preference ${\cal P}$ on $\mathfrak{H}$, it gives conditions under which some recursive structure ${\cal V}$ represents ${\cal P}$.

\section{Preference-Augmented States}
\label{sec:PA}

Theorem~\ref{thm:recursive_utility} provides a recursive evaluation of history-dependent preferences, but it does not by itself reduce the state dimension. We now investigate how the history dependence can be summarized by augmenting the physical Markov state with the preference memory, to obtain a potentially smaller PA state.
To clarify, we only consider reduction of the preference memory; the physical Markov state is not subject to reduction.
In tractable cases, the PA state is low-dimensional or otherwise structured.
We first define the desired representation on a PA state.
Then, we verify that this representation exists and how it is determined endogenously by ${\cal P}$.

\subsection{Representation}

We begin with the definition of a generic PA-state system. It consists of a set of state spaces and transition maps that depend on the current PA state, current consumption, and next period physical state.
The state system specifies the state spaces, physical-state projection, and transition maps.

\begin{defn}[PA-state system]
\label{defn:memory}
A preference-augmented state system is a tuple
\[
\mathfrak{X} = (\{{\cal X}_t\}_{t=0}^{T+1}, \{\varsigma_t\}_{t=0}^{T+1}, \{\Gamma_t\}_{t=0}^{T})
\]
such that:

(i) For all $t \in [T+1]$, ${\cal X}_t$ is a non-empty compact metrizable PA-state space.

(ii) For all $t \in [T+1]$, $\varsigma_t:{\cal X}_t\to{\cal S}_t$ is continuous.

(iii) For all $t\in[T]$, $\Gamma_t:{\cal X}_t\times{\cal C}\times{\cal S}_{t+1}\to{\cal X}_{t+1}$ is a continuous transition map satisfying
\[
\varsigma_{t+1}(\Gamma_t(x_t,c,s'))=s',\, \forall x_t \in {\cal X}_t, c \in {\cal C}, s' \in {\cal S}_{t+1}.
\]
\end{defn}
\noindent
The original full-history system is in fact an instance of Definition~\ref{defn:memory} where ${\cal X}_t = {\cal H}_t$ and $\varsigma_t = \sigma_t^s$ for all $t \in [T+1]$, and $\Gamma_t = \iota_t$ for all $t \in [T]$.

We now modify the definitions of time and risk aggregators for a given PA-state system $\mathfrak{X}$.

\begin{defn}[PA time aggregator]
\label{defn:memory_time_aggregator}
A \emph{PA time aggregator} for period $t\in[T-1]$ is a function $W_t^*:{\cal X}_t\times{\cal C}\times\Re\to\Re$ such that, for every $(x_t,c)\in{\cal X}_t\times{\cal C}$, the mapping $v \to W_t^*(x_t,c,v)$ is nondecreasing.
\end{defn}

\begin{defn}[PA risk aggregator]
\label{defn:memory_risk_aggregator}
A {\it PA risk aggregator} for period $t \in [T-1]$ is a function $M_t^* : {\cal X}_t \times {\cal L} \rightarrow \Re$ such that: (i) if $\tilde{u} \leq \tilde{v}$, then $M_t^*(x_t, \tilde{u}) \leq M_t^*(x_t, \tilde{v})$ for all $x_t \in {\cal X}_t$; and (ii) $M_t^*(x_t, {\bf v}) = v$ for all $v \in \Re$ and $x_t \in {\cal X}_t$.
\end{defn}

A recursive structure is defined with respect to a fixed PA-state system. This convention separates the state and its evolution from the preference recursion.

\begin{defn}[PA recursive structure]
\label{defn:PA_recursive}
A {\it PA recursive structure} on a PA-state system $\mathfrak{X}$ is a tuple
\[
{\cal V}^* = (\{W_t^*\}_{t=0}^{T-1}, \{M_t^*\}_{t=0}^{T-1}, V_{T+1}^*)
\]
where:
\begin{enumerate}
    \item $W_t^* : {\cal X}_t \times {\cal C} \times \Re \rightarrow \Re$ is a jointly continuous PA time aggregator for all $t \in [T-1]$.
    \item $M_t^* : {\cal X}_t \times {\cal L} \rightarrow \Re$ is a jointly continuous PA risk aggregator for all $t \in [T-1]$.
    \item $V_{T+1}^* : {\cal X}_{T+1} \rightarrow \Re$ is continuous.
\end{enumerate}
The pair $(\mathfrak{X}, {\cal V}^*)$ generates a utility system $\{U_t^*\}_{t=0}^{T+1}$ where $U_{T+1}^* = V_{T+1}^*$ and $U_t^* : {\cal X}_t \times {\cal F}_t \rightarrow \Re$ are defined by:
\begin{subequations}
\label{eq:memory_recursive}
\begin{align}
U_T^*(x_T, c) = & U_{T+1}^*(\Gamma_T(x_T, c, \star)),\, \forall x_T \in {\cal X}_T, c \in {\cal C},\\
U_t^*(x_t, f) = & W_t^*(x_t, c, M_t^*(x_t, \tilde{u}_{t+1}^*(x_t, f))),\, \forall x_t \in {\cal X}_t, f = (c, f_+) \in {\cal F}_t, t \in [T-1],
\end{align}
\end{subequations}
where the next period random continuation utility on the PA state $\tilde{u}_{t+1}^*(x_t, f) \in {\cal L}$ for $t \in [T-1]$ is defined by $[\tilde{u}_{t+1}^*(x_t, f)](s) \triangleq U_{t+1}^*(\Gamma_t(x_t, c, s), f_+(s))$ for all $s \in {\cal S}$.
\end{defn}
\noindent
In many examples in Section~\ref{sec:examples}, history enters the recursion only through a low-dimensional memory coordinate.

The next condition describes when the original preferences can be represented on a PA state.

\begin{defn}[Recursive factorization]
\label{defn:memory_representation}
Let $\mathfrak{X}$ be a PA-state system, let ${\cal V}^*$ be a recursive structure on $\mathfrak{X}$, and let $\{U_t^*\}_{t=0}^{T+1}$ be the utility system generated by $(\mathfrak{X}, {\cal V}^*)$.
Then $(\mathfrak{X}, {\cal V}^*)$ \emph{factorizes} a compatible utility system ${\bf U}$ when there exist continuous summary maps $\sigma_t : {\cal H}_t \rightarrow {\cal X}_t$ satisfying $\varsigma_t\circ\sigma_t=\sigma_t^s$ for all $t \in [T+1]$ such that:
 
(i) (Utility factorization) The generated utilities $\{U_t^*\}_{t=0}^{T+1}$ satisfy:
\begin{subequations}
\label{eq:factorization}
\begin{align}
U_{T+1}(h_{T+1}) = & U_{T+1}^*(\sigma_{T+1}(h_{T+1})),\, \forall h_{T+1} \in {\cal H}_{T+1},\\
U_t(h_t, f) = & U_t^*(\sigma_t(h_t), f),\, \forall h_t \in {\cal H}_t,\, f \in {\cal F}_t,\, t \in [T].
\end{align}
\end{subequations}

(ii) (Transition consistency) The summary maps are consistent with the transition maps:
\begin{equation}
\label{eq:transition_consistency}
\sigma_{t+1}(\iota_t(h_t,c,s))= \Gamma_t(\sigma_t(h_t),c,s),\, \forall h_t\in{\cal H}_t,\, c\in{\cal C},\,  s\in{\cal S}_{t+1},\, t\in[T].
\end{equation}
\end{defn}
\noindent
If $(\mathfrak X,{\cal V}^*)$ factorizes a compatible utility system ${\bf U}$, it represents ${\cal P}$.
The clause in Definition~\ref{defn:memory_representation} that $\varsigma_t\circ\sigma_t=\sigma_t^s$ is essential for our definition of factorization. A utility-only factorization can pool histories with different current Markov states.

\subsection{Construction from Preferences}

The previous subsection described an idealized recursive representation on a PA-state system. Now, we derive this representation endogenously from ${\cal P}$ through an equivalence relation on histories.
Two histories are canonically equivalent when they have the same current physical state, give the same evaluation to every common continuation plan, and remain equivalent after every common one-step extension.
Throughout this subsection, we continue to use the fixed compatible utility system ${\bf U}$.

\begin{defn}[Canonical PA-state equivalence]
\label{defn:equivalence}
For terminal histories, define $\odot_{T+1} \subseteq {\cal H}_{T+1} \times {\cal H}_{T+1}$ by $h_{T+1} \odot_{T+1} h_{T+1}'$ if and only if $U_{T+1}(h_{T+1})=U_{T+1}(h_{T+1}')$.
Define $\odot_t\subseteq{\cal H}_t\times{\cal H}_t$ by $h_t \odot_t h_t'$ if and only if: (i) $\sigma_t^s(h_t)=\sigma_t^s(h_t')$; (ii) $(h_t,f)\sim_{(t)}(h_t',f)$ for all $f\in{\cal F}_t$; and (iii) $(\iota_t(h_t,c,s),\iota_t(h_t',c,s))\in \odot_{t+1}$ for all $c\in{\cal C}$ and $s\in{\cal S}_{t+1}$.
\end{defn}
\noindent
Condition (i) preserves the physical state among equivalence classes. Condition (ii) identifies histories that give the same evaluation to every common continuation plan.
Condition (iii) imposes forward stability, ensuring that common one-step extensions remain equivalent. Together, these conditions make the quotient a valid recursive state rather than merely a collection of preference classes.

We next verify that $\{\odot_t\}_{t=0}^{T+1}$ are indeed equivalence relations on the space of histories.

\begin{lem}
\label{lem:odot_equivalence}
For all $t\in[T+1]$, $\odot_t$ is an equivalence relation on ${\cal H}_t$.
\end{lem}
\noindent
The canonical PA-state system is defined by taking the quotient of the space of histories with respect to $\{\odot_t\}_{t=0}^{T+1}$.
Let the finite set ${\cal S}$ be endowed with the discrete topology, and let $\tau_{{\cal H}_t}$ denote the topology on ${\cal H}_t$ inherited from ${\cal S}$ and ${\cal C}$.

We first define the necessary data, and then verify that it actually forms a PA-state system.
The canonical PA-state system is determined by the preference-equivalence classes of ${\cal P}$. The quotient utilities, effective aggregators, and recursive factorizations considered below are relative to the fixed compatible utility system ${\bf U}$.

\begin{defn}[Canonical PA data]
\label{defn:memory_canonical}
The canonical PA quotient data consist of:
\begin{enumerate}
    \item For all $t\in[T+1]$, let ${\cal X}_t\triangleq{\cal H}_t/\odot_t$ be the canonical PA-state space with the canonical quotient projection $\sigma_t:{\cal H}_t\to{\cal X}_t$.
    \item For all $t\in[T+1]$, let $\varsigma_t:{\cal X}_t\to{\cal S}_t$ satisfy $\varsigma_t\circ\sigma_t=\sigma_t^s$.
    \item For all $t\in[T]$, let $\Gamma_t : {\cal X}_t \times {\cal C} \times {\cal S}_{t+1} \rightarrow {\cal X}_{t+1}$ satisfy
    \[
    \Gamma_t(\sigma_t(h_t),c,s') \triangleq \sigma_{t+1}(\iota_t(h_t,c,s')),\, \forall h_t\in{\cal H}_t,\ c\in{\cal C},\ s'\in{\cal S}_{t+1}.
    \]
\end{enumerate}
\end{defn}
\noindent
The map $\varsigma_t$ introduced above exists by Theorem~\ref{thm:quotient_universal_property} because $\sigma_t^s$ is continuous and constant on $\odot_t$-classes.

The following lemma establishes that the canonical PA-state space has the necessary topological structure.

\begin{lem}[Topology of the canonical quotient]
\label{lem:quotient_topology}
(i) For all $t\in[T+1]$, the relation $\odot_t$ is closed in ${\cal H}_t\times{\cal H}_t$, and ${\cal X}_t={\cal H}_t/\odot_t$ is compact and metrizable in the quotient topology.

(ii) For all $t\in[T+1]$, the canonical projection $\sigma_t : {\cal H}_t\to{\cal X}_t$ is continuous, closed, and a quotient map. Moreover, the readout map $\varsigma_t$ is well-defined and continuous.
\end{lem}

The canonical transition maps are well-defined and continuous in the quotient topology, established in the next lemma.

\begin{lem}[Canonical PA transition]
\label{lem:quotient_transition_continuity}
For each $t\in[T]$, the map $\Gamma_t:{\cal X}_t\times{\cal C}\times{\cal S}_{t+1}
\to{\cal X}_{t+1}$ defined by
\[
\Gamma_t(\sigma_t(h_t),c,s') = \sigma_{t+1}(\iota_t(h_t,c,s')),\, \forall h_t \in {\cal H}_t, c \in {\cal C}, s' \in {\cal S}_{t+1},
\]
is well-defined and continuous. Moreover, $\varsigma_{t+1}(\Gamma_t(x_t,c,s'))=s'$ for every $x_t\in{\cal X}_t$, $c\in{\cal C}$, and $s'\in{\cal S}_{t+1}$.
\end{lem}
\noindent
Consequently, the canonical PA-state system is $\mathfrak{X}^* = (\{{\cal X}_t\}_{t=0}^{T+1}, \{\varsigma_t\}_{t=0}^{T+1}, \{\Gamma_t\}_{t=0}^{T})$.

We can now factorize a compatible utility system ${\bf U}$ based on the endogenous construction of Definition~\ref{defn:memory_canonical}. We propose a collection $\{U_t^\circ\}_{t=0}^{T+1}$ of quotient utility functions $U_t^\circ : {\cal X}_t \times {\cal F}_t \rightarrow \Re$ for $t \in [T]$ and $U_{T+1}^\circ : {\cal X}_{T+1} \rightarrow \Re$, defined as follows:
\begin{subequations}
\label{eq:canonical_utility}
\begin{align}
    U_{T+1}^\circ(\sigma_{T+1}(h_{T+1})) \triangleq & U_{T+1}(h_{T+1}),\, \forall h_{T+1} \in {\cal H}_{T+1},\\
    U_t^\circ(\sigma_t(h_t),f) \triangleq & U_t(h_t,f),\, \forall h_t\in{\cal H}_t, f\in{\cal F}_t, t \in [T].
\end{align}
\end{subequations}

\begin{lem}[Quotient utilities]
\label{lem:quotient_utilities}
The utility functions $\{U_t^\circ\}_{t=0}^{T+1}$ from Eq.~\eqref{eq:canonical_utility} are well-defined and continuous.
\end{lem}
\noindent
The above quotient utilities are the bridge between the canonical PA-state system and the recursive preferences defined on it.

\subsection{Main Result}

The canonical quotient has the required topology and transition structure.
The quotient utilities induce a well-defined grand preference on the canonical
PA-state system. Lemma~\ref{lem:axiom_inheritance} shows that this preference inherits the behavioral axioms from ${\cal P}$, so the abstract recursive representation theorem applies on the quotient.

\begin{thm}[Canonical PA recursive representation]
\label{thm:recursive_memory}
Let ${\bf U}$ be a compatible utility system for ${\cal P}$.
Suppose that ${\cal P}$ satisfies
Axioms~\ref{axiom:weak_order}--\ref{axiom:weak_separability}
and that ${\bf U}$ satisfies
Assumption~\ref{assu:CE_richness}. Then the canonical PA-state system $\mathfrak X^*$ admits a PA recursive structure ${\cal V}^*$ such that $(\mathfrak X^*,{\cal V}^*)$ factorizes ${\bf U}$.
Consequently, $(\mathfrak{X}^*,{\cal V}^*)$ represents ${\cal P}$.
\end{thm}

Theorem~\ref{thm:recursive_memory} uses the PA framework in its reductive
direction: starting from primitive preferences over full histories, it derives
the canonical PA-state system and a recursive factorization of the
fixed compatible utility system ${\bf U}$.

\begin{rem}[Constructive use of PA states]
For the constructive direction, let
\[
\mathfrak X = (\{{\cal X}_t\}_{t=0}^{T+1}, \{\varsigma_t\}_{t=0}^{T+1},\{\Gamma_t\}_{t=0}^{T})
\]
be a PA-state system, and let ${\cal V}^*$ be a recursive structure on $\mathfrak X$. Choose $x_0\in{\cal X}_0$ satisfying $\varsigma_0(x_0)=\bar s_0$, and define $\widehat\sigma_0(\bar s_0)=x_0$. Recursively, set $\widehat\sigma_{t+1}(\iota_t(h_t,c,s'))=\Gamma_t(\widehat\sigma_t(h_t),c,s')$. Continuity of $\{\widehat\sigma_t\}_{t=0}^{T+1}$ and the identities $\varsigma_t \circ \widehat\sigma_t = \sigma_t^s$ follow from continuity of $\{\Gamma_t\}_{t=0}^{T}$ and the readout identities $\varsigma_{t+1}(\Gamma_t(x_t,c,s'))=s'$.
The pair $(\mathfrak{X}, {\cal V}^*)$ then induces a full-history utility system $\widehat{\bf U} = \{\widehat U_t\}_{t=0}^{T+1}$ by
\begin{align*}
    \widehat U_{T+1}(h_{T+1}) \triangleq & U_{T+1}^*(\widehat\sigma_{T+1}(h_{T+1})),\, \forall h_{T+1} \in {\cal H}_{T+1},\\
    \widehat U_t(h_t,f) \triangleq & U_t^*(\widehat\sigma_t(h_t),f),\, \forall h_t \in {\cal H}_t,\, f \in {\cal F}_t,\, t\in[T].
\end{align*}
The induced summary maps $\{\widehat\sigma_t\}_{t=0}^{T+1}$ make $(\mathfrak X,{\cal V}^*)$ a factorization of $\widehat{\bf U}$.
\end{rem}

\subsection{Minimality}

The canonical PA-state system is minimal in the following sense. Any other reachable recursive factorization of the same compatible utility system ${\bf U}$ refines the canonical quotient.
We compare another PA-state system to the canonical one in the following definition.
Let $\mathfrak{X}^* = (\{{\cal X}_t\}_{t=0}^{T+1}, \{\varsigma_t\}, \{\Gamma_t\}_{t=0}^{T})$ be the canonical PA-state system, with canonical projections $\{\sigma_t\}_{t=0}^{T+1}$, and quotient utilities $\{U_t^\circ\}_{t=0}^{T+1}$ from Lemma~\ref{lem:quotient_utilities}.

\begin{defn}[Memory correspondence]
\label{defn:minimality}
Fix the compatible utility system ${\bf U}$. Let
\[
\mathfrak{X}' = (\{{\cal X}_t'\}_{t=0}^{T+1}, \{\varsigma_t'\}_{t=0}^{T+1}, \{\Gamma_t'\}_{t=0}^{T})
\]
be a PA-state system, let ${\cal V}' = (\{W_t'\}_{t=0}^{T-1}, \{M_t'\}_{t=0}^{T-1}, V_{T+1}')$ be a recursive structure on $\mathfrak{X}'$, and let ${\bf U}'=\{U_t'\}_{t=0}^{T+1}$ be the utility system generated by $(\mathfrak{X}',{\cal V}')$.
Suppose that $(\mathfrak{X}',{\cal V}')$ factorizes ${\bf U}$ through summary maps $\boldsymbol{\sigma}' = \{\sigma_t'\}_{t=0}^{T+1}$.
For each $t\in[T+1]$, let $\hat{\cal X}_t' \triangleq \sigma_t'({\cal H}_t)$ be endowed with the subspace topology inherited from ${\cal X}_t'$.
A \emph{memory correspondence} is a collection of continuous maps $\theta_t:\hat{\cal X}_t'\to{\cal X}_t$ for $t \in [T+1]$ such that:
\begin{enumerate}
\item $\theta_t\circ\sigma_t'=\sigma_t$ for all $t\in[T+1]$.
\item $\Gamma_t'(x_t',c,s)\in\hat{\cal X}_{t+1}'$ and $\theta_{t+1}(\Gamma_t'(x_t',c,s))=\Gamma_t(\theta_t(x_t'),c,s)$, for all $x_t'\in\hat{\cal X}_t'$, $c\in{\cal C}$, $s\in{\cal S}_{t+1}$, and $t\in[T]$.
\item $U_t'(x_t',f)=U_t^\circ(\theta_t(x_t'),f)$ for all $x_t'\in\hat{\cal X}_t'$, $f\in{\cal F}_t$, and $t\in[T]$. In addition, $U_{T+1}'(x_{T+1}')=U_{T+1}^\circ(\theta_{T+1}(x_{T+1}'))$ for all $x_{T+1}'\in\hat{\cal X}_{T+1}'$.
\end{enumerate}
\end{defn}
\noindent
The maps $\{\theta_t\}_{t=0}^{T+1}$ send reachable states of $\mathfrak X'$ onto the canonical PA states and preserve summaries, transitions, and utilities.

\begin{prop}[Minimality and uniqueness]
\label{prop:minimality}
Fix a compatible utility system ${\bf U}$ for ${\cal P}$.
Let $\mathfrak{X}'$ be a PA-state system and let ${\cal V}'$ be a
recursive structure on $\mathfrak{X}'$. Suppose that
$(\mathfrak{X}',{\cal V}')$ factorizes
${\bf U}$ through summary maps $\boldsymbol\sigma' = \{\sigma_t'\}_{t=0}^{T+1}$.
Then there exists a unique memory correspondence associated with this factorization. If, in addition, $h_t\odot_t h_t'$ implies $\sigma_t'(h_t)=\sigma_t'(h_t')$ for every $t\in[T+1]$, then each $\theta_t:\hat{\cal X}_t'\to{\cal X}_t$ is a homeomorphism.
\end{prop}
\noindent
Proposition~\ref{prop:minimality} shows that the canonical quotient is the coarsest reachable recursive factorization of the fixed compatible utility system ${\bf U}$.
Moreover, if a factorization separates exactly the canonical preference-equivalence classes, then its reachable state spaces are homeomorphic to the corresponding canonical PA-state spaces.

\section{Dynamic Programming}
\label{sec:DP}

We now proceed to develop a Bellman recursion for our original MDP with respect to the PA-state system.
With our representation result for fully history-dependent preferences from Theorem~\ref{thm:recursive_utility}, we can compute the utility $U_t^{\varphi}(h_t)$ of a policy $\varphi \in \Phi$ according to $U_{T+1}^{\varphi} = U_{T+1}$ and:
\begin{subequations}
\label{eq:recursive_utility_policy}
\begin{align}
U_T^{\varphi}(h_T) = & U_{T+1}^{\varphi}(\iota_T(h_T, \varphi_T(h_T),\star)),\, \forall h_T \in {\cal H}_T,\\
U_t^{\varphi}(h_t) = & W_t(h_t, \varphi_t(h_t), M_t(h_t, \tilde{u}_{t+1}^{\varphi}(h_t, \varphi_t(h_t)))),\, \forall h_t \in {\cal H}_t,\, t \in [T-1],
\end{align}
\end{subequations}
where $\tilde{u}_{t+1}^{\varphi}(h_t, c) \in {\cal L}$ is defined by $[\tilde{u}_{t+1}^{\varphi}(h_t, c)](s) \triangleq U_{t+1}^{\varphi}(\iota_t(h_t, c, s))$ for all $s \in {\cal S}$.

For this section, we fix a PA-state system $\mathfrak{X}$, a recursive structure ${\cal V}^*$ on $\mathfrak{X}$, and summary maps $\{\sigma_t\}_{t=0}^{T+1}$ such that $(\mathfrak{X},{\cal V}^*)$ factorizes a compatible utility system ${\bf U}$ for ${\cal P}$.
This $(\mathfrak{X},{\cal V}^*)$ is not necessarily the canonical PA-state system. We assume the summary maps for the recursive factorization are onto. Otherwise, we can replace each ${\cal X}_t$ by the compact reachable image $\sigma_t({\cal H}_t)$ and then restrict $\Gamma_t$.

We next suppose that the consumption feasibility constraints only depend on the PA state.

\begin{assumption}[Markov feasibility]
\label{assu:Markov_feasibility}
If $\sigma_t(h_t)=\sigma_t(h_t')$, then ${\cal A}_t(h_t)={\cal A}_t(h_t')$, and we write ${\cal A}_t(x_t) \triangleq {\cal A}_t(h_t)$ for all $h_t \in \sigma_t^{-1}(x_t)$.
\end{assumption}
\noindent
This assumption precludes additional history dependence beyond the PA state.
Without this assumption, two histories could be preference equivalent but have different feasible consumption sets.
If feasibility depends on additional history that is not preference-relevant, the Bellman state must be enlarged by the corresponding feasibility state.

We now define the class of Markov policies with respect to the PA state.

\begin{defn}[PA Markov policy]
\label{defn:Markov_policy}
A \emph{PA Markov policy} is a sequence $\pi = \{\pi_t\}_{t=0}^T$ where $\pi_t : {\cal X}_t \rightarrow {\cal C}$ such that $\pi_t(x_t) \in {\cal A}_t(x_t)$ for all $x_t \in {\cal X}_t$ and $t \in [T]$. We let $\Pi = \prod_{t=0}^T {\rm Sel}({\cal A}_t)$ denote the set of all feasible Markov policies.
\end{defn}
\noindent
Any assignment of a feasible consumption to every PA state and time defines an admissible Markov policy.

We let $U_t^{*\pi} : {\cal X}_t \rightarrow \Re$ be the period $t \in [T+1]$ policy utility function of a Markov policy $\pi \in \Pi$. By the recursive structure on $\mathfrak{X}$ and Definition~\ref{defn:PA_recursive}, $\{U_t^{*\pi}\}$ have the recursive representation where $U_{T+1}^{*\pi}\triangleq U_{T+1}^*$ and:
\begin{subequations}
\label{eq:memory_recursive_policy}
\begin{align}
U_T^{*\pi}(x_T) = & U_{T+1}^{*\pi}(\Gamma_T(x_T, \pi_T(x_T), \star)),\, \forall x_T \in {\cal X}_T,\\
U_t^{*\pi}(x_t) = & W_t^*(x_t, \pi_t(x_t), M_t^*(x_t, \tilde{u}_{t+1}^{*\pi}(x_t, \pi_t(x_t)))),\, \forall x_t \in {\cal X}_t,\, t \in [T-1],
\end{align}
\end{subequations}
where $\tilde{u}_{t+1}^{*\pi}(x_t, c) \in {\cal L}$ is defined by $[\tilde{u}_{t+1}^{*\pi}(x_t, c)](s) \triangleq U_{t+1}^{*\pi}(\Gamma_t(x_t, c, s))$ for all $s \in {\cal S}$.

We impose additional regularity conditions on the feasibility correspondence with respect to the PA state. These conditions are sufficient for continuity of the optimal value function and attainment of an optimal consumption level in every state.

\begin{assumption}[DP regularity]
\label{assu:DP_regular}
For all $t\in[T]$, the correspondence ${\cal A}_t:{\cal X}_t\rightrightarrows{\cal C}$ is continuous with nonempty compact values.
\end{assumption}

Let $J_t^* : {\cal X}_t \rightarrow \Re$ be the optimal period $t \in [T]$ PA value function, where $\{J_t^*\}_{t=0}^T$ are defined via:
\[
J_t^*(x_t) = \sup_{\pi \in \Pi} U_t^{*\pi}(x_t),\, \forall x_t \in {\cal X}_t, t \in [T].
\]
For $t \in [T-1]$, $x_t \in {\cal X}_t$, and $c \in {\cal C}$, the optimal PA continuation utility $\tilde{j}_{t+1}^*(x_t, c) \in {\cal L}$ is defined by $[\tilde{j}_{t+1}^*(x_t, c)](s) \triangleq J_{t+1}^*(\Gamma_t(x_t, c, s))$ for all $s \in {\cal S}$.

We now establish the principle of optimality and show that the optimal value functions $\{J_t^*\}_{t=0}^T$ satisfy a Bellman recursion.

\begin{thm}[Bellman recursion]
\label{thm:DP_Bellman}
Suppose $(\mathfrak{X},{\cal V}^*)$, with summary maps $\{\sigma_t\}_{t=0}^{T+1}$, factorizes a compatible utility system ${\bf U}$ for ${\cal P}$, and Assumptions~\ref{assu:Markov_feasibility} and~\ref{assu:DP_regular} hold. Then $\{J_t^*\}_{t=0}^T$ satisfy:
\begin{subequations}
\label{eq:Bellman}
\begin{align}
J_T^*(x_T) = & \max_{c \in {\cal A}_T(x_T)} U_{T+1}^*(\Gamma_T(x_T, c, \star)),\, \forall x_T \in {\cal X}_T,\\
J_t^*(x_t) = & \max_{c \in {\cal A}_t(x_t)} W_t^*(x_t, c, M_t^*(x_t, \tilde{j}_{t+1}^*(x_t, c))),\, \forall x_t \in {\cal X}_t,\, t \in [T-1].
\end{align}
\end{subequations}
The maxima are attained. There exists a single $\bar\pi\in\Pi$ which attains the displayed maxima at every state and time, and satisfies $U_t^{*\bar\pi}(x_t)=J_t^*(x_t)$ for every $x_t\in{\cal X}_t$ and $t\in[T]$.
Moreover, each $J_t^*$ is continuous.
\end{thm}
\noindent
Theorem~\ref{thm:DP_Bellman} gives a Bellman recursion for $\{J_t^*\}_{t=0}^T$.
We next verify that any Bellman selector from Eq.~\eqref{eq:Bellman} induces a full-history policy that is optimal among all feasible history-dependent policies.

Let $\{J_t^*\}_{t=0}^T$ be the optimal value functions from Theorem~\ref{thm:DP_Bellman}. The theorem guarantees existence of at least one Bellman selector $\bar\pi\in\Pi$ satisfying:
\begin{subequations}
\label{eq:Bellman_selector}
\begin{align}
    \bar{\pi}_T(x_T) \in & \arg\max_{c\in{\cal A}_T(x_T)} U_{T+1}^*(\Gamma_T(x_T, c, \star)),\, \forall x_T\in{\cal X}_T,\\
    \bar{\pi}_t(x_t) \in & \arg\max_{c \in {\cal A}_t(x_t)} W_t^*(x_t, c, M_t^*(x_t, \tilde{j}_{t+1}^*(x_t, c))),\, \forall x_t \in {\cal X}_t,\, t \in [T-1].
\end{align}
\end{subequations}
The corresponding induced full-history policy $\bar{\varphi}\in\Phi$ is then defined by $\bar{\varphi}_t(h_t)=\bar{\pi}_t(\sigma_t(h_t))$, for all $h_t \in {\cal H}_t$ and $t \in [T]$.

\begin{thm}[Verification]
\label{thm:DP_verification}
Suppose $(\mathfrak{X},{\cal V}^*)$, with summary maps $\{\sigma_t\}_{t=0}^{T+1}$, factorizes a compatible utility system ${\bf U}$ for ${\cal P}$, and Assumptions~\ref{assu:Markov_feasibility} and~\ref{assu:DP_regular} hold. Let $\bar\pi\in\Pi$ be any Bellman selector satisfying Eq.~\eqref{eq:Bellman_selector}, and let $\bar\varphi$ be its induced full-history policy. Then, for every $\varphi\in\Phi$,
\[
U_t^\varphi(h_t)\leq J_t^*(\sigma_t(h_t))=U_t^{\bar\varphi}(h_t),\, \forall h_t\in{\cal H}_t,\ t\in[T].
\]
\end{thm}
\noindent
By Theorem~\ref{thm:DP_verification}, the induced policy $\bar{\varphi}$ is optimal among all policies in $\Phi$.
In particular, the full-history optimal value functions satisfy the factorization $J_t(h_t)=J_t^*(\sigma_t(h_t))$ for all $h_t\in{\cal H}_t$ and $t\in[T]$.

\section{Belief and Taste Separation}
\label{sec:separation}

In general, the effects of beliefs and tastes are entangled in the recursive representation of Eq.~\eqref{eq:memory_recursive}. The time and risk aggregators, and transition maps, are all allowed to depend on $x_t$.
When both aggregators depend on $x_t$, observed behavior might not be enough to determine whether a state variable acts through beliefs, tastes, or both.
The related identification problem between state-dependent utility and subjective probability is well known; see Karni et al.~\cite{karni1983state}.
Separation of beliefs and tastes is desirable from an economic viewpoint, where risk and ambiguity preferences are assigned to the belief coordinate and time preferences to the taste coordinate.
We develop a separated preference-augmented (SPA) state system in this section.
This separation improves interpretation and may facilitate empirical identification.

\subsection{Rectangularity}

We begin by imposing additional structure on the canonical quotient PA-state system from Section~\ref{sec:PA}. Rectangularity requires each quotient space ${\cal X}_t={\cal H}_t/\odot_t$ to have product structure between the physical state space and an explicit preference-memory state space.

\begin{defn}[Rectangular canonical PA state]
\label{defn:memory_rectangular_canonical}
The canonical PA-state system is rectangular if there exist memory state spaces $\{{\cal M}_t\}_{t=0}^{T+1}$ and mappings $\{\kappa_t\}_{t=0}^{T+1}$ such that:
\begin{enumerate}
\item For all $t\in[T+1]$, ${\cal M}_t$ is a compact metrizable space and $\kappa_t : {\cal X}_t \rightarrow {\cal S}_t \times {\cal M}_t$ is a homeomorphism that satisfies $\operatorname{pr}_{{\cal S}_t}\circ\kappa_t=\varsigma_t$. Then, we can identify ${\cal X}_t$ with ${\cal S}_t \times {\cal M}_t$.
\item For $t\in[T]$, a rectangular PA state is a pair $x_t=(s_t,m_t)\in{\cal X}_t$, where $s_t$ is the physical Markov state and $m_t$ is the preference-memory state. The terminal rectangular PA state is $x_{T+1}=(\star,m_{T+1})\in{\cal X}_{T+1}$.
\end{enumerate}
\end{defn}
\noindent
The physical Markov state and preference memory play different roles. The physical state forms the fixed environmental base for prediction and may also enter feasibility, and $\varsigma_t$ is the coordinate projection onto the physical state.
The preference-memory component $m_t\in{\cal M}_t$ summarizes the additional history dependence of recursive evaluation through preferences only. This rectangularization makes explicit how $x_t$ is an augmentation of $s_t$.

Definition~\ref{defn:memory_rectangular_canonical} explicitly disentangles the effect of the physical state from the preference memory.
We refer to this construction as the rectangular PA-state system.
Under this homeomorphism, we identify ${\cal X}_t$ with
${\cal S}_t\times{\cal M}_t$, where ${\cal M}_t$ is independent of the physical state. We say a representation satisfying Definition~\ref{defn:memory_rectangular_canonical} is rectangular for this reason.
For all $t\in[T]$, let $\Gamma_t^m:{\cal X}_t\times{\cal C}\times{\cal S}_{t+1}\to{\cal M}_{t+1}$ be the memory update defined by
\[
\Gamma_t^m(x_t,c,s') \triangleq\operatorname{pr}_{{\cal M}_{t+1}}\left(\kappa_{t+1}(\Gamma_t(x_t,c,s'))\right),\, \forall x_t \in {\cal X}_t, c \in {\cal C},\, s' \in {\cal S}_{t+1}.
\]
The maps $\{\Gamma_t^m\}_{t=0}^{T}$ are derived from the original $\{\Gamma_t\}_{t=0}^{T}$ and update the preference-memory state only.
Under the coordinate identification ${\cal X}_t\cong{\cal S}_t\times{\cal M}_t$, the canonical PA transition can be written as
\[
\Gamma_t((s_t,m_t),c,s') = (s',\Gamma_t^m((s_t,m_t),c,s')).
\]

\begin{assumption}[Rectangularity]
\label{assu:rectangular}
The canonical PA-state system is rectangular.
\end{assumption}
\noindent
Assumption~\ref{assu:rectangular} concerns the shape of the canonical quotient. It was not needed for the representation, minimality, or DP results from Section~\ref{sec:PA} and Section~\ref{sec:DP}, which were all developed directly on the raw canonical quotient.
Under a chosen rectangularization of the canonical PA state, we can further decompose memory into belief and taste coordinates.
Definition~\ref{defn:aggregator_belief_taste_equivalence} below compares $M_t^*(s,m,\cdot)$ and $W_t^*(s,m,\cdot,\cdot)$ across all physical states for a fixed memory. This is a counterfactual comparison, which is well defined only when memory values are identified across states through the product structure ${\cal X}_t={\cal S}_t\times{\cal M}_t$.

Since ${\cal S}_t$ is finite, ${\cal X}_t$ is the disjoint union of the clopen fibers ${\cal X}_t(s)\triangleq\{\sigma_t(h_t):\sigma_t^s(h_t)=s\}$. The desired rectangularity holds if and only if these fibers are homeomorphic to a common compact metrizable set ${\cal M}_t$.
A sufficient condition is that the fibers $\{{\cal X}_t(s)\}_{s\in{\cal S}_t}$ admit a common parameterization by a compact metrizable space ${\cal M}_t$, with a parameter range independent of $s$. The directly specified PA and SPA state systems in
Section~\ref{sec:examples} are rectangular by construction. This does not, by itself, establish rectangularity of the canonical PA quotient of the induced preferences unless the displayed state system is also a canonical factorization.
The fiber identifications ${\cal X}_t(s)\cong{\cal M}_t$ for each $s \in {\cal S}_t$ constitute a choice of which memory values at different physical states are to be regarded as the same. The subsequent separation is relative to that choice.

\subsection{Definition}

We can modify the definition of a PA-state system to separately account for beliefs and tastes.

\begin{defn}[Separated PA-state system]
\label{defn:memory_separated}
A separated preference-augmented state system is a tuple
\[
\mathfrak{X}^\dagger = (\{{\cal Y}_t\}_{t=0}^{T+1}, \{{\cal Z}_t\}_{t=0}^{T+1}, \{{\cal X}_t^\dagger\}_{t=0}^{T+1}, \{\varsigma_t^\dagger\}_{t=0}^{T+1}, \{\Gamma_t^\dagger\}_{t=0}^T)
\]
of belief and taste state spaces $\{{\cal Y}_t\}_{t=0}^{T+1}$ and $\{{\cal Z}_t\}_{t=0}^{T+1}$, SPA-state spaces $\{{\cal X}_t^\dagger\}_{t=0}^{T+1}$, physical readouts $\{\varsigma_t^\dagger\}_{t=0}^{T+1}$, and SPA transition maps $\{\Gamma_t^\dagger\}_{t=0}^T$ such that:
\begin{enumerate}
\item For all $t \in [T+1]$, ${\cal Y}_t$ and ${\cal Z}_t$ are Hausdorff topological spaces, and ${\cal X}_t^\dagger \subseteq {\cal S}_t \times {\cal Y}_t \times {\cal Z}_t$ is a non-empty compact metrizable subspace.
\item For all $t \in [T]$, an SPA state is a triple $\xi_t = (s_t,y_t,z_t) \in {\cal X}_t^\dagger$. The terminal SPA state is $\xi_{T+1}=(\star,y_{T+1},z_{T+1})\in{\cal X}_{T+1}^\dagger$.
\item For $\xi_t=(s_t,y_t,z_t)\in{\cal X}_t^\dagger$, define the physical-state map $\varsigma_t^\dagger(\xi_t)=s_t$.
\item For all $t \in [T+1]$, define $r_t^y : {\cal X}_t^\dagger \rightarrow {\cal S}_t \times {\cal Y}_t$ by $r_t^y(\xi_t)\triangleq(s_t,y_t)$ and $r_t^z : {\cal X}_t^\dagger \rightarrow {\cal S}_t \times {\cal Z}_t$ by $r_t^z(\xi_t)\triangleq(s_t,z_t)$. Let ${\cal X}_t^{\dagger,y}\triangleq r_t^y({\cal X}_t^\dagger)\subseteq{\cal S}_t\times{\cal Y}_t$ and ${\cal X}_t^{\dagger,z}\triangleq r_t^z({\cal X}_t^\dagger)\subseteq{\cal S}_t\times{\cal Z}_t$.
\item For each $t\in[T]$, $\Gamma_t^\dagger:{\cal X}_t^\dagger\times{\cal C}\times{\cal S}_{t+1} \to {\cal X}_{t+1}^\dagger $ is a continuous transition map that satisfies $\varsigma_{t+1}^\dagger(\Gamma_t^\dagger(\xi_t,c,s'))=s'$ for all $\xi_t \in {\cal X}_t^\dagger$, $c \in {\cal C}$, and $s' \in {\cal S}_{t+1}$.
\end{enumerate}
\end{defn}
\noindent
The state space ${\cal X}_t^\dagger \subseteq {\cal S}_t \times {\cal Y}_t \times {\cal Z}_t$ specifies the admissible physical, belief, and taste triples.
The belief state might more appropriately be called the belief/risk-evaluation state, since it can also encode probabilistic beliefs, ambiguity sets, and risk measures; for brevity we keep belief state.
The physical readout $\varsigma_t^\dagger$ is automatically continuous as the restriction of the first-coordinate projection.

The SPA transition map is
\begin{equation}
\label{eq:transition_separated}
    \xi_{t+1}=\Gamma_t^\dagger(\xi_t,c_t,s_{t+1}),\, \forall t\in[T].
\end{equation}
Although $\Gamma_t^\dagger$ is the primitive transition, we write its belief and taste coordinate maps as
\[
\Gamma_t^\dagger(\xi_t,c,s') = (s', \Gamma_t^y(\xi_t,c,s'), \Gamma_t^z(\xi_t,c,s')),\, \forall \xi_t \in {\cal X}_t^\dagger, c \in {\cal C}, s' \in {\cal S}_{t+1},
\]
where $\Gamma_t^y:{\cal X}_t^\dagger\times{\cal C}\times{\cal S}_{t+1}\to{\cal Y}_{t+1}$ and $\Gamma_t^z:{\cal X}_t^\dagger\times{\cal C}\times{\cal S}_{t+1}\to{\cal Z}_{t+1}$ are continuous coordinate update maps for all $t \in [T]$.

We call the SPA transition \emph{decoupled} if, for each $t \in [T]$, there exist continuous maps $\bar\Gamma_t^y: {\cal X}_t^{\dagger,y}\times{\cal C}\times{\cal S}_{t+1} \to{\cal Y}_{t+1}$ and $\bar\Gamma_t^z:{\cal X}_t^{\dagger,z}\times{\cal C}\times{\cal S}_{t+1}\to{\cal Z}_{t+1}$ such that
\[
\Gamma_t^y(\xi_t,c,s')=\bar\Gamma_t^y(r_t^y(\xi_t),c,s'),\,\Gamma_t^z(\xi_t,c,s')=\bar\Gamma_t^z(r_t^z(\xi_t),c,s'),\, \forall \xi_t \in {\cal X}_t^\dagger, c \in {\cal C}, s' \in {\cal S}_{t+1}.
\]
In this case, the belief update does not depend directly on the taste coordinate and the taste update does not depend directly on the belief coordinate. Both belief and taste states may still depend on the physical state, the common consumption, and the common next state.
The admissible SPA state space and transition map jointly restrict the belief--taste combinations and trajectories that can arise.

We modify the original Definition~\ref{defn:PA_recursive} for a recursive structure to account for separated beliefs and tastes, starting with the SPA time and risk aggregators.

\begin{defn}[SPA time aggregator]
\label{defn:memory_time_aggregator_separated}
An \emph{SPA time aggregator} for period $t\in[T-1]$ is a function $W_t^\dagger:{\cal X}_t^{\dagger,z}\times{\cal C}\times\Re\to\Re$ such that, for every $(s_t,z_t,c)\in{\cal X}_t^{\dagger,z}\times{\cal C}$, the mapping $v\to W_t^\dagger(s_t,z_t,c,v)$ is nondecreasing.
\end{defn}

\begin{defn}[SPA risk aggregator]
\label{defn:memory_risk_aggregator_separated}
An \emph{SPA risk aggregator} for period $t\in[T-1]$ is a function $M_t^\dagger:{\cal X}_t^{\dagger,y}\times{\cal L}\to\Re$ that satisfies: (i) if $\tilde{u}\leq\tilde{v}$, then $M_t^\dagger(s_t,y_t,\tilde{u})\leq M_t^\dagger(s_t,y_t,\tilde{v})$ for all $(s_t,y_t)\in{\cal X}_t^{\dagger,y}$; and (ii) $M_t^\dagger(s_t,y_t,{\bf v})=v$ for all $v\in\Re$ and $(s_t,y_t)\in{\cal X}_t^{\dagger,y}$.
\end{defn}

The SPA recursive structure below is defined with respect to a fixed SPA-state system.

\begin{defn}[SPA recursive structure]
\label{defn:SPA_recursive}
Let $\mathfrak{X}^\dagger$ be an SPA-state system. An {\it SPA recursive structure} on
$\mathfrak{X}^\dagger$ is a tuple
\[
{\cal V}^\dagger = (\{W_t^\dagger\}_{t=0}^{T-1}, \{M_t^\dagger\}_{t=0}^{T-1}, V_{T+1}^\dagger)
\]
where:
\begin{enumerate}
    \item $W_t^\dagger:{\cal X}_t^{\dagger,z}\times{\cal C}\times\Re\to\Re$ is a jointly continuous SPA time aggregator for all $t\in[T-1]$.
    \item $M_t^\dagger:{\cal X}_t^{\dagger,y}\times{\cal L}\to\Re$ is a jointly continuous SPA risk aggregator for all $t\in[T-1]$.
    \item $V_{T+1}^\dagger : {\cal X}_{T+1}^\dagger \rightarrow \Re$ is continuous.
\end{enumerate}
The pair $(\mathfrak{X}^\dagger, {\cal V}^\dagger)$ generates a utility system $\{U_t^\dagger\}_{t=0}^{T+1}$ where $U_{T+1}^\dagger = V_{T+1}^\dagger$ and $U_t^\dagger : {\cal X}_t^\dagger \times {\cal F}_t \rightarrow \Re$ are defined by:
\begin{subequations}
\label{eq:memory_recursive_separated}
\begin{align}
U_T^\dagger(\xi_T,c) = & U_{T+1}^\dagger(\Gamma_T^\dagger(\xi_T,c,\star)),\, \forall \xi_T \in {\cal X}_T^\dagger, c \in {\cal C},\\
U_t^\dagger(\xi_t,f) = & W_t^\dagger(s_t,z_t,c,M_t^\dagger(s_t,y_t,\tilde u_{t+1}^\dagger(\xi_t,f))),\, \forall \xi_t \in {\cal X}_t^\dagger, f = (c, f_+) \in {\cal F}_t, t \in [T-1],
\end{align}
\end{subequations}
where the separated next period continuation utility $\tilde{u}_{t+1}^\dagger(\xi_t, f) \in {\cal L}$ for $t \in [T-1]$ is defined by:
\[
[\tilde u_{t+1}^\dagger(\xi_t,f)](s)\triangleq U_{t+1}^\dagger(\Gamma_t^\dagger(\xi_t,c,s),f_+(s)),\, \forall s \in {\cal S}.
\]
\end{defn}

An SPA recursive structure represents ${\cal P}$ when there are belief and taste summary maps that connect back to the original utility system ${\bf U}$.

\begin{defn}[Separated recursive factorization]
\label{defn:memory_representation_separated}
Let $\mathfrak{X}^\dagger$ be an SPA-state system, ${\cal V}^\dagger$ be an SPA recursive structure on $\mathfrak{X}^\dagger$, and let $\{U_t^\dagger\}_{t=0}^{T+1}$ be the utilities generated by $(\mathfrak{X}^\dagger, {\cal V}^\dagger)$.
Then $(\mathfrak{X}^\dagger, {\cal V}^\dagger)$ \emph{factorizes} a compatible utility system ${\bf U}$ if, for all $t \in [T+1]$, there exist continuous summary maps $\sigma_t^y:{\cal H}_t\to{\cal Y}_t$ and $\sigma_t^z:{\cal H}_t\to{\cal Z}_t$ such that:
\[\sigma_t^\dagger(h_t)\triangleq(\sigma_t^s(h_t),\sigma_t^y(h_t),\sigma_t^z(h_t)) \in {\cal X}_t^\dagger,\, \forall h_t \in {\cal H}_t. 
\]

(i) (Utility factorization) The generated utilities $\{U_t^\dagger\}_{t=0}^{T+1}$ satisfy:
\begin{subequations}
\label{eq:factorization_separated}
\begin{align}
U_{T+1}(h_{T+1})= & U_{T+1}^\dagger(\sigma_{T+1}^\dagger(h_{T+1})),\, \forall h_{T+1} \in {\cal H}_{T+1},\\
U_t(h_t, f) = & U_t^\dagger(\sigma_t^\dagger(h_t),f),\, \forall h_t \in {\cal H}_t,\, f \in {\cal F}_t,\, t \in [T].
\end{align}
\end{subequations}

(ii) (Transition consistency) The summary maps are consistent with the transition maps:
\begin{equation}
\label{eq:transition_consistency_separated}
\sigma_{t+1}^\dagger(\iota_t(h_t,c,s)) =  \Gamma_t^\dagger(\sigma_t^\dagger(h_t),c,s),\, \forall h_t\in{\cal H}_t,\, 
c\in{\cal C},\, s \in {\cal S}_{t+1}, t\in[T].
\end{equation}
\end{defn}
\noindent
We note the parallel between
Definitions~\ref{defn:memory_representation}
and~\ref{defn:memory_representation_separated}: both impose utility
factorization and transition consistency.
When $(\mathfrak{X}^\dagger,{\cal V}^\dagger)$ factorizes a compatible utility system ${\bf U}$, it represents ${\cal P}$.

\subsection{Construction}

We now construct SPA belief and taste coordinates from a selected PA factorization on the canonical PA-state system.
This construction is based on equivalence relations for beliefs and tastes (rather than equivalence of total preferences as in the previous Section~\ref{sec:PA}).
Fix a PA factorization $(\mathfrak X^*,{\cal V}^*)$ of the compatible utility system ${\bf U}$ on the canonical PA-state system supplied by Theorem~\ref{thm:recursive_memory}, together with a rectangularization and the selected global aggregators in ${\cal V}^*$.
The PA-state system $\mathfrak X^*$ is canonical; the recursive structure ${\cal V}^*$ is selected because the global extensions may not be unique.
We define belief and taste equivalence in terms of this ${\cal V}^*$.

\begin{defn}[Belief and taste equivalence]
\label{defn:aggregator_belief_taste_equivalence}
(i) For $t\in[T-1]$, $m,m'\in{\cal M}_t$ are \emph{belief equivalent} $m\sim_t^y m'$ if $M_t^*(s,m,\tilde u)=M_t^*(s,m',\tilde u)$ for every $s\in{\cal S}_t$ and $\tilde u\in{\cal L}$.
For $t\in\{T,T+1\}$, $m\sim_t^y m'$ for all $m,m'\in{\cal M}_t$.

(ii) For $t\in[T-1]$, $m,m'\in{\cal M}_t$ are \emph{taste equivalent} $m\sim_t^z m'$ if $W_t^*(s,m,c,v)=W_t^*(s,m',c,v)$ for every $s\in{\cal S}_t$, $c\in{\cal C}$, and $v\in\Re$.
For $t\in\{T,T+1\}$, $m \sim_t^z m'$ if and only if $m=m'$.
\end{defn}
\noindent
There are no aggregators for periods $t \in [T, T+1]$, so the above equivalences are terminal conventions rather than aggregator-induced equivalences. Since there is no further uncertainty, the belief coordinate is inactive; all remaining terminal memory is retained in the taste coordinate.

While the canonical PA state is derived endogenously from ${\cal P}$, Definition~\ref{defn:aggregator_belief_taste_equivalence} implicitly depends on the selected global extensions $\{M_t^*\}_{t=0}^{T-1}$ and $\{W_t^*\}_{t=0}^{T-1}$ because the belief and taste quotients are defined globally.
The resulting belief/taste separation is canonical relative to the chosen rectangularization and the selected global aggregator extensions.

The next lemma shows that these are closed equivalence relations, so their quotient spaces are compact and metrizable.

\begin{lem}[Closedness of aggregator-induced equivalence]
\label{lem:closedness_aggregator_equivalence}
The relations $\{\sim_t^y\}_{t=0}^{T+1}$ and
$\{\sim_t^z\}_{t=0}^{T+1}$ are closed equivalence relations.
\end{lem}

We now define the aggregator-induced separated belief and taste states, as a quotient of the existing rectangular PA state.

\begin{defn}[Separated quotient maps and admissible separated states]
\label{defn:admissible_separated_states}
(i) For each $t\in[T+1]$, let ${\cal Y}_t={\cal M}_t/\sim_t^y$ and ${\cal Z}_t={\cal M}_t/\sim_t^z$, and let $p_t^y:{\cal M}_t\to{\cal Y}_t$ and $p_t^z:{\cal M}_t\to{\cal Z}_t$ denote the canonical quotient maps with respect to $\sim_t^y$ and $\sim_t^z$, respectively.

(ii) For each $t\in[T+1]$, let $\Xi_t:{\cal X}_t\to{\cal S}_t\times{\cal Y}_t\times{\cal Z}_t$ be defined by $\Xi_t(s_t,m_t)=(s_t,p_t^y(m_t),p_t^z(m_t))$. The SPA-state space is ${\cal X}_t^\dagger\triangleq\Xi_t({\cal X}_t)$.

(iii) For each $t\in[T+1]$, define $q_t^y:{\cal X}_t\to{\cal S}_t\times{\cal Y}_t$ and $q_t^z:{\cal X}_t\to{\cal S}_t\times{\cal Z}_t$ by $q_t^y(s_t,m_t)=(s_t,p_t^y(m_t))$ and $q_t^z(s_t,m_t)=(s_t,p_t^z(m_t))$. Since ${\cal X}_t^\dagger=\Xi_t({\cal X}_t)$, we have ${\cal X}_t^{\dagger,y}=q_t^y({\cal X}_t)$ and ${\cal X}_t^{\dagger,z}=q_t^z({\cal X}_t)$.
\end{defn}
\noindent
Since the belief and taste equivalence relations are closed by Lemma~\ref{lem:closedness_aggregator_equivalence} and each ${\cal M}_t$ is compact metrizable, the quotient spaces ${\cal Y}_t$ and ${\cal Z}_t$ are compact metrizable by Theorem~\ref{thm:closed_quotient}.
The canonical belief and taste summary maps are derived from the objects of Definition~\ref{defn:admissible_separated_states}. The belief summary map is $\sigma_t^y\triangleq\operatorname{pr}_{{\cal Y}_t}\circ\Xi_t\circ\sigma_t$ and the taste summary map is $\sigma_t^z\triangleq\operatorname{pr}_{{\cal Z}_t}\circ\Xi_t\circ\sigma_t$ for all $t\in[T+1]$, all of which are continuous as compositions of continuous maps. Since $\varsigma_t\circ\sigma_t=\sigma_t^s$ and $\Xi_t$ is the identity on ${\cal S}_t$ in the first coordinate, we have $\Xi_t\circ\sigma_t=(\sigma_t^s,\sigma_t^y,\sigma_t^z)$. The overall canonical SPA summary of Definition~\ref{defn:memory_representation_separated} is then $\sigma_t^\dagger=\Xi_t\circ\sigma_t$.

Under the following assumption, the belief and taste coordinates jointly recover the original PA state.
In this case, the triple of physical, belief, and taste states carries the same information about preferences as the original PA state.

\begin{assumption}[Aggregator exhaustiveness]
\label{assu:exhaustiveness}
For every $t\in[T+1]$, if $m,m'\in{\cal M}_t$ satisfy $m\sim_t^y m'$ and $m\sim_t^z m'$, then $m=m'$.
\end{assumption}
\noindent
Assumption~\ref{assu:exhaustiveness} requires that no variation in preference memory be invisible to both aggregators.

\begin{lem}[Identification of SPA states]
\label{lem:E_identification}
Suppose Assumption~\ref{assu:exhaustiveness} holds. Then $\Xi_t:{\cal X}_t\to{\cal X}_t^\dagger$ is a homeomorphism for every $t\in[T+1]$.
\end{lem}
\noindent
This lemma establishes that the SPA triple $(s_t,y_t,z_t)$ is a complete representation of the canonical PA state, so passing to separated coordinates does not change evaluation of plans.
Assumption~\ref{assu:exhaustiveness} is exactly injectivity of the mapping $\Xi_t$.
This identification is also part of the construction of the canonical SPA transition.

\begin{defn}[Canonical SPA transition]
\label{defn:canonical_separated_transitions}
Suppose Assumption~\ref{assu:exhaustiveness} holds.
For $t\in[T]$, let $\Gamma_t^\dagger:{\cal X}_t^\dagger\times{\cal C}\times{\cal S}_{t+1}\to{\cal X}_{t+1}^\dagger$ be defined by $\Gamma_t^\dagger(\xi_t,c,s)\triangleq\Xi_{t+1}(\Gamma_t(x_t,c,s))$ for $x_t=\Xi_t^{-1}(\xi_t)$ and $s \in {\cal S}_{t+1}$.
\end{defn}
\noindent
Under this construction, the mappings $\{\Gamma_t^\dagger\}_{t=0}^T$ satisfy Eq.~\eqref{eq:transition_separated}.

\subsection{Main Result}

Now we show that a separated recursive representation exists once the canonical SPA state is established.
The following result can be viewed as a lossless reparameterization of the PA state.

\begin{thm}[SPA recursive representation]
\label{thm:memory_recursive_separated}
Fix a compatible utility system ${\bf U}$ for ${\cal P}$, and let $(\mathfrak{X}^*,{\cal V}^*)$ be a PA factorization of ${\bf U}$ on the canonical PA-state system supplied by Theorem~\ref{thm:recursive_memory}. Suppose Assumption~\ref{assu:rectangular} holds, and fix a rectangularization of $\mathfrak{X}^*$ together with the selected global aggregator extensions in ${\cal V}^*$. If Assumption~\ref{assu:exhaustiveness} holds, then there exists an SPA-state system $\mathfrak{X}^\dagger$ and an SPA recursive structure ${\cal V}^\dagger$ on $\mathfrak{X}^\dagger$ such that $(\mathfrak{X}^\dagger,{\cal V}^\dagger)$ factorizes the same compatible utility system ${\bf U}$. Consequently, $(\mathfrak{X}^\dagger,{\cal V}^\dagger)$ represents ${\cal P}$.
\end{thm}
\noindent
Theorem~\ref{thm:memory_recursive_separated} maps the selected PA factorization on the canonical PA-state system into an SPA factorization with distinct belief and taste coordinates. Under Assumption~\ref{assu:exhaustiveness}, the joint coordinate map is a homeomorphism, so the SPA state is a lossless reparameterization of the canonical PA state. Any remaining interaction appears through the admissible SPA-state space and transition maps, rather than through cross-dependence of the two aggregators.

\begin{rem}[Constructive use of SPA states]
In the constructive direction, we may specify $\mathfrak{X}^\dagger$ and ${\cal V}^\dagger$ satisfying Definitions~\ref{defn:memory_separated} and~\ref{defn:SPA_recursive}, respectively.
Choose $\xi_0\in{\cal X}_0^\dagger$ with $\varsigma_0^\dagger(\xi_0)=\bar s_0$, then define $\widehat\sigma_0^\dagger(\bar s_0)=\xi_0$ and recursively $\widehat\sigma_{t+1}^\dagger(\iota_t(h_t,c,s')) = \Gamma_t^\dagger(\widehat\sigma_t^\dagger(h_t),c,s')$.
Continuity of $\{\widehat\sigma_t^\dagger\}$ and the identities $\varsigma_t^\dagger \circ \widehat\sigma_t^\dagger = \sigma_t^s$ follow from the conditions on $\{\Gamma_t^\dagger\}_{t=0}^T$. Define the induced full-history utility system $\widehat{\bf U}=\{\widehat U_t\}_{t=0}^{T+1}$ by:
\begin{align*}
    \widehat U_t(h_t,f) = & U_t^\dagger(\widehat\sigma_t^\dagger(h_t),f),\, \forall h_t \in {\cal H}_t,\, f \in {\cal F}_t,\, t\in[T],\\
    \widehat U_{T+1}(h_{T+1}) = & U_{T+1}^\dagger (\widehat\sigma_{T+1}^\dagger(h_{T+1})),\, \forall h_{T+1} \in {\cal H}_{T+1}.
\end{align*}
Then, $(\mathfrak X^\dagger,{\cal V}^\dagger)$ is a factorization of $\widehat{\bf U}$ on reachable states. Assumptions~\ref{assu:rectangular} and~\ref{assu:exhaustiveness} are needed to derive the SPA representation from the canonical PA one; they are not needed for the constructive direction.
\end{rem}

\subsection{Separated Dynamic Programming}

We can analogously develop a Bellman recursion for an MDP with respect to $(\mathfrak{X}^\dagger, {\cal V}^\dagger)$.
We impose the following modified DP regularity conditions.
As in the PA case, we take the separated summary maps to be onto. If not, replace each ${\cal X}_t^\dagger$ by the compact reachable image $\sigma_t^\dagger({\cal H}_t)$ and restrict the SPA transition accordingly.

\begin{assumption}[Separated Markov feasibility]
\label{assu:separated_Markov_feasibility}
If $\sigma_t^\dagger(h_t)=\sigma_t^\dagger(h_t')$, then ${\cal A}_t(h_t)={\cal A}_t(h_t')$, and we write ${\cal A}_t^\dagger(\xi_t) \triangleq {\cal A}_t(h_t)$ for all $h_t \in (\sigma_t^\dagger)^{-1}(\xi_t)$.
\end{assumption}

\begin{assumption}[Separated DP regularity]
\label{assu:DP_separated_regular}
For all $t\in[T]$, the correspondence ${\cal A}_t^\dagger:{\cal X}_t^\dagger\rightrightarrows{\cal C}$ has nonempty compact values and is continuous.
\end{assumption}
\noindent
In many examples, feasibility depends only on $(s_t,z_t)$, in which case we write ${\cal A}_t^\dagger(s_t,z_t)$.

\begin{defn}[SPA Markov policy]
\label{defn:separated_Markov_policy}
An \emph{SPA Markov policy} is a sequence $\pi^\dagger = \{\pi_t^\dagger\}_{t=0}^T$, where $\pi_t^\dagger : {\cal X}_t^\dagger \rightarrow {\cal C}$ such that $\pi_t^\dagger(\xi_t) \in {\cal A}_t^\dagger(\xi_t)$ for all $\xi_t\in{\cal X}_t^\dagger$ and $t\in[T]$.
Let $\Pi^\dagger = \prod_{t=0}^T {\rm Sel}({\cal A}_t^\dagger)$ denote the class of separated feasible Markov policies.
\end{defn}

We have the following Bellman recursion for the separated case.
Let $U_t^{\dagger\pi^\dagger} : {\cal X}_t^\dagger \rightarrow \Re$ be the period $t \in [T+1]$ policy utility function of a Markov policy $\pi^\dagger \in \Pi^\dagger$ on the separated state, defined recursively by $U_{T+1}^{\dagger\pi^\dagger} \triangleq U_{T+1}^\dagger$ and 
\begin{subequations}
\label{eq:memory_recursive_policy_separated}
\begin{align}
U_T^{\dagger\pi^\dagger}(\xi_T) = & U_{T+1}^{\dagger\pi^\dagger}(\Gamma_T^\dagger(\xi_T, \pi_T^\dagger(\xi_T),\star)),\, \forall \xi_T \in {\cal X}_T^\dagger,\\
U_t^{\dagger\pi^\dagger}(\xi_t) = & W_t^\dagger(s_t,z_t, \pi_t^\dagger(\xi_t), M_t^\dagger(s_t,y_t, \tilde{u}_{t+1}^{\dagger\pi^\dagger}(\xi_t, \pi_t^\dagger(\xi_t)))),\, \forall \xi_t \in {\cal X}_t^\dagger, t \in [T-1],
\end{align}
\end{subequations}
where $\tilde{u}_{t+1}^{\dagger\pi^\dagger}(\xi_t, c) \in {\cal L}$ is defined by $[\tilde{u}_{t+1}^{\dagger\pi^\dagger}(\xi_t, c)](s) \triangleq U_{t+1}^{\dagger\pi^\dagger}(\Gamma_t^\dagger(\xi_t, c, s))$ for all $s \in {\cal S}$.
Let $J_t^\dagger : {\cal X}_t^\dagger \rightarrow \Re$ be the separated period $t \in [T]$ optimal value function, defined by:
\[
J_t^\dagger(\xi_t) = \sup_{\pi^\dagger \in \Pi^\dagger} U_t^{\dagger \pi^\dagger}(\xi_t),\, \forall \xi_t \in {\cal X}_t^\dagger, t \in [T].
\]
The random continuation utility $\tilde{j}_{t+1}^\dagger(\xi_t, c)$ for $t \in [T-1]$ on the separated belief and taste state space is defined by $[\tilde{j}_{t+1}^\dagger(\xi_t, c)](s) \triangleq J_{t+1}^\dagger(\Gamma_t^\dagger(\xi_t, c, s))$ for all $s \in {\cal S}$.

\begin{cor}[Bellman recursion, separated]
\label{cor:DP_separated}
Suppose $(\mathfrak{X}^\dagger,{\cal V}^\dagger)$, with summary maps $\{\sigma_t^\dagger\}_{t=0}^{T+1}$, factorizes a compatible utility system ${\bf U}$ for ${\cal P}$, and Assumptions~\ref{assu:separated_Markov_feasibility} and~\ref{assu:DP_separated_regular} hold.
Then $\{J_t^\dagger\}_{t=0}^T$ satisfy:
\begin{subequations}
\label{eq:Bellman_separated}
\begin{align}
J_T^\dagger(\xi_T) = & \max_{c \in {\cal A}_T^\dagger(\xi_T)} U_{T+1}^\dagger(\Gamma_T^\dagger(\xi_T,c,\star)),\, \forall \xi_T \in {\cal X}_T^\dagger,\\
J_t^\dagger(\xi_t) = & \max_{c \in {\cal A}_t^\dagger(\xi_t)} W_t^\dagger(s_t,z_t, c, M_t^\dagger(s_t,y_t, \tilde{j}_{t+1}^\dagger(\xi_t, c))),\, \forall \xi_t \in {\cal X}_t^\dagger,\, t \in [T-1].
\end{align}
\end{subequations}
The maxima are attained. There exists a single $\bar\pi^\dagger\in\Pi^\dagger$ which attains the displayed maxima at every state and time, and satisfies $U_t^{\dagger\bar\pi^\dagger}(\xi_t)=J_t^\dagger(\xi_t)$ for every $\xi_t\in{\cal X}_t^\dagger$ and $t\in[T]$. Moreover, each $J_t^\dagger$ is continuous.
\end{cor}

The proof of Corollary~\ref{cor:DP_separated} shows that there exists at least one SPA selector $\bar\pi^\dagger\in\Pi^\dagger$ satisfying:
\begin{subequations}
\label{eq:Bellman_selector_separated}
\begin{align}
\bar\pi_T^\dagger(\xi_T) \in & \arg\max_{c\in{\cal A}_T^\dagger(\xi_T)} U_{T+1}^\dagger(\Gamma_T^\dagger(\xi_T,c,\star)),\, \forall \xi_T\in{\cal X}_T^\dagger,\\
\bar\pi_t^\dagger(\xi_t) \in & \arg\max_{c\in{\cal A}_t^\dagger(\xi_t)} W_t^\dagger(s_t,z_t,c,M_t^\dagger(s_t,y_t,\tilde{j}_{t+1}^\dagger(\xi_t,c))),\, \forall \xi_t\in{\cal X}_t^\dagger,\, t\in[T-1].
\end{align}
\end{subequations}
Any Bellman selector for Eq.~\eqref{eq:Bellman_selector_separated} induces a history-dependent policy $\bar\varphi \in \Phi$ through the separated summary maps, defined by $\bar\varphi_t(h_t)=\bar\pi_t^\dagger(\sigma_t^\dagger(h_t))$ for all $h_t \in {\cal H}_t$ and $t \in [T]$.

\begin{cor}[Verification, separated]
\label{cor:DP_verification_separated}
Suppose $(\mathfrak{X}^\dagger,{\cal V}^\dagger)$, with summary maps $\{\sigma_t^\dagger\}_{t=0}^{T+1}$, factorizes a compatible utility system ${\bf U}$ for ${\cal P}$, and Assumptions~\ref{assu:separated_Markov_feasibility} and~\ref{assu:DP_separated_regular} hold.
Let $\bar{\pi}^\dagger \in \Pi^\dagger$ be any Bellman selector satisfying Eq.~\eqref{eq:Bellman_selector_separated}, and let $\bar{\varphi}$ be its induced full-history policy.
Then, for every $\varphi \in \Phi$, 
\[
U_t^\varphi(h_t)\leq J_t^\dagger(\sigma_t^\dagger(h_t))=U_t^{\bar\varphi}(h_t),\, \forall h_t\in{\cal H}_t,\, t\in[T].
\]
\end{cor}
\noindent
It follows that $\bar\varphi$ is optimal in $\Phi$.
Solving the value problems in Eq.~\eqref{eq:Bellman_separated} may be simplified by the separated structure of the time and risk aggregators.

\section{Examples}
\label{sec:examples}

This section develops constructive examples of economically interpretable PA or SPA-state systems and aggregators that satisfy the conditions of the recursive and DP results.
These examples are categorized into five groups. We first discuss history-independent models where the preference-memory component is trivial. Next, we discuss SPA preferences with either active belief or taste coordinates (and the other inactive). Then we consider SPA preferences with decoupled transitions followed by SPA models with coupled transitions. Finally, we consider fully entangled PA models.

We let
\[
\mathbb{E}_{s_t}[\tilde{u}] \triangleq \sum_{s' \in {\cal S}} q_t(s' \vert s_t) \tilde{u}(s'),\, \forall \tilde{u} \in {\cal L},
\]
denote expectation with respect to the physical transition kernel.
We also let $u : I \rightarrow \Re$ be an instantaneous consumption utility on a compact interval $I$ which contains the effective domains needed below, and $\beta \in (0, 1]$ be a discount factor.

The displayed time and risk aggregators and non-terminal memory updates apply for periods $t\in[T-1]$. Let $\bar U_T : {\cal X}_T \times {\cal C} \rightarrow \Re$ be the desired continuous terminal utility which can depend on the period $T$ PA state $x_T$ and consumption $c_T$.
Then, we take ${\cal X}_{T+1}=\{\star\}\times{\cal X}_T\times{\cal C}$, $\Gamma_T(x_T,c,\star) = (\star,x_T,c)$, and $V_{T+1}^*(\star,x_T,c) = \bar U_T(x_T,c)$.
For an analogous SPA implementation, let $\bar U_T^\dagger : {\cal X}_T^\dagger \times {\cal C} \rightarrow \Re$ be the desired continuous terminal utility. Then, we take ${\cal X}_{T+1}^\dagger = \{\star\}\times\{\ast\}\times ({\cal X}_T^\dagger\times{\cal C})$, $\Gamma_T^\dagger(\xi_T,c,\star)=(\star,\ast,(\xi_T,c))$, and $V_{T+1}^\dagger(\star,\ast,(\xi_T,c))=\bar U_T^\dagger(\xi_T,c)$.

Starred primitives refer to PA preferences and daggered primitives refer to SPA preferences.
Inactive belief or taste coordinates are denoted by $\ast$, and inactive aggregator arguments are suppressed.
Some displayed formulas are effective-domain specifications, for example when logarithms, powers, or ratios are used. We assume the hypotheses of our extension results hold, so we may select global continuous monotone extensions where needed.

\subsection{History Independence}

The preference-memory component is inactive in history-independent models, so ${\cal M}_t = \{\ast\}$ and $x_t = (s_t, \ast)$ for all $t \in [T]$. 
The transition map is $\Gamma_t((s_t,\ast),c,s')=(s',\ast)$ for all $t \in [T-1]$.
The time and risk aggregators only depend on the current physical Markov state, and are written $W_t^*(s_t, \cdot)$ and $M_t^*(s_t, \cdot)$ for all $t \in [T-1]$.
These examples are trivially PA preferences, but they also automatically satisfy separation of beliefs and tastes.

\begin{example}[Discounted expected utility]
The time aggregator $W_t^*(c, v) = u(c) + \beta v$ and risk aggregator $M_t^*(s_t, \tilde{u}) = \mathbb{E}_{s_t}[\tilde{u}]$ recover discounted expected utility.
\end{example}

\begin{example}[Epstein--Zin]
Let $\rho \in (0,1)$ be the intertemporal substitution parameter and $\gamma \in (0,1)$ be the risk-aversion parameter.
Take the CES time aggregator $W_t^*(c, v) = [(1-\beta) c^\rho + \beta v^\rho]^{1/\rho}$ and the power mean CRRA risk aggregator $M_t^*(s_t, \tilde{u}) = (\mathbb{E}_{s_t}[ \tilde{u}^\gamma ])^{1/\gamma}$ to recover Epstein-Zin preferences.
\end{example}

\begin{example}[Fixed-prior smooth ambiguity]
This example considers ambiguity over subjective transition models.
Let $\Theta$ be a compact set of transition models, and suppose $\theta \to p_\theta(\cdot \vert s_t)$ is continuous.
Each $\theta \in \Theta$ indexes a subjective candidate transition kernel $p_\theta(\cdot \vert s_t) \in \Delta({\cal S})$.
For each $t \in [T-1]$, let $\nu_t$ be a prior distribution over $\Theta$.
Let $\phi : \Re \rightarrow \Re$ model DM's ambiguity attitude, and $\psi : \Re \rightarrow \Re$ model DM's risk attitude, where both $\phi$ and $\psi$ are strictly increasing homeomorphisms with relevant domains.
Let $\langle p_\theta(\cdot \vert s_t),\psi(\tilde{u})\rangle \triangleq \sum_{s' \in {\cal S}} p_\theta(s' \vert s_t) \psi(\tilde{u}(s'))$ denote expectation with respect to $p_\theta(\cdot \vert s_t)$.
Then, take the time aggregator $W_t^*(c, v) = u(c) + \beta v$ and the risk aggregator
\[
M_t^*(s_t, \tilde{u}) = \phi^{-1}\!\left(\int_{\Theta}
\phi\!\left(\psi^{-1}\!\left(\langle p_\theta(\cdot \vert s_t),\psi(\tilde{u})\rangle\right)\right)
d\nu_t(\theta) \right),
\]
to recover recursive smooth ambiguity preferences over subjective transition models (see \cite{klibanoff2009recursive}). No preference-memory state is required since the priors $\{\nu_t\}_{t=0}^{T-1}$ are fixed.
\end{example}

\begin{example}[Markov risk measures]
In this example, an adversary can perturb the reference transition kernels $q_t(\cdot \vert s_t)$.
Let $l_t : {\cal S} \rightarrow \Re_{\geq 0}$ be the adversary's perturbation density with respect to $q_t(\cdot \vert s_t)$ (corresponding to the Radon--Nikodym derivative), subject to the constraint $\mathbb{E}_{s_t}[l_t] = 1$. The distorted probability density due to $l_t$ is then $p_t(s' \vert s_t) = l_t(s') q_t(s' \vert s_t)$ for all $s' \in {\cal S}$.
Let $\Delta_t(s_t) \triangleq \{l_t \in \Re_{\geq 0}^{\cal S} : \mathbb{E}_{s_t}[l_t] = 1\}$ denote the set of admissible densities in period $t$ in state $s_t$.
Let ${\cal U}_t(s_t)\subseteq\Delta_t(s_t)$ be a nonempty compact convex set of admissible densities representing the adversary's uncertainty set.
Then, take the time aggregator $W_t^*(c, v) = c + v$ and the risk aggregator $M_t^*(s_t, \tilde{u}) = \inf_{l_t \in {\cal U}_t(s_t)} \mathbb{E}_{s_t} [l_t \tilde{u}]$ for $t \in [T-1]$ to recover nested Markov risk measures (see \cite{ruszczynski2010risk}).
\end{example}

\subsection{Pure Beliefs}

The taste state is inactive in all the examples in this subsection, so ${\cal Z}_t=\{\ast\}$ and $\xi_t=(s_t,y_t,\ast)$ for all $t \in [T]$.

\begin{example}[Bayesian learning]
\emph{State system.} Let $\Theta=\{\theta_1,\ldots,\theta_J\}$ index a finite family of full-support subjective transition models $p_\theta(\cdot\vert s_t)$, possibly misspecified relative to the physical kernel $q_t$.
The belief coordinate is the posterior distribution $y_t \in \Delta(\Theta)$ where $y_t(\theta)=\Pr(\theta\vert h_t)$ for all $\theta \in \Theta$.
The separated PA state is $\xi_t=(s_t,y_t,\ast)\in{\cal S}_t\times\Delta(\Theta)\times\{\ast\}$.
After observing $s_{t+1}=s'$, we update $y_t$ by Bayes' rule to get:
\[
[\Gamma_t^y(s_t,y_t,c_t,s')](\theta) = \frac{ p_\theta(s'\vert s_t)y_t(\theta)}{\sum_{\theta'\in\Theta}p_{\theta'}(s'\vert s_t)y_t(\theta')},\, \forall \theta \in \Theta.
\]
Then the overall transition map is $\Gamma_t^\dagger((s_t,y_t,\ast),c_t,s') = \left( s', \Gamma_t^y(s_t,y_t,c_t,s'), \ast \right)$.

\emph{Aggregators.} We can take the standard additive time aggregator $W_t^\dagger(c,v)=u(c)+\beta v$.
To represent smooth ambiguity over the candidate transition models, let $\phi$ and $\psi$ be strictly increasing homeomorphisms on the relevant domains, and define
\[
M_t^\dagger(s_t, y_t, \tilde{u}) = \phi^{-1}\!\left( \sum_{\theta \in \Theta}
\phi\!\left(\psi^{-1}\!\left(\langle p_\theta(\cdot \vert s_t),
\psi(\tilde{u})\rangle\right)\right) y_t(\theta) \right),
\]
which evaluates smooth ambiguity under the current posterior (see \cite{ju2012ambiguity}).
\end{example}

\begin{example}[Hidden-regime filtering]
This example is a finite-state SPA representation of hidden-regime beliefs. The physical state process remains governed by $q_t(\cdot\vert s_t)$. The decision maker nevertheless evaluates uncertainty using a subjective hidden-regime model.

\emph{State system.}
We suppose the regimes are $\zeta_t\in\{0,1\}$, where $\zeta_t=1$ is good and $\zeta_t=0$ is bad.
The hidden regime is not part of the physical state; it is a subjective model used to evaluate transition uncertainty.
The belief state is $y_t=\Pr(\zeta_t=1\vert h_t)\in{\cal Y}\triangleq[0,1]$, and the SPA state is $\xi_t=(s_t,y_t,\ast)\in{\cal S}_t\times{\cal Y}\times\{\ast\}$.

The regime transition matrix $P=(P_{ij})_{i,j\in\{0,1\}}$ is known, where $P_{ij} = {\rm Pr}(\zeta_{t+1} = j \vert \zeta_t = i) > 0$. Conditional on the next regime $i\in\{0,1\}$ and current observed state $s_t$, DM assigns next-state probabilities $p_i(\cdot \vert s_t) \in \Delta({\cal S})$ (which have full support) for $i \in \{0,1\}$.
Given $y_t$, the forecast probability of the good regime in period $t+1$ is $\widehat y_{t+1} = P_{11}y_t+P_{01}(1-y_t)$.
After observing $s_{t+1}=s'$, Bayes' rule gives
\[
\Gamma_t^y(s_t,y_t,c_t,s') = \frac{ \widehat y_{t+1}p_1(s'\vert s_t) }{ \widehat y_{t+1}p_1(s'\vert s_t) + (1-\widehat y_{t+1})p_0(s'\vert s_t)},
\]
where the denominator is strictly positive by the full support assumption.
The overall SPA transition map is then $\Gamma_t^\dagger((s_t,y_t,\ast),c_t,s') = \left( s', \Gamma_t^y(s_t,y_t,c_t,s'), \ast \right)$.

\emph{Aggregators.}
Take the standard $W_t^\dagger(c,v)=u(c)+\beta v$. The predictive distribution of the next observed state is
\[
\widehat p(s'\vert s_t,y_t) \triangleq \widehat y_{t+1}p_1(s'\vert s_t) + (1-\widehat y_{t+1})p_0(s'\vert s_t),\, \forall s' \in {\cal S},
\]
and we take the aggregator $M_t^\dagger(s_t,y_t,\tilde u) = \sum_{s' \in {\cal S}}\widehat p(s'\vert s_t,y_t)\tilde u(s')$.
The dependence on the hidden regime enters only through the belief coordinate $y_t$ and the subjective risk aggregator $M_t^\dagger$.
\end{example}

\begin{example}[Multiplier preferences]
This example recovers the multiplier preferences of \cite{hansen2001robust}.
Let ${\cal S}=\{0,1\}$ where $s_t=1$ is good and $s_t=0$ is bad, and let $y_t \in {\cal Y} = [0,1]$ be DM's confidence.
The confidence transition map is $y_{t+1} = \Gamma_t^y(y_t, s_{t+1}) = \alpha s_{t+1} + (1-\alpha) y_t$ for $\alpha \in [0,1]$.

We again take the time aggregator $W_t^\dagger(c, v) = u(c) + \beta v$. Let $\gamma:{\cal Y} \to \Re_{>0}$ be continuous and increasing. Then the multiplier-based risk aggregator is:
\[
M_t^\dagger(s_t, y_t, \tilde{u}) = -\gamma(y_t) \ln \mathbb{E}_{s_t}[\exp(-\tilde{u}/\gamma(y_t))].
\]
Higher confidence raises ambiguity tolerance: as $\gamma(y_t) \rightarrow \infty$, we approach a risk-neutral aggregator, while $\gamma(y_t) \rightarrow 0$ approaches the minimum continuation utility over the support of $q_t(\cdot\vert s_t)$.
\end{example}

\subsection{Pure Tastes}

In all examples in this subsection, the belief state is inactive at nonterminal dates so ${\cal Y}_t=\{\ast\}$ and $\xi_t=(s_t,\ast,z_t)$.
First let $z_t$ be the habit stock, then fix a compact interval ${\cal Z}$ and suppose feasible consumption keeps $z_t\in{\cal Z}$.
We initialize with $z_0 > 0$, and the normalized habit transition map is $\Gamma_t^z(z_t,c_t) = \alpha z_t+(1-\alpha)c_t$, for persistence parameter $\alpha \in [0,1]$.
The overall transition map is then $\Gamma_t^\dagger((s_t,\ast,z_t),c_t,s') = (s',\ast,\Gamma_t^z(z_t,c_t))$.

\begin{example}[Additive habit]
Take $W_t^\dagger(z_t, c, v) = u(c - z_t) + \beta v$ and $M_t^\dagger(s_t, \tilde{u}) = \mathbb{E}_{s_t}[\tilde{u}]$.
The instantaneous utility from consumption depends on the additive difference from the habit stock.
\end{example}

\begin{example}[Ratio habit]
For a ratio habit, we take $W_t^\dagger(z_t, c, v) = u(c/z_t) + \beta v$ for $c, z_t > 0$, and $M_t^\dagger(s_t, \tilde{u}) = (\mathbb{E}_{s_t}[\tilde{u}^\gamma])^{1/\gamma}$.
The instantaneous utility from consumption depends on the ratio with respect to the habit stock.
\end{example}

\begin{example}[State-dependent habit]
This example models state-dependent habit formation.
Suppose the taste transition map is now $z_{t+1} = \Gamma_t^z(z_t, c_t, s_{t+1}) = \alpha(s_{t+1}) z_t + (1 - \alpha(s_{t+1})) c_t$, where $\alpha(s_{t+1}) \in (0,1)$ reflects state-dependent persistence, e.g., habits form faster during a crisis.
We then take the time aggregator $W_t^\dagger(z_t,c,v)=u(c-z_t)+\beta v$ and the risk aggregator $M_t^\dagger(s_t, \tilde{u}) = \mathbb{E}_{s_t}[\tilde{u}]$.
\end{example}

\begin{example}[Relative consumption]
In this example, the evolution of the habit stock depends on a state-dependent consumption benchmark, such as an aggregate economic measure.
Let $\tilde{C}(s_t)$ be the state-dependent benchmark, and $z_{t+1} = \Gamma_t^z(z_t, c_t, s_{t+1}) = \alpha z_t + (1-\alpha)(c_t - \tilde{C}(s_{t+1}))$.
Let $g : {\cal Z} \rightarrow \Re$ be a continuous function that measures additional utility from the relative standing captured by $z_t$.
We can then take $W_t^\dagger(z_t, c, v) = u(c) + g(z_t) + \beta v$ and $M_t^\dagger(s_t, \tilde{u}) = (\mathbb{E}_{s_t}[\tilde{u}^\gamma])^{1/\gamma}$ for $\gamma \in (0,1)$.
\end{example}

\begin{example}[Durable stock]
For this example, let $z_t$ be a durable stock with unnormalized transition map $z_{t+1} = \Gamma_t^z(z_t, c_t) = \alpha z_t + c_t$.
We take the Cobb-Douglas time aggregator $W_t^\dagger(z_t, c, v) = u(c^{\theta} z_t^{1-\theta}) + \beta v$ for $c, z_t > 0$ and $\theta \in (0,1)$, and the risk aggregator $M_t^\dagger(s_t, \tilde{u}) = \mathbb{E}_{s_t}[\tilde{u}]$.
The stock of durable goods increases utility from present consumption. 
\end{example}

For the next example, let $z_t$ denote wealth.
Wealth often appears as a physical constraint, e.g., as a budget constraint which does not affect tastes. Here we emphasize wealth as a taste state.
Let $\tilde{R} \in {\cal L}$ be a random return rate (e.g., market returns) and $\tilde{Y} \in {\cal L}$ be random income.
The wealth is initialized with $z_0 > 0$, and the transition map is:
\begin{equation}
\label{eq:wealth}
z_{t+1} = \Gamma_t^z(z_t, c_t, s_{t+1}) = \tilde{R}(s_{t+1}) (z_t - c_t) + \tilde{Y}(s_{t+1}),\, \forall t \in [T-1].
\end{equation}
At date $T$, use the standing terminal convention with terminal payoff $\bar U_T^\dagger((s_T,\ast,z_T),c_T) = u_T(c_T,z_T-c_T)$.
In the wealth examples throughout the rest of this section, we suppose returns are bounded and positive, income is bounded, and $0\le c_t \leq z_t$.

\begin{example}[Wealth-dependent utility]
We can take the modified EZ time aggregator
\[
W_t^\dagger(z_t, c, v) = [(1-\beta) (c^{1-\theta} z_t^\theta)^\rho + \beta v^\rho]^{1/\rho}
\]
for $\theta \in (0,1)$, where wealth increases the enjoyment of current consumption.
The alternative EZ time aggregator $W_t^\dagger(z_t, c, v) = [(1-\beta) c^\rho + \beta (v^{1-\theta} z_t^\theta)^\rho]^{1/\rho}$ means wealth increases 
future continuation utility.
We can then take the standard risk aggregator $M_t^\dagger(s_t, \tilde{u}) = \mathbb{E}_{s_t}[\tilde{u}]$.
\end{example}

\subsection{SPA with Decoupled Transitions}

These examples have both active belief and taste states, with decoupled transitions.

\begin{example}[Sentiment and habit]
This example combines sentiment and habit stock with separated aggregators and decoupled transitions.

\emph{State system.}
Let ${\cal S}=\{0,1\}$, where $s=1$ is a normal state and $s=0$ is a crash. Let $y_t\in[0,1]$ be the sentiment and let $z_t$ be the habit stock.
The transition maps are $\Gamma_t^y(y_t,s_{t+1}) = \alpha_1 s_{t+1}+(1-\alpha_1)y_t$ for $\alpha_1\in(0,1)$, and $\Gamma_t^z(z_t,c_t) = \alpha_2 z_t+(1-\alpha_2)c_t$ for $\alpha_2\in(0,1)$.

\emph{Aggregators.}
Take the taste-dependent time aggregator $W_t^\dagger(z_t,c,v)=u(c-z_t)+\beta v$.
Let $p:[0,1]\to[0,1]$ be continuous and increasing, where $p(y_t)$ is the subjective probability of the normal state $s=1$. Then define the risk aggregator $M_t^\dagger(y_t,\tilde{u}) = p(y_t)\tilde{u}(1) + (1-p(y_t))\tilde{u}(0)$.
The dependence on the next habit and sentiment states enters through the continuation utility
\[
[\tilde{u}_{t+1}^\dagger(\xi_t,f)](s)=U_{t+1}^\dagger(\Gamma_t^\dagger(\xi_t,c,s),f_+(s)),\, \forall s \in {\cal S},
\]
where $\Gamma_t^\dagger(\xi_t,c,s) = (s, \Gamma_t^y(y_t,s), \Gamma_t^z(z_t,c))$, but not directly through the risk aggregator.
\end{example}

\begin{example}[Sentiment and wealth]
This example combines wealth and robustness with separated aggregators, and decoupled transitions.

\emph{State system.} Let ${\cal S}=\{0,1\}$ where $s_t = 1$ is a market boom and $s_t = 0$ is a market crash. Let $y_t \in {\cal Y}=[0,1]$ denote the sentiment and $z_t \in {\cal Z} =[\underline z,\overline z] \subset \Re_{>0}$ denote the cumulative wealth.
The transition maps are $y_{t+1} = \Gamma_t^y(y_t, s_{t+1}) = \alpha s_{t+1} + (1-\alpha)y_t$ for the sentiment and $z_{t+1} = \Gamma_t^z(z_t, c_t, s_{t+1}) = R(s_{t+1})(z_t - c_t) + Y(s_{t+1})$ for the wealth.

\emph{Aggregators.}
We can take the taste-dependent time aggregator $W_t^\dagger(z_t, c, v) = u(c^\theta z_t^{1-\theta}) + \beta v$.
Let $\gamma : {\cal Y} \rightarrow \Re_{>0}$ be a continuous and increasing ambiguity tolerance function (as sentiment improves, DM becomes less ambiguity-averse).
Then take the risk aggregator $M_t^\dagger(s_t, y_t, \tilde{u}) = - \gamma(y_t) \ln \mathbb{E}_{s_t} [\exp{(-\tilde{u}/\gamma(y_t))}]$.
The belief and wealth states are correlated through the common state, but neither update depends directly on the other preference coordinate and they are separated through the time and risk aggregators.
\end{example}

\subsection{SPA with Coupled Transitions}

These examples have separated belief and taste aggregators, with coupled belief and taste transitions.
For a compact interval $I=[\underline a,\overline a]$, let $\Pi_I(r)\triangleq\min\{\overline a,\max\{\underline a,r\}\}$ for $r \in \Re$ denote clipping onto $I$.

\begin{example}[Wealth-driven confidence]
Let ${\cal Y}=[\underline y,\overline y]\subset \Re_{>0}$ and ${\cal Z}=[\underline z,\overline z]\subset \Re_{>0}$.
The belief state $y_t$ is confidence, and the taste state $z_t$ is wealth.
Wealth evolves according to $z_{t+1} = \Gamma_t^z(z_t, c_t, s_{t+1}) = (z_t - c_t) R(s_{t+1})$, so the expected wealth is
\[
\bar z_{t+1}(s_t,z_t,c_t) \triangleq \sum_{s' \in{\cal S}} q_t(s' \vert s_t) \Gamma_t^z(z_t,c_t,s').
\]
Confidence follows the dynamic:
\[
y_{t+1} = \Gamma_t^y(\xi_t,c_t,s') = \Pi_{\cal Y} \left[ \exp\left( \alpha\log y_t + (1-\alpha) \frac{\Gamma_t^z(z_t,c_t,s')- \bar z_{t+1}(s_t,z_t,c_t)}{z_t} \right) \right],
\]
which depends on the realized wealth relative to expected wealth.
We can then take the time aggregator $W_t^\dagger(z_t, c, v) = u(c^\theta z_t^{1-\theta})+\beta v$ for $\theta \in (0,1)$.
Let $\gamma : {\cal Y} \rightarrow \Re_{>0}$ be continuous and increasing, and define $M_t^\dagger(s_t, y_t, \tilde{u}) = -\gamma(y_t) \ln \mathbb{E}_{s_t}[\exp(-\tilde{u}/\gamma(y_t))]$.
The confidence transition includes an adjustment for the realized relative change in wealth (and thus depends on wealth, so the transitions are coupled). 
\end{example}

\begin{example}[Optimism-driven aspirations]
\emph{State system.}
Let ${\cal S}=\{0,1\}$, where $1$ is a favorable state. Let ${\cal Y}=[0,1]$ and ${\cal Z}=[\underline z,\overline z]$, and write $\xi_t=(s_t,y_t,z_t)\in{\cal S}_t\times{\cal Y}\times{\cal Z}$.
The belief coordinate $y_t$ is optimism, interpreted as the subjective probability of the favorable state. The taste coordinate $z_t$ is a reference stock.

Let $\alpha_1, \alpha_2 \in[0,1]$ and $\eta\geq0$. Define $\Gamma_t^y(\xi_t,c_t,s') = \alpha_1 y_t+(1-\alpha_1)s'$ and
\[
\Gamma_t^z(\xi_t,c_t,s') = \Pi_{\cal Z} \left( \alpha_2 z_t+(1-\alpha_2)c_t+\eta y_t\right),
\]
so the overall transition is $\Gamma_t^\dagger(\xi_t,c_t,s') = \left( s', \Gamma_t^y(\xi_t,c_t,s'), \Gamma_t^z(\xi_t,c_t,s') \right)$.

\emph{Aggregators.}
Let $W_t^\dagger(z_t,c,v) = u(c)-\lambda z_t+\beta v$ for $\lambda \geq 0$, and let $M_t^\dagger(s_t,y_t,\tilde u) = y_t\tilde u(1)+(1-y_t)\tilde u(0)$.
The aggregators are separated: $W_t^\dagger$ depends on the taste coordinate $z_t$ while $M_t^\dagger$ depends on the belief coordinate $y_t$. The transition maps are coupled, however, since the belief state $y_t$ affects the evolution of the taste state $z_{t+1}$.
Optimism raises future aspirations, which then lower future felicity through the reference stock.
\end{example}

\subsection{Entangled PA}

We conclude our taxonomy by discussing fully entangled preferences.

\begin{example}[Anxiety]
\emph{State system.}
Let ${\cal S} = \{0, 1\}$ where $s_t = 1$ reflects normal conditions and $s_t = 0$ reflects disaster.
Let the preference-memory state $m_t \in [0,1]$ be the subjective probability of a disaster (anxiety), which evolves according to:
\[
m_{t+1} = \Gamma_t^m(x_t, c_t, s_{t+1}) = \alpha\, m_t + (1-\alpha) (1 - s_{t+1}),
\]
for $\alpha \in (0,1)$.

\emph{Aggregators.}
We let $\lambda > 0$ measure the effect of anxiety on utility.
Then we take the time aggregator $W_t^*(m_t, c, v) = u(c) - \lambda m_t^2 + \beta v$ and the risk aggregator $M_t^*(m_t, \tilde{u}) = m_t \tilde{u}(0) + (1 - m_t) \tilde{u}(1)$.
These preferences are entangled because $m_t$ appears in both the time and risk aggregators.
Observed behavior may reflect either pessimistic subjective beliefs through $M_t^*$, or a direct utility cost of anxiety through $W_t^*$.
\end{example}

\begin{example}[Wealth target]
This example is based on a wealth target, where risk sensitivity changes above and below the target.

\emph{State system.} Let $m_t$ be the current wealth, and  $m_{t+1} = \Gamma_t^m(m_t, c_t, s_{t+1}) = (m_t - c_t) R(s_{t+1})$ be the wealth transition.

\emph{Aggregators.}
Let $\eta_\kappa(m)=\bar{\eta} \tanh(\kappa(m - \bar{m}))$ for $\bar{\eta} > 0$ and $\kappa > 0$.
We take the wealth-dependent time aggregator $W_t^*(m_t,c,v)=u(c^\theta m_t^{1-\theta})+\beta v$ for $\theta\in(0,1)$ on $m_t > 0$ and $c > 0$.
The smooth wealth-dependent entropic aggregator is
\[
M_t^*(s_t,m_t,\tilde{u})=
\begin{cases}
-\dfrac{1}{\eta_\kappa(m_t)}
\log \mathbb E_{s_t}[\exp(-\eta_\kappa(m_t)\tilde{u})],
& \eta_\kappa(m_t)\neq 0,\\
\mathbb E_{s_t}[\tilde{u}],
& \eta_\kappa(m_t)=0.
\end{cases}
\]
Risk attitudes change depending on the exogenous wealth benchmark $\bar m$.
Below the target, $\eta_\kappa(m)<0$ drives risk-seeking behavior; above the target, $\eta_\kappa(m)>0$ drives risk-averse behavior. 
Preferences are entangled because wealth affects both the time aggregator and the risk aggregator.
\end{example}

\begin{example}[Wealth-dependent ambiguity]
In this example, wealth changes ambiguity tolerance monotonically.
Let $m_t$ be the wealth with transition map $m_{t+1} = \Gamma_t^m(m_t, c_t, s_{t+1}) = (m_t - c_t) R(s_{t+1})$.
Let $\gamma : {\cal M} \rightarrow \Re_{>0}$ be a continuous and increasing ambiguity tolerance function.
We take the time aggregator $W_t^*(m_t, c, v) = u(c^\theta m_t^{1-\theta}) + \beta v$ for $\theta \in (0,1)$ and the risk aggregator $M_t^*(s_t, m_t, \tilde{u}) = - \gamma(m_t) \ln(\mathbb{E}_{s_t}[\exp{(-\tilde{u}/\gamma(m_t))}])$.
Preferences are entangled because wealth enters current utility through $W_t^*$ and ambiguity tolerance through $M_t^*$.
\end{example}

\begin{example}[Campbell-Cochrane]
This is a finite-state external-habit version of the Campbell--Cochrane model \cite{campbell1999force}. We augment the external-habit specification with a surplus-dependent risk aggregator.

\emph{State system.} Let ${\cal S}$ be a finite set of growth states. Then $x_t=(s_t,m_t)\in{\cal S}_t\times{\cal M}$ where $m_t$ is the log surplus ratio and ${\cal M}\triangleq[\underline m,\overline m]$ is compact.
For the transitions, let $g:{\cal S}\to[\underline g,\overline g]$, where $g(s_{t+1})$ is the log consumption growth in state $s_{t+1}$.
Let $m^\star\in{\cal M}$ be the steady-state log surplus ratio, and let $\lambda:{\cal M}\to\Re_{>0}$ be continuous and decreasing.
Then
\[
\Gamma_t^m(m_t,c_t,s') = \alpha m_t+(1-\alpha)m^\star+\lambda(m_t)g(s'),
\]
where $\alpha \in (0,1)$ and we assume the invariance $\Gamma_t^m({\cal M},{\cal C},{\cal S})\subseteq{\cal M}$ holds for all $t \in [T-1]$.

\emph{Aggregators.}
Suppose ${\cal C}\subseteq[0,\bar c]$, and let $u$ be continuous and increasing on $[0,\bar c\,e^{\overline m}]$. 
Take the time aggregator $W_t^*(m_t,c,v) = u(c e^{m_t})+\beta v$ where the surplus ratio scales the utility of current consumption.
Let $\tau:{\cal M}\to[\underline\tau,\overline\tau]\subset \Re_{>0}$ be continuous and increasing where larger values of $\tau(m_t)$ correspond to greater risk tolerance, and define the risk aggregator
\[
M_t^*(s_t,m_t,\tilde u) = -\tau(m_t) \log \left( \sum_{s'\in{\cal S}} q_t(s'\vert s_t) \exp\!\left(-\frac{\tilde u(s')}{\tau(m_t)}\right) \right).
\]
Risk attitudes depend on the surplus; a higher surplus makes the decision-maker less risk averse.
The surplus $m_t$ affects both current felicity and risk sensitivity, so this is an entangled PA representation rather than a separated SPA representation.
\end{example}

\section{Conclusion}
\label{sec:conclusion}

This paper developed a recursive representation for history-dependent preferences over uncertain finite-horizon consumption streams.
We identified conditions under which the history dependence compresses into a canonical PA-state system.
We further showed that the canonical system is the coarsest among reachable recursive factorizations of the selected compatible utility system.
Under Markov feasibility and standard DP compactness and continuity assumptions on feasible consumption, we obtained a Bellman recursion on the PA state and verified that a PA Bellman selector induces an optimal full-history policy.
Under additional rectangularity and exhaustiveness conditions, the preference-memory component separates into distinct belief and taste coordinates in both the recursive representation and the Bellman recursion.
The unrestricted PA representation permits entangled beliefs and tastes; separation of beliefs and tastes requires additional structure.
Our examples organize PA and SPA preference models into a taxonomy ranging from history-independent to fully entangled.

Our framework can be used in two complementary directions: reductive and constructive. In the reductive direction, it derives a canonical PA-state system by quotienting with respect to an equivalence relation on histories. Under additional rectangularity and exhaustiveness conditions, it derives separated belief and taste coordinates to produce an SPA-state system.
In the constructive direction, it generates recursive preferences from a chosen state-transition system and aggregators (as in the examples).
The corresponding Bellman and verification results apply whenever feasibility is Markov with respect to the selected PA or SPA state and the DP regularity conditions hold.

The full-history, PA, and SPA formulations differ in what the state records, how it evolves, and which coordinates enter the time and risk aggregators.
They nevertheless share a common abstract structure: a state--transition system together with a recursive structure on it. We developed the mathematical machinery at this abstract level and applied it to the full-history, PA, and SPA preferences to obtain our main results.
Natural extensions include forward-looking targets, continuous-time formulations, and infinite-horizon models with limited memory.

\bibliographystyle{plain}
\bibliography{References}

\appendix

\section{Mathematical Background}
\label{app:background}

This section collects supporting results for ease of reference.

\subsection{Continuous Utility Representation}

We establish existence of compatible utility systems by using the following result.

\begin{thm}[Debreu's representation theorem]
\label{thm:debreu}
Let $X$ be a second-countable topological space and $\succsim$ a weak order on $X$ with closed upper and lower contour sets. Then there is a continuous $u:X\to\Re$ such that $x\succsim x'$ if and only if $u(x)\ge u(x')$ \citep{debreu1954representation}.
\end{thm}

\subsection{Quotient Spaces}

A surjection $q:X\to Y$ is a quotient map if $U\subseteq Y$ is open exactly when $q^{-1}(U)$ is open in $X$. An equivalence relation on $X$ is closed if it is closed as a subset of $X\times X$.

\begin{thm}[Universal property of quotient maps]
\label{thm:quotient_universal_property}
Let $q:X\to Y$ be a quotient map and let $Z$ be a topological space. If $f:X\to Z$ is constant on the fibers of $q$, then there exists a unique map $\bar f:Y\to Z$ such that $f=\bar f\circ q$. Moreover, $\bar f$ is continuous if and only if $f$ is continuous \cite[Theorem~22.2]{munkres2000topology}.
\end{thm}

\begin{thm}[Quotients by closed equivalence relations]
\label{thm:closed_quotient}
Let $X$ be a compact Hausdorff space and $R\subseteq X\times X$ be a closed equivalence relation. Then the quotient $X/R$ is compact Hausdorff and the projection $q:X\to X/R$ is a closed map. If, in addition, $X$ is metrizable, then $X/R$ is metrizable \citep{engelking1989general}.
\end{thm}

\begin{thm}[Whitehead's product theorem]
\label{thm:whitehead}
If $q:X\to Y$ is a quotient map and $Z$ is a locally compact Hausdorff space, then $q\times\mathrm{id}_Z:X\times Z\to Y\times Z$ is a quotient map \citep{engelking1989general}.
\end{thm}

\begin{thm}[Compact-to-Hausdorff quotient criterion]
\label{thm:compact_to_hausdorff}
Let $X$ be compact, let $Y$ be Hausdorff, and let $f:X\to Y$ be continuous. Then $f$ is a closed map. In particular, if $f$ is
surjective, then $f$ is a quotient map. If $f$ is bijective, then $f$
is a homeomorphism, see \cite[Secs.~22 and~26]{munkres2000topology}.
\end{thm}

\subsection{Parametric Optimization}

We use this optimization result in the Bellman recursion to show continuity of the optimal value functions and nonemptiness of the argmax correspondence, from which an optimal Markov policy can be selected.

\begin{defn}[Continuous correspondence]
Let $C:\Theta\rightrightarrows X$ be a correspondence.

(i) $C$ is upper hemicontinuous if, for every $\theta\in\Theta$ and every open set $O\subseteq X$ with $C(\theta)\subseteq O$, there is a neighborhood $N$ of $\theta$ such that $C(\theta')\subseteq O$ for all
$\theta'\in N$.

(ii) $C$ is lower hemicontinuous if, for every $\theta\in\Theta$ and every open set $O\subseteq X$ with $C(\theta)\cap O\neq\emptyset$, there is a neighborhood $N$ of $\theta$ such that $C(\theta')\cap O\neq\emptyset$ for all $\theta'\in N$.

(iii) $C$ is continuous if it is both upper and lower hemicontinuous.
\end{defn}

\begin{thm}[Berge's maximum theorem]
\label{thm:berge}
Let $\Theta$ and $X$ be topological spaces, let $f:X\times\Theta\to\Re$ be
continuous, and let $C:\Theta\rightrightarrows X$ be a continuous correspondence with nonempty compact values. Then the value function $f^*(\theta)\triangleq\max_{x\in C(\theta)}f(x,\theta)$ is continuous, and the argmax correspondence $C^*(\theta)\triangleq\{x\in C(\theta):f(x,\theta)=f^*(\theta)\}$ is nonempty, compact-valued, and upper hemicontinuous
\citep{aliprantis2006infinite,berge1963topological}.
\end{thm}

\section{General Recursive Structure}
\label{app:general}

This section develops the machinery for our recursive representation on a general state--transition system. We apply this result twice, first for full-history preferences, then for preferences on a PA state.

\subsection{Abstract Primitives}

\begin{defn}[State--transition system]
\label{defn:state-transition}
A \emph{state--transition system} is a tuple ${\cal G}=(\{{\cal K}_t\}_{t=0}^{T+1},\{G_t\}_{t=0}^{T})$ where:

(i) For each $t\in[T+1]$, ${\cal K}_t$ is a non-empty compact metrizable space.

(ii) For each $t\in[T]$, $G_t:{\cal K}_t\times{\cal C}\times{\cal S}_{t+1}\to{\cal K}_{t+1}$ is a continuous transition map.
\end{defn}
\noindent
We continue to treat the initial state as fixed ${\cal S}_0 = \{\bar{s}_0\}$ and the terminal state ${\cal S}_{T+1} = \{\star\}$ as a dummy, so the terminal transition is always $G_T(k_T, c, \star)$.
Every full-history, PA, or SPA state system has an underlying abstract state--transition system (omitting the physical readout). This system is ${\cal G}^\mathfrak{H} = (\{{\cal H}_t\}, \{\iota_t\})$ for full-history preferences, ${\cal G}^\mathfrak{X} = (\{{\cal X}_t\}, \{\Gamma_t\})$ for PA preferences, and ${\cal G}^{\mathfrak{X}^\dagger} = (\{{\cal X}_t^\dagger\}, \{\Gamma_t^\dagger\})$.

For all $t \in [T]$, the plan spaces are ${\cal F}_t$ and the grand domains are ${\cal D}_t\triangleq{\cal K}_t\times{\cal F}_t$, which now reflect the abstract state space ${\cal K}_t$.

\begin{defn}[Grand preference]
\label{defn:abs_preference}
A \emph{grand preference} on ${\cal G}$ is a tuple ${\cal P}=(\{\succeq_{(t)}\}_{t=0}^T,V_{T+1})$ where:

(i) For each $t \in [T]$, $\succeq_{(t)}$ is a binary relation on ${\cal D}_t$.

(ii) $V_{T+1}:{\cal K}_{T+1}\to\Re$ is continuous.
\end{defn}
\noindent
The conditional preferences are defined in the usual way.

\begin{defn}[Conditional preferences]
\label{defn:abs_conditional_preference}
For each $k_t \in {\cal K}_t$, the conditional preference $\succeq_{k_t}$ is defined on ${\cal F}_t$ by $f \succeq_{k_t} g$ if and only if $(k_t, f) \succeq_{(t)} (k_t, g)$, for all $f, g \in {\cal F}_t$.
The conditional preferences are the collection $\{\succeq_{k_t}\}$.
\end{defn}

\begin{defn}[Compatible utility system]
\label{defn:abs_utility_system}
A utility system is a family ${\bf U}=\{U_t\}_{t=0}^{T+1}$, where $U_t:{\cal K}_t\times{\cal F}_t\to\Re$ is continuous for each $t\in[T]$, and $U_{T+1}:{\cal K}_{T+1}\to\Re$ is continuous.
The utility system ${\bf U}$ is compatible with ${\cal P}$ if: $U_{T+1}=V_{T+1}$; $U_T(k_T,c)=U_{T+1}(G_T(k_T,c,\star))$ for all $k_T\in{\cal K}_T$ and $c\in{\cal C}$; and $U_t(k_t,f)\geq U_t(k_t',g)$ if and only if $(k_t,f)\succeq_{(t)}(k_t',g)$ for all $k_t, k_t' \in {\cal K}_t$, $f, g \in {\cal F}_t$, and $t\in[T]$.
\end{defn}
\noindent
The slices of a compatible utility system ${\bf U}$ represent the conditional preferences.

\begin{defn}[Global time aggregator on an abstract state space]
\label{defn:abstract_time_aggregator}
For $t\in[T-1]$, a \emph{global time aggregator} on ${\cal K}_t$ is a function $W_t:{\cal K}_t\times{\cal C}\times\Re\to\Re$ such that $v\to W_t(k_t,c,v)$ is nondecreasing for every $(k_t,c)\in{\cal K}_t\times{\cal C}$.
\end{defn}

Recall that ${\cal L}$ is the space of bounded random variables $\tilde{u} : {\cal S} \rightarrow \Re$ on ${\cal S}$.
The space ${\cal L}\equiv \Re^{\cal S}$ is equipped with the pointwise order and the supremum norm, and ${\bf v}\in{\cal L}$ denotes the constant vector with value $v$.

\begin{defn}[Global risk aggregator on an abstract state space]
\label{defn:abstract_risk_aggregator}
For $t\in[T-1]$, a \emph{global risk aggregator} on ${\cal K}_t$ is a function $M_t:{\cal K}_t\times{\cal L}\to\Re$ such that: (i) if $\tilde{u}\leq\tilde{v}$, then $M_t(k_t,\tilde{u})\leq M_t(k_t,\tilde{v})$ for every $k_t\in{\cal K}_t$; and (ii) $M_t(k_t,{\bf v})=v$ for all $k_t\in{\cal K}_t$ and $v\in\Re$.
\end{defn}
\noindent
Monotonicity and constant normalization imply that every global risk aggregator satisfies the inequalities:
\[
\min_s \tilde{u}(s) \leq M_t(k_t,\tilde{u}) \leq \max_s \tilde{u}(s),\, \forall k_t\in{\cal K}_t,\, \tilde{u} \in {\cal L},
\]
and so is internal.

Next we define a general recursive structure for the abstract state--transition system.

\begin{defn}[General recursive structure]
Let ${\cal G}=(\{{\cal K}_t\}_{t=0}^{T+1},\{G_t\}_{t=0}^{T})$ be a state--transition system.
A \emph{general recursive structure} on ${\cal G}$ is a tuple ${\cal V}=(\{W_t\}_{t=0}^{T-1},\{M_t\}_{t=0}^{T-1},V_{T+1}^{\cal V})$ where:
\begin{enumerate}
    \item $W_t:{\cal K}_t\times{\cal C}\times\Re\to\Re$ is a jointly continuous global time aggregator for all $t \in [T-1]$.
    \item $M_t:{\cal K}_t \times {\cal L} \to \Re$ is a jointly continuous global risk aggregator for all $t \in [T-1]$.
    \item $V_{T+1}^{\cal V} : {\cal K}_{T+1} \rightarrow \Re$ is continuous.
\end{enumerate}
The pair $({\cal G}, {\cal V})$ generates a utility system ${\bf U}^{\cal V} = \{U_t^{\cal V}\}_{t=0}^{T+1}$ according to $U_{T+1}^{\cal V} = V_{T+1}^{\cal V}$ and:
\begin{align*}
U_T^{\cal V}(k_T,c) = & U_{T+1}^{\cal V}(G_T(k_T,c,\star)),\, \forall k_T \in {\cal K}_T,\, c \in {\cal C},\\
U_t^{\cal V}(k_t,f) = & W_t\big(k_t,c,M_t(k_t,\tilde{u}_{t+1}^{\cal V}(k_t,f))\big),\, \forall k_t \in {\cal K}_t,\, f = (c, f_+) \in {\cal F}_t,\, t\in[T-1],
\end{align*}
where $\tilde{u}_{t+1}^{\cal V}(k_t, f) \in {\cal L}$ is defined by $[\tilde{u}_{t+1}^{\cal V}(k_t, f)](s) \triangleq U_{t+1}^{\cal V}(G_t(k_t, c, s), f_+(s))$ for all $s \in {\cal S}$.
\end{defn}

\begin{defn}[General representation]
The pair $({\cal G}, {\cal V})$ represents ${\cal P}$ if its generated utility system ${\bf U}^{\cal V}$ is compatible with ${\cal P}$.
\end{defn}

\subsection{Abstract Axioms}

We now state Axiom~\ref{axiom:weak_order}--Axiom~\ref{axiom:weak_separability} for ${\cal G}$.

\begin{axiom}[Weak order]\label{axiom:abs_weak_order}
For each $t\in[T]$, $\succeq_{(t)}$ is complete and transitive on ${\cal D}_t$.
\end{axiom}

\begin{axiom}[Continuity]\label{axiom:abs_continuity}
For each $t\in[T]$ and $(k_t, f) \in {\cal K}_t \times {\cal F}_t$, the sets $\{(k_t', g) : (k_t', g) \succeq_{(t)} (k_t, f)\}$ and $\{(k_t', g) : (k_t', g) \preceq_{(t)} (k_t, f)\}$ are closed in ${\cal D}_t$.
\end{axiom}

\begin{axiom}[Dynamic consistency]\label{axiom:abs_dynamic_consistency}
For all $t\in[T-1]$, $k_t\in{\cal K}_t$, and $c \in {\cal C}$, if
\[
(G_t(k_t,c,s),f_+(s))\succeq_{(t+1)}(G_t(k_t,c,s),g_+(s)),\, \forall s\in{\cal S},
\]
then $(k_t,(c,f_+))\succeq_{(t)}(k_t,(c,g_+))$ for all $f=(c,f_+),\,g=(c,g_+)\in{\cal F}_t$.
\end{axiom}
\noindent
Applying Axiom~\ref{axiom:abs_dynamic_consistency} in both directions gives the indifference version. For any $t \in [T-1]$, $k_t \in {\cal K}_t$, and $f, g \in {\cal F}_t$ with $f = (c, f_+)$ and $g = (c, g_+)$, if $(G_t(k_t, c, s), f_+(s)) \sim_{(t+1)} (G_t(k_t, c, s), g_+(s))$ for all $s \in {\cal S}$, then $(k_t, f) \sim_{(t)} (k_t, g)$.

\begin{axiom}[Terminal compatibility]\label{axiom:abs_terminal}
For all $k_T,k_T'\in{\cal K}_T$ and $c,c'\in{\cal C}$, $(k_T,c)\succeq_{(T)}(k_T',c')$ if and only if $V_{T+1}(G_T(k_T,c,\star))\ge V_{T+1}(G_T(k_T',c',\star))$.
\end{axiom}

\begin{axiom}[Compensated consumption monotonicity]\label{axiom:abs_monotonicity}
For $t\in[T-1]$, $k_t\in{\cal K}_t$, and $c>c'$, let $f_+,g_+\in({\cal F}_{t+1})^{\cal S}$. If $(G_t(k_t,c,s),f_+(s))\sim_{(t+1)}(G_t(k_t,c',s),g_+(s))$ for all $s\in{\cal S}$, then $(k_t,(c,f_+))\succ_{(t)}(k_t,(c',g_+))$.
\end{axiom}

\begin{axiom}[Weak separability]\label{axiom:abs_separability}
For $t\in[T-1]$, $k_t\in{\cal K}_t$, and $c,c'\in{\cal C}$, let $f_+,g_+,f_+',g_+'\in({\cal F}_{t+1})^{\cal S}$. If $(G_t(k_t,c,s),f_+(s))\sim_{(t+1)}(G_t(k_t,c',s),f_+'(s))$ and $(G_t(k_t,c,s),g_+(s))\sim_{(t+1)}(G_t(k_t,c',s),g_+'(s))$ for all $s\in{\cal S}$, then $(k_t,(c,f_+))\succeq_{(t)}(k_t,(c,g_+))$ if and only if $(k_t,(c',f_+'))\succeq_{(t)}(k_t,(c',g_+'))$.
\end{axiom}

\subsection{Existence of Compatible Utility System}

We establish the existence of a compatible utility system ${\bf U}$ in the following lemma.

\begin{lem}[Existence of a compatible utility system]
\label{lem:existence}
If Axiom~\ref{axiom:abs_weak_order}, Axiom~\ref{axiom:abs_continuity}, and Axiom~\ref{axiom:abs_terminal} hold, then there exists a compatible utility system for ${\cal P}$.
\end{lem}
\begin{proof}
\emph{Step 1: Grand domains are second countable.}
Each ${\cal K}_t$ is a compact metric space by assumption. Each ${\cal F}_t$ is compact and metrizable by backward induction starting from ${\cal F}_{T+1}=\{\emptyset\}$ using ${\cal F}_t={\cal C}\times({\cal F}_{t+1})^{\cal S}$ and finiteness of ${\cal S}$. Each grand domain ${\cal D}_t={\cal K}_t\times{\cal F}_t$ is then compact metrizable, and so is second countable.

\emph{Step 2: Representation for periods $[T-1]$.}
For each $t\in[T-1]$, $\succeq_{(t)}$ is a continuous weak order by Axiom~\ref{axiom:abs_weak_order} and Axiom~\ref{axiom:abs_continuity}. A continuous weak order on a second countable topological space admits a continuous utility representation by Theorem~\ref{thm:debreu}. Specifically, there exists a continuous function $U_t:{\cal D}_t\to\Re$ such that
\[
U_t(k_t,f)\geq U_t(k_t',g)\iff (k_t,f)\succeq_{(t)}(k_t',g),\, \forall (k_t,f),(k_t',g)\in{\cal D}_t.
\]

\emph{Step 3: Representation for terminal period.}
Set $U_{T+1}\triangleq V_{T+1}$, which is continuous by Definition~\ref{defn:abs_preference}. Then define $U_T:{\cal D}_T\to\Re$ by
\[
U_T(k_T,c) \triangleq V_{T+1}(G_T(k_T,c,\star)) = U_{T+1}(G_T(k_T,c,\star)),\, \forall k_T \in {\cal K}_T,\, c \in {\cal C}.
\]
The function $U_T$ is continuous as the composition of the continuous map $(k_T,c) \to G_T(k_T,c,\star)$ with the continuous utility function $V_{T+1}$. By Axiom~\ref{axiom:abs_terminal}, $(k_T,c)\succeq_{(T)}(k_T',c')$ if and only if $V_{T+1}(G_T(k_T,c,\star))\geq V_{T+1}(G_T(k_T',c',\star))$, so $U_T$ represents $\succeq_{(T)}$ on ${\cal D}_T$.

\emph{Step 4: Compatibility.}
We initialized the recursion with $U_{T+1}=V_{T+1}$ and
\[
U_T(k_T,c)=U_{T+1}(G_T(k_T,c,\star)),\, \forall k_T\in{\cal K}_T,\, c\in{\cal C}.
\]
Each $U_t$ for $t\in[T]$ is continuous and represents $\succeq_{(t)}$. Then, ${\bf U}$ is compatible with ${\cal P}$.
\end{proof}

\subsection{Reduction}

The arguments of the utility functions $\{U_t\}_{t=0}^T$ are life paths in ${\cal K}_t \times {\cal F}_t$ which combine the abstract state and the continuation plan. Plans in ${\cal F}_t$ are complex trees of future consumption choices and are difficult to work with directly, especially for long planning horizons. As the first step of our recursive decomposition, we reduce preferences over plans to preferences over their next-period continuation utilities. Then all payoff-relevant continuation information for the one-step recursion is summarized by its next-period continuation utilities. This step leads to the existence of a general aggregator of current consumption and continuation utilities. In the second step of our decomposition, we refine this general aggregator into separate time and risk aggregators.

Next we define the sets of attainable continuation-utility vectors, where the risk preferences are behaviorally identified.

\begin{defn}[Attainable one-step utility vectors]
\label{defn:abs_attainable_vectors}
Fix a compatible utility system ${\bf U}$.

(i) For all $t\in[T-1]$, $k_t \in {\cal K}_t$, and $f=(c,f_+)\in{\cal F}_t$, the \emph{one-step continuation-utility vector} $\tilde{u}_{t+1}(k_t,f)\in{\cal L}$ is defined by $[\tilde{u}_{t+1}(k_t,f)](s) \triangleq U_{t+1}(G_t(k_t,c,s),f_+(s))$ for all $s \in {\cal S}$.

(ii) For all $t \in [T-1]$, $k_t \in {\cal K}_t$, and $c \in {\cal C}$, let $\Lambda_t^{\bf U}(k_t, c) \triangleq \{\tilde{u}_{t+1}(k_t,(c,f_+)) : f_+\in({\cal F}_{t+1})^{\cal S}\}$, $\Lambda_t^{\bf U}(k_t) \triangleq \bigcup_{c\in{\cal C}}\Lambda_t^{\bf U}(k_t, c)$, and $\Lambda_t^{\bf U} \triangleq \bigcup_{k_t \in {\cal K}_t} \Lambda_t^{\bf U}(k_t)$.
\end{defn}

The next proposition shows that we can evaluate a plan $f = (c, f_+)$ based solely on the current period consumption $c$ and the next period continuation utilities (rather than the entirety of $f_+$).

\begin{prop}[Reduction to one-step continuation utilities]
\label{prop:reduction}
Suppose ${\cal P}$ satisfies Axiom~\ref{axiom:abs_dynamic_consistency}, and ${\bf U}$ is a compatible utility system. Then, for all $t\in[T-1]$ and $k_t \in {\cal K}_t$, there exists an effective aggregator $F_t^0(k_t, \cdot, \cdot)$ defined on the attainable domain $\{(c,\tilde{u}):c\in{\cal C},\tilde{u}\in\Lambda_t^{\bf U}(k_t, c)\}$ such that
\begin{equation}
\label{eq:effective_reduction}
U_t(k_t,f)=F_t^0(k_t, c,\tilde{u}_{t+1}(k_t,f)),\, \forall f=(c,f_+) \in {\cal F}_t.
\end{equation}
Moreover, $F_t^0(k_t, c,\cdot)$ is increasing with respect to pointwise dominance on its attainable domain.
\end{prop}
\begin{proof}
\emph{Step 1: Define effective aggregator $F_t^0$.}
Fix $t\in[T-1]$, $k_t \in {\cal K}_t$, and $c\in{\cal C}$, and let $\tilde{u}\in\Lambda_t^{\bf U}(k_t, c)$. By definition of $\Lambda_t^{\bf U}(k_t, c)$, there exists some $f_+\in({\cal F}_{t+1})^{\cal S}$ such that $\tilde{u}(s)=U_{t+1}(G_t(k_t, c,s),f_+(s))$ for all $s\in{\cal S}$. Define $f=(c,f_+)$, and set
\begin{equation}
\label{eq:reduction}
F_t^0(k_t, c, \tilde{u}) \triangleq U_t(k_t, f).
\end{equation}

\emph{Step 2: Show $F_t^0$ is well-defined.}
The construction of $F_t^0$ in Eq.~\eqref{eq:reduction} is well-defined. Suppose there is another continuation plan $g_+\in({\cal F}_{t+1})^{\cal S}$ such that $\tilde{u}(s)=U_{t+1}(G_t(k_t, c,s),g_+(s))$ for all $s\in{\cal S}$, and let $g=(c,g_+)$. Since ${\bf U}$ is compatible with ${\cal P}$, equality of period $t+1$ utilities implies $(G_t(k_t, c,s),f_+(s))\sim_{(t+1)}(G_t(k_t, c,s),g_+(s))$ for all $s\in{\cal S}$. Then $(k_t, f)\sim_{(t)}(k_t, g)$ by the indifference version of Axiom~\ref{axiom:abs_dynamic_consistency} and $U_t(k_t, f)=U_t(k_t, g)$ since $U_t$ represents $\succeq_{(t)}$. This shows that the construction of $F_t^0(k_t, c,\tilde{u})$ does not depend on the particular continuation plan used to realize $\tilde{u}$.

\emph{Step 3: Verify recursion.}
Let $f=(c,f_+)\in{\cal F}_t$ and $\tilde{u}_{t+1}(k_t, f)\in{\cal L}$. Then $\tilde{u}_{t+1}(k_t, f)\in\Lambda_t^{\bf U}(k_t, c)$, and by construction in Eq.~\eqref{eq:reduction} we have $U_t(k_t, f)=F_t^0(k_t, c,\tilde{u}_{t+1}(k_t, f))$.

\emph{Step 4: Establish monotonicity of $F_t^0$.}
We show $F_t^0$ is monotone in the attainable continuation-utility vector. Fix $c\in{\cal C}$, choose $\tilde{u},\tilde{v}\in\Lambda_t^{\bf U}(k_t, c)$ such that $\tilde{u} \geq \tilde{v}$, and let $f=(c,f_+)$ and $g=(c,g_+)$ realize $\tilde{u}$ and $\tilde{v}$, respectively.
Then $U_{t+1}(G_t(k_t, c,s),f_+(s))\geq U_{t+1}(G_t(k_t, c,s),g_+(s))$ for all $s\in{\cal S}$, which implies
\[
(G_t(k_t, c,s),f_+(s))\succeq_{(t+1)}(G_t(k_t, c,s),g_+(s)),\, \forall s \in {\cal S},
\]
since ${\bf U}$ is compatible with ${\cal P}$. By Axiom~\ref{axiom:abs_dynamic_consistency}, it follows that $(k_t, f)\succeq_{(t)}(k_t, g)$ and therefore
\[
F_t^0(k_t, c,\tilde{u}) = U_t(k_t, f) \geq U_t(k_t, g) = F_t^0(k_t, c,\tilde{v}).
\]
Hence $F_t^0(k_t, c,\cdot)$ is increasing with respect to pointwise dominance on $\Lambda_t^{\bf U}(k_t, c)$.
\end{proof}

\subsection{Induced Risk Comparison}

Eq.~\eqref{eq:effective_reduction} gives a recursion for the utility through the aggregators $\{F_t^0\}_{t=0}^{T-1}$.
However, both current period consumption and continuation utility are coupled arbitrarily inside $\{F_t^0\}_{t=0}^{T-1}$. The next step is to separate the utility of $f = (c, f_+)$ into two channels, which causes distinct time and risk aggregators to emerge.
To extract the time and risk aggregators from $\{F_t^0\}_{t=0}^{T-1}$, we first define an induced risk comparison over attainable continuation utilities in ${\cal L}$.
For each $t\in[T-1]$ and $k_t\in{\cal K}_t$, define
\[
{\cal J}_t^{\bf U}(k_t)
\triangleq
\left\{
(\tilde u,\tilde{v})\in
\Lambda_t^{\bf U}(k_t)\times\Lambda_t^{\bf U}(k_t):
\exists c\in{\cal C}
\text{ such that }
\tilde u,\tilde{v}\in\Lambda_t^{\bf U}(k_t,c)
\right\},
\]
to be the set of jointly attainable pairs in $\Lambda_t^{\bf U}(k_t,c)$.

\begin{defn}
\label{defn:induced}
Let ${\bf U}$ be a compatible utility system.
For each $t\in[T-1]$, $k_t\in{\cal K}_t$, and $(\tilde u,\tilde{v})\in{\cal J}_t^{\bf U}(k_t)$, write $\tilde u\succeq_{k_t}^{\mathrm{risk}}\tilde{v}$ if there exist plans $f=(c,f_+)$ and $g=(c,g_+)$ for the same current consumption $c \in {\cal C}$ that realize $\tilde u$ and $\tilde{v}$, respectively, and satisfy $(k_t,f)\succeq_{(t)}(k_t,g)$.
\end{defn}
\noindent
We next confirm that the risk comparison in Definition~\ref{defn:induced} is independent of the continuation plans used to realize $\tilde{u}$ and $\tilde{v}$ by Axiom~\ref{axiom:abs_dynamic_consistency}.
We specified that $\succeq_{k_t}^{\mathrm{risk}}$ is only defined over jointly attainable pairs in $\Lambda_t^{\bf U}(k_t)$.
Two continuation utility vectors may be jointly realized at more than one current consumption level. Axiom~\ref{axiom:abs_separability} handles this case so $\succeq_{k_t}^{\mathrm{risk}}$ is independent of the current consumption.

\begin{lem}
\label{lem:induced}
Suppose ${\cal P}$ satisfies Axiom~\ref{axiom:abs_dynamic_consistency} and Axiom~\ref{axiom:abs_separability}, and ${\bf U}$ is a compatible utility system. For all $t\in[T-1]$ and $k_t \in {\cal K}_t$, the induced risk comparison $\succeq_{k_t}^{\mathrm{risk}}$ is well-defined on ${\cal J}_t^{\bf U}(k_t)$.
\end{lem}
\begin{proof}
\emph{Step 1: Independence of realizing plans.}
Fix $t\in[T-1]$, $k_t \in {\cal K}_t$, and $\tilde{u},\tilde{v}\in\Lambda_t^{\bf U}(k_t)$.
Suppose $f=(c,f_+)$ and $g=(c,g_+)$ realize $\tilde{u}$ and $\tilde{v}$, respectively, at the same current consumption $c \in {\cal C}$. Suppose also that $f'=(c,f_+')$ and $g'=(c,g_+')$ realize $\tilde{u}$ and $\tilde{v}$, respectively, at the same current consumption $c$. Then we have $U_{t+1}(G_t(k_t, c,s),f_+(s))=U_{t+1}(G_t(k_t, c,s),f_+'(s))=\tilde{u}(s)$ and $U_{t+1}(G_t(k_t, c,s),g_+(s))=U_{t+1}(G_t(k_t, c,s),g_+'(s))=\tilde{v}(s)$ for all $s\in{\cal S}$.

Since ${\bf U}$ is compatible with ${\cal P}$, equality of the period $t+1$ utilities implies
\[
(G_t(k_t, c,s),f_+(s))\sim_{(t+1)}(G_t(k_t, c,s),f_+'(s)),\, \forall s\in{\cal S},
\]
and
\[
(G_t(k_t, c,s),g_+(s))\sim_{(t+1)}(G_t(k_t, c,s),g_+'(s)),\, \forall s\in{\cal S}.
\]
By the indifference version of Axiom~\ref{axiom:abs_dynamic_consistency}, it follows that $(k_t, f)\sim_{(t)}(k_t, f')$ and $(k_t, g)\sim_{(t)}(k_t, g')$. By transitivity of $\succeq_{(t)}$, $(k_t, f)\succeq_{(t)}(k_t, g)$ if and only if $(k_t, f')\succeq_{(t)}(k_t, g')$. This reasoning shows that the induced risk comparison between $\tilde{u}$ and $\tilde{v}$ does not depend on the particular continuation plans used to realize them at the current consumption $c$.

\emph{Step 2: Independence of current consumption.}
Suppose $\tilde{u}$ and $\tilde{v}$ are jointly attainable at both $c, c' \in {\cal C}$. Let $f=(c,f_+)$ and $g=(c,g_+)$ realize $\tilde{u}$ and $\tilde{v}$ at $c$, and let $f'=(c',f_+')$ and $g'=(c',g_+')$ realize them at $c'$.
Since ${\bf U}$ is compatible with ${\cal P}$, equality of the realized continuation utilities implies:
\begin{align*}
(G_t(k_t, c,s),f_+(s)) \sim_{(t+1)} & (G_t(k_t, c',s),f_+'(s)),\, \forall s\in{\cal S},\\
(G_t(k_t, c,s),g_+(s)) \sim_{(t+1)} & (G_t(k_t, c',s),g_+'(s)),\, \forall s\in{\cal S}.
\end{align*}
By Axiom~\ref{axiom:abs_separability}, $(k_t,(c,f_+))\succeq_{(t)}(k_t,(c,g_+))$ if and only if $(k_t,(c',f_+'))\succeq_{(t)}(k_t,(c',g_+'))$. It follows that the induced risk comparison between $\tilde{u}$ and $\tilde{v}$ is independent of the current consumption at which the pair is realized.
\end{proof}

\subsection{Separation}

The risk comparison $\succeq_{k_t}^{\mathrm{risk}}$ need not be complete on all of $\Lambda_t^{\bf U}(k_t)$.
The following Assumption~\ref{assu:abs_CE_richness} generalizes our richness assumption for the abstract setting.
It supplies a continuous, normalized, internal, and monotone certainty-equivalent completion of the partially identified $\succeq_{k_t}^{\mathrm{risk}}$ to the full effective domain.
Let
\[
D_t^M \triangleq \left\{ (k_t,\tilde{u}) \in {\cal K}_t \times {\cal L} : \tilde{u}\in\Lambda_t^{\bf U}(k_t) \right\}
\]
be the effective domain of the period $t \in [T-1]$ risk aggregator.

\begin{assumption}[Certainty-equivalent richness for ${\cal G}$]
\label{assu:abs_CE_richness}
Fix a compatible utility system ${\bf U}$. For each $t\in[T-1]$ there is a jointly continuous functional $M_t^0 : D_t^M \rightarrow \Re$, such that for each $k_t\in{\cal K}_t$:

(i) \textnormal{(Certainty equivalent)} $M_t^0(k_t,\cdot)$ extends and represents $\succeq_{k_t}^{\mathrm{risk}}$ on jointly attainable pairs: for any plans $f=(c,f_+)$ and $g=(c,g_+)$ with the same current consumption, $f\succeq_{k_t}g$ if and only if $M_t^0(k_t,\tilde{u}_{t+1}(k_t,f))\ge M_t^0(k_t,\tilde{u}_{t+1}(k_t,g))$, and $M_t^0(k_t,{\bf v})=v$ whenever ${\bf v} \in \Lambda_t^{\bf U}(k_t)$.

(ii) \textnormal{(Internality)}
$\min_{s \in {\cal S}}\tilde{u}(s)\le M_t^0(k_t,\tilde{u})\le\max_{s \in {\cal S}}\tilde{u}(s)$ for all $\tilde{u} \in \Lambda_t^{\bf U}(k_t)$.

(iii) \textnormal{(Deterministic solvability)} Fix $f=(c,f_+)\in{\cal F}_t$ and let $v=M_t^0(k_t,\tilde{u}_{t+1}(k_t,f))$, then there exists $(c,g_+)\in{\cal F}_t$ such that ${\bf v} = \tilde{u}_{t+1}(k_t, (c,g_+))$.

(iv) \textnormal{(Monotone completion)} If $\tilde{u},\tilde{v}\in\Lambda_t^{\bf U}(k_t)$ and $\tilde{u}\leq\tilde{v}$, then $M_t^0(k_t,\tilde{u}) \leq M_t^0(k_t,\tilde{v})$.
\end{assumption}

We now separate the general aggregators $\{F_t^0\}_{t=0}^{T-1}$ into time and risk aggregators in the proposition below.
Assumption~\ref{assu:abs_CE_richness} supplies the effective risk aggregators $\{M_t^0\}_{t=0}^{T-1}$.
For fixed $t \in [T-1]$, $k_t \in {\cal K}_t$, and $c\in{\cal C}$, let ${\cal I}_t^{\bf U}(k_t, c) \triangleq M_t^0(k_t,\Lambda_t^{\bf U}(k_t, c))$ be the induced effective scalar domain of attainable certainty equivalents. Then, let
\[
D_t^W\triangleq\left\{(k_t,c,v) \in {\cal K}_t \times {\cal C} \times \Re : v\in{\cal I}_t^{\bf U}(k_t,c) \right\},
\]
be the effective domain of the period $t \in [T-1]$ time aggregator.

\begin{prop}[Separation of time and risk]
\label{prop:separation}
Suppose ${\cal P}$ satisfies Axiom~\ref{axiom:abs_dynamic_consistency}, Axiom~\ref{axiom:abs_monotonicity}, and Axiom~\ref{axiom:abs_separability}. Fix a compatible utility system ${\bf U}$ satisfying Assumption~\ref{assu:abs_CE_richness}. Then, for all $t\in[T-1]$, there exists an effective time aggregator $W_t^0 : D_t^W \to \Re$ such that, for all $f=(c,f_+)\in{\cal F}_t$,
\[
U_t(k_t, f)=W_t^0(k_t, c,M_t^0(k_t,\tilde{u}_{t+1}(k_t, f))),\, \forall k_t \in {\cal K}_t.
\]
Moreover, for all $k_t \in {\cal K}_t$ and $c \in {\cal C}$, if $v \geq v'$ for $v,v'\in{\cal I}_t^{\bf U}(k_t, c)$, then $W_t^0(k_t,c,v)\geq W_t^0(k_t,c,v')$.
For all $k_t \in {\cal K}_t$, if $c>c'$ and $v \in {\cal I}_t^{\bf U}(k_t, c) \cap {\cal I}_t^{\bf U}(k_t, c')$, then $W_t^0(k_t,c,v)>W_t^0(k_t,c',v)$.
\end{prop}
\begin{proof}
\emph{Step 1: Fix the effective risk aggregator.} By Assumption~\ref{assu:abs_CE_richness}, let $M_t^0:D_t^M\to\Re$ be the jointly continuous certainty-equivalent completion. For any fixed $k_t\in{\cal K}_t$, its restriction $M_t^0(k_t,\cdot)$ represents the completed induced risk comparison on $\Lambda_t^{\bf U}(k_t)$.

\emph{Step 2: Construct $W_t^0$.}
We now define $W_t^0$ on $D_t^W$. Fix $k_t \in {\cal K}_t$ and $(c, v) \in {\cal C} \times \Re$ with $v\in{\cal I}_t^{\bf U}(k_t, c)$. Then, there exists $\tilde{u}\in\Lambda_t^{\bf U}(k_t, c)$ such that $v=M_t^0(k_t,\tilde{u})$. By deterministic solvability, ${\bf v}\in\Lambda_t^{\bf U}(k_t, c)$ so there exists $a_+\in({\cal F}_{t+1})^{\cal S}$ such that
\[
U_{t+1}(G_t(k_t, c,s),a_+(s))=v,\, \forall s\in{\cal S}.
\]
Let $a^{c,v}=(c,a_+) \in {\cal F}_t$ be the plan that realizes $(c, {\bf v})$, and define
\begin{equation}
\label{eq:time_aggregator_construction}
    W_t^0(k_t,c,v)\triangleq U_t(k_t,a^{c,v}).
\end{equation}
The construction in Eq.~\eqref{eq:time_aggregator_construction} is well-defined, because it is independent of the plan used to realize ${\bf v}$. If $a^{c,v}=(c,a_+)$ and $b^{c,v}=(c,b_+)$ both realize $(c,{\bf v})$, then
\[
U_{t+1}(G_t(k_t, c,s),a_+(s))=U_{t+1}(G_t(k_t, c,s),b_+(s)),\, \forall s \in {\cal S}.
\]
Compatibility of ${\bf U}$ and the indifference version of Axiom~\ref{axiom:abs_dynamic_consistency} then imply $(k_t,a^{c,v})\sim_{(t)}(k_t,b^{c,v})$, which means $U_t(k_t,a^{c,v})=U_t(k_t,b^{c,v})$.

\emph{Step 3: Show monotonicity in continuation value.}
Fix $c\in{\cal C}$ and let $v\geq v'$ for $v,v'\in{\cal I}_t^{\bf U}(k_t, c)$.
By deterministic solvability, we can choose plans $a^{c,v}=(c,a_+)$ and $a^{c,v'}=(c,a_+')$ that realize $(c,{\bf v})$ and $(c,{\bf v'})$, respectively, so that
\[
U_{t+1}(G_t(k_t, c,s),a_+(s)) = v \geq v' = U_{t+1}(G_t(k_t, c,s),a_+'(s)),\, \forall s\in{\cal S}.
\]
By Axiom~\ref{axiom:abs_dynamic_consistency} it follows that $(k_t,a^{c,v})\succeq_{(t)}(k_t,a^{c,v'})$ and so $W_t^0(k_t,c,v)\geq W_t^0(k_t,c,v')$ by Eq.~\eqref{eq:time_aggregator_construction}.

\emph{Step 4: Show compensated monotonicity.}
Fix $c>c'$ and suppose $v\in{\cal I}_t^{\bf U}(k_t, c)\cap{\cal I}_t^{\bf U}(k_t, c')$. 
Then, choose plans $a^{c,v}=(c,a_+)$ and $a^{c',v}=(c',a_+')$ that realize $(c,{\bf v})$ and $(c',{\bf v})$, respectively. By compatibility of ${\bf U}$, we have
\[
U_{t+1}(G_t(k_t, c,s),a_+(s)) = U_{t+1}(G_t(k_t, c',s),a_+'(s)) = v,\, \forall s \in {\cal S}.
\]
By Axiom~\ref{axiom:abs_monotonicity}, we have $(k_t,a^{c,v})\succ_{(t)}(k_t,a^{c',v})$ (since the continuation utilities are equal) and therefore $W_t^0(k_t, c,v)>W_t^0(k_t, c',v)$ by Eq.~\eqref{eq:time_aggregator_construction}.

\emph{Step 5: Verify separation.}
Let $f=(c,f_+)\in{\cal F}_t$, set $\tilde{u} = \tilde{u}_{t+1}(k_t, f) \in \Lambda_t^{\bf U}(k_t, c)$, and let $v=M_t^0(k_t,\tilde{u})$.
By deterministic solvability, ${\bf v}\in\Lambda_t^{\bf U}(k_t, c)$, so there is a plan $a^{c,v} \in {\cal F}_t$ realizing $(c,{\bf v})$. Since $M_t^0(k_t,\tilde u)=v=M_t^0(k_t,{\bf v})$, and since $f=(c,f_+)$ and $a^{c,v}=(c,a_+)$ have the same current consumption $c$, the certainty-equivalent representation in Assumption~\ref{assu:abs_CE_richness}(i) implies $(k_t,f)\sim_{(t)}(k_t,a^{c,v})$.
Therefore, $U_t(k_t,f)=U_t(k_t,a^{c,v})=W_t^0(k_t,c,v)$.
Substituting $v=M_t^0(k_t,\tilde{u}_{t+1}(k_t, f))$ yields the conclusion $U_t(k_t, f)=W_t^0(k_t, c,M_t^0(k_t,\tilde{u}_{t+1}(k_t, f)))$.
\end{proof}
\noindent
The aggregators $\{M_t^0\}_{t=0}^{T-1}$ represent the induced risk comparison on the effective domain, and $\{W_t^0\}_{t=0}^{T-1}$ represent deterministic tradeoffs between current consumption and certainty-equivalent continuation utility.

Joint continuity of $M_t^0$ on $D_t^M$ is part of Assumption~\ref{assu:abs_CE_richness}. We establish joint continuity of $W_t^0$ on $D_t^W$ in the following lemma.

\begin{lem}[Joint continuity of the effective time aggregator]
\label{lem:effective_time_continuity}
Suppose Assumption~\ref{assu:abs_CE_richness} holds.
For each $t\in[T-1]$, $D_t^W$ is compact and the effective time aggregator $W_t^0:D_t^W\to\Re$ from Proposition~\ref{prop:separation} is jointly continuous.
\end{lem}

\begin{proof}
Fix $t \in [T-1]$ and define ${\cal Q}_t \triangleq {\cal K}_t \times {\cal C} \times ({\cal F}_{t+1})^{\cal S}$, which is compact because ${\cal K}_t$, ${\cal C}$, and ${\cal F}_{t+1}$ are compact and ${\cal S}$ is finite.
Then define the mapping $\Psi_t : {\cal K}_t\times{\cal C}\times({\cal F}_{t+1})^{\cal S} \to{\cal L}$ by
\[
[\Psi_t(k_t,c,f_+)](s) \triangleq U_{t+1}(G_t(k_t,c,s),f_+(s)),\, \forall s\in{\cal S}.
\]
Since ${\cal S}$ is finite and $G_t$ and $U_{t+1}$ are continuous, $\Psi_t$ is continuous.

Next, define $p_t:{\cal Q}_t\to D_t^W$ by
\[
p_t(k_t, c,f_+) = (k_t,c,M_t^0(k_t,\Psi_t(k_t, c, f_+))).
\]
Since $\Psi_t$ is continuous and $M_t^0$ is jointly continuous, the map $p_t$ is continuous. It is also onto $D_t^W$ by the definition of ${\cal I}_t^{\bf U}(k_t, c)$, so $D_t^W=p_t({\cal Q}_t)$ is compact. Since $D_t^W \subseteq {\cal K}_t\times{\cal C}\times\Re$, it is additionally Hausdorff. Then $p_t$ is a continuous surjection from the compact space ${\cal Q}_t$ onto the Hausdorff space $D_t^W$, so it is a quotient map by Theorem~\ref{thm:compact_to_hausdorff}.

Now define $\mathscr U_t:{\cal Q}_t\to\Re$ by $\mathscr U_t(k_t, c,f_+) \triangleq U_t(k_t,(c,f_+))$, which is continuous by continuity of $U_t$.
Suppose $p_t(k_t,c,f_+) = p_t(k_t',c',g_+)$, then $k_t=k_t'$, $c=c'$, and $M_t^0(k_t,\Psi_t(k_t,c,f_+))=M_t^0(k_t,\Psi_t(k_t,c,g_+))$.
The two continuation-utility vectors $\Psi_t(k_t,c,f_+)$ and $\Psi_t(k_t,c,g_+)$ are jointly attainable at the same current consumption $c$. Since $M_t^0(k_t,\cdot)$ represents $\succeq_{k_t}^{\mathrm{risk}}$ on jointly attainable pairs, equality of the certainty equivalents implies $(k_t,(c,f_+))\sim_{(t)}(k_t,(c,g_+))$.
Then $\mathscr U_t(k_t,c,f_+) = \mathscr U_t(k_t,c,g_+)$ by compatibility of ${\bf U}$, which shows that $\mathscr U_t$ is constant on the fibers of $p_t$.

Because $p_t$ is a quotient map, by Theorem~\ref{thm:quotient_universal_property} there is a unique jointly continuous function $\overline{W}_t^0:D_t^W\to\Re$ such that $\mathscr U_t = \overline{W}_t^0\circ p_t$, explicitly
\[
U_t(k_t,(c,f_+)) = \overline{W}_t^0 (k_t, c, M_t^0(k_t,\Psi_t(k_t,c,f_+))).
\]
Proposition~\ref{prop:separation} gives the same recursion with $W_t^0$. Since $p_t$ is onto $D_t^W$, the factor through $p_t$ is unique, and $W_t^0=\overline{W}_t^0$ on $D_t^W$.
We conclude that $W_t^0$ is jointly continuous.
\end{proof}

\subsection{Continuous Monotone Extensions}

The next step is to find continuous and monotone global extensions of the behaviorally identified effective aggregators.
To proceed, Lemma~\ref{lem:graph_compact} below establishes compactness of the effective domains.

\begin{lem}
\label{lem:graph_compact}
Let ${\cal K}$ and ${\cal Q}$ be compact metric spaces and $g:{\cal K}\times{\cal Q}\to\Re^{N}$ be a continuous function. Then, the set $\{(k,g(k,q)):k\in{\cal K},\,q\in{\cal Q}\} \subseteq {\cal K}\times\Re^{N}$ is compact.
\end{lem}
\begin{proof}
Define $F:{\cal K}\times{\cal Q}\to{\cal K}\times\Re^N$ by $F(k,q)=(k,g(k,q))$.
The map $F$ is continuous, and ${\cal K}\times{\cal Q}$ is compact. Then
\[
F({\cal K}\times{\cal Q}) = \{(k,g(k,q)):k\in{\cal K},\,q\in{\cal Q}\}
\]
is compact as the continuous image of a compact set.
\end{proof}

We select sets ${\cal K}={\cal K}_t$ and ${\cal Q}={\cal C}\times({\cal F}_{t+1})^{\cal S}$, and the mapping $g : {\cal K} \times {\cal Q} \rightarrow {\cal L}$ defined by
\[
[g(k_t,c,f_+)](s) \triangleq U_{t+1}(G_t(k_t,c,s),f_+(s)),\, \forall s \in {\cal S},
\]
which is continuous because $G_t$ and $U_{t+1}$ are continuous and ${\cal S}$ is finite.
Applying Lemma~\ref{lem:graph_compact} with these selections shows that
\[
D_t^M=\{(k_t,g(k_t,c,f_+)):k_t\in{\cal K}_t,\ c\in{\cal C},\ f_+\in({\cal F}_{t+1})^{\cal S}\}
\]
is compact. Similarly, the map $(k_t,c,f_+)\to(k_t,c,M_t^0(k_t,g(k_t,c,f_+))) : {\cal K}_t\times{\cal C}\times({\cal F}_{t+1})^{\cal S} \rightarrow {\cal K}_t\times{\cal C}\times\Re$ is continuous, and so its image $D_t^W$ is compact.

\begin{defn}[Fiber-isotone]
\label{defn:fiber-isotone}
A function $\phi:D\subseteq{\cal K}\times\Re^N\to\Re$ is \emph{fiber-isotone} if $\phi(k,p)\leq\phi(k,p')$ whenever $(k,p),(k,p')\in D$ and $p\leq p'$ coordinatewise.
\end{defn}

The next result provides monotone continuous global extensions.

\begin{thm}[Fiberwise monotone extension]
\label{thm:monotone_extension}
Let ${\cal K}$ be a compact metric space, let $P=\Re^N$ with the coordinatewise order, and let $D\subseteq{\cal K}\times P$ be nonempty and compact. If $\phi:D\to\Re$ is continuous and fiber-isotone, then there exists a continuous fiber-isotone extension $\Phi:{\cal K}\times P\to\Re$ satisfying
\[
\min_{(k,p) \in D} \phi(k,p) \leq \Phi \leq \max_{(k,p) \in D} \phi(k,p).
\]
\end{thm}

\begin{proof}
Let $R=[\alpha,\beta]^N\subset P$ be a compact box, such that the interior of $R$ contains the projection of $D$ onto $P$. Let $E={\cal K}\times R$ be equipped with the preorder $\preceq$ defined by the product of the equality order on ${\cal K}$ and the coordinatewise order on $R$:
\[
(k,p)\preceq(k',p') \iff k=k', p\leq p'.
\]
Since $E$ is compact and $\preceq$ is closed, $(E,\preceq)$ is normally preordered \cite[Theorem~2.4]{minguzzi2013normally}.

Let $a = \min_D\phi$ and $b = \max_D\phi$.
If $a=b$, then we can take the constant function $\Phi\equiv a$ as the global extension. Otherwise, suppose $a<b$, and define $\bar\phi:D\to[0,1]$ by $\bar\phi=(\phi-a)/(b-a)$. The function $\bar\phi$ is continuous and isotone on the compact subspace $D\subseteq E$. By \cite[Theorem~3.4]{minguzzi2013normally}, there exists a continuous isotone extension $\bar\Phi:E\to[0,1]$. Define $\widehat\Phi:E\to[a,b]$ by $\widehat\Phi=a+(b-a)\bar\Phi$. Then $\widehat\Phi$ is continuous, isotone, agrees with $\phi$ on $D$, and satisfies $a\leq\widehat\Phi\leq b$.

To conclude, define the truncation operator $r:P\to R$ by $r(p)_n=\max\{\alpha,\min\{\beta,p_n\}\}$ for $n=1,\ldots,N$. The map $r$ is continuous and isotone, and $r(x)=x$ for all $x \in R$. Then let $\Phi:{\cal K}\times P\to\Re$ be defined by $\Phi(k,p)=\widehat\Phi(k,r(p))$. The function $\Phi$ is the desired extension: it is continuous, fiber-isotone, satisfies $a\leq\Phi\leq b$, and agrees with $\phi$ on $D$.
\end{proof}

Now we use Theorem~\ref{thm:monotone_extension} to obtain global extensions of the effective time and risk aggregators.

\begin{cor}[Global risk aggregator]
\label{cor:risk_extension}
Suppose $M_t^0:D_t^M\to\Re$ is jointly continuous, fiber-isotone, internal, and normalized. Then there is a jointly continuous global risk aggregator $M_t:{\cal K}_t\times{\cal L}\to\Re$ that agrees with $M_t^0$ on $D_t^M$.
\end{cor}
\begin{proof}
\emph{Step 1: Bound effective continuation utilities and add constant vectors.}
The set $D_t^M$ is nonempty since ${\cal K}_t$, ${\cal C}$, and ${\cal F}_{t+1}$ are nonempty. It is compact by Lemma~\ref{lem:graph_compact}. Since ${\cal S}$ is finite, the upper and lower bounds
\[
\underline U\triangleq\min\{\tilde{u}(s):(k_t,\tilde{u})\in D_t^M, s\in{\cal S}\},\, \overline U\triangleq\max\{\tilde{u}(s):(k_t,\tilde{u})\in D_t^M, s\in{\cal S}\},
\]
are well-defined.
Then let ${\cal B}\triangleq[\underline U,\overline U]^{\cal S}$, so that $D_t^M\subseteq{\cal K}_t\times{\cal B}$, and define
\[
C_t^M\triangleq\{(k_t,{\bf v}) \in {\cal K}_t \times {\cal L} : k_t\in{\cal K}_t, v\in[\underline U,\overline U]\},
\]
which is compact.

\emph{Step 2: Extend $M_t^0$ to compact augmented domain.}
Define $\phi:D_t^M\cup C_t^M\to\Re$ by $\phi(k_t,\tilde{u})=M_t^0(k_t,\tilde{u})$ on $D_t^M$ and $\phi(k_t,{\bf v})=v$ on $C_t^M$.
Both definitions give continuous functions, and they coincide on $D_t^M \cap C_t^M$ by normalization of $M_t^0$ on its effective domain. Since $D_t^M$ and $C_t^M$ are closed and the restriction of $\phi$ to each set is continuous, $\phi$ is continuous on $D_t^M\cup C_t^M$ by the pasting lemma. Internality of $M_t^0$ implies $\underline U\leq\phi\leq\overline U$, and the constants ${\bf \underline U}$ and ${\bf \overline U}$ in $C_t^M$ show that the utilities $\underline U$ and $\overline U$ are attained.

\emph{Step 3: Verify fiber isotonicity on the augmented domain.}
Let $(k_t, \tilde{u}), (k_t, \tilde{u}') \in D_t^M\cup C_t^M$.
If both $(k_t, \tilde{u}), (k_t, \tilde{u}') \in D_t^M$, the conclusion follows from fiber isotonicity of $M_t^0$. If both $(k_t, \tilde{u}), (k_t, \tilde{u}') \in C_t^M$, it follows from the usual order on constants. For the cross cases, if $\tilde{u}\leq{\bf v}$, then $M_t^0(k_t,\tilde{u})\leq\max_s\tilde{u}(s)\leq v$ by internality; if ${\bf v}\leq\tilde{u}$, then similarly $v\leq\min_s\tilde{u}(s)\leq M_t^0(k_t,\tilde{u})$ by internality. This reasoning shows that $\phi$ is fiber-isotone on $D_t^M\cup C_t^M$.

\emph{Step 4: Apply monotone extension theorem.}
By Theorem~\ref{thm:monotone_extension}, there is a continuous fiber-isotone extension $\Phi:{\cal K}_t\times{\cal L}\to[\underline U,\overline U]$ that agrees with $\phi$ on $D_t^M\cup C_t^M$. For $v\geq\overline U$, isotonicity and boundedness give $\overline U=\Phi(k_t,{\bf \overline U})\leq\Phi(k_t,{\bf v})\leq\overline U$, so $\Phi(k_t,{\bf v})=\overline U$. Similarly, $\Phi(k_t,{\bf v})=\underline U$ whenever $v\leq\underline U$.

\emph{Step 5: Normalization.}
Define the function $\rho:\Re\to\Re$ by $\rho(v)=(v-\overline U)_+-(\underline U-v)_+$. This function is continuous and nondecreasing, equals zero on $[\underline U,\overline U]$, equals $v-\overline U$ for $v\geq\overline U$, and equals $v-\underline U$ for $v\leq\underline U$. Then define $M_t : {\cal K}_t \times {\cal L} \rightarrow \Re$ via:
\[
M_t(k_t,\tilde{u})\triangleq\Phi(k_t,\tilde{u})+\frac{1}{|{\cal S}|}\sum_{s\in{\cal S}}\rho(\tilde{u}(s)).
\]

\emph{Step 6: Verify $M_t$ is a risk aggregator.}
The function $M_t$ is jointly continuous as the sum of continuous functions. It is isotone under pointwise dominance because both $\Phi(k_t,\cdot)$ and $\rho(\cdot)$ are isotone.
For every constant vector ${\bf v}$, if $v\in[\underline U,\overline U]$, then $\Phi(k_t,{\bf v})=v$ and $\rho(v)=0$. If $v\geq\overline U$, then $\Phi(k_t,{\bf v})=\overline U$ and $\rho(v)=v-\overline U$. If $v\leq\underline U$, then $\Phi(k_t,{\bf v})=\underline U$ and $\rho(v)=v-\underline U$. Then, $M_t(k_t,{\bf v})=v$ for every $v\in\Re$.

Finally, if $(k_t,\tilde u)\in D_t^M$, then $\tilde u\in{\cal B}$, so $\rho(\tilde u(s))=0$ for every $s\in{\cal S}$. Since $\Phi=\phi=M_t^0$ on $D_t^M$, we have $M_t=M_t^0$ on $D_t^M$. We conclude that $M_t$ is a jointly continuous, normalized risk aggregator extending $M_t^0$.
\end{proof}

\begin{cor}[Global time-aggregator extension]
\label{cor:time_extension}
(i) The effective domain $D_t^W$ is compact.

(ii) Suppose $W_t^0:D_t^W\to\Re$ is jointly continuous and fiber-isotone in its continuation value: whenever $(k_t,c,v),(k_t,c,v')\in D_t^W$ and $v \leq v'$, we have $W_t^0(k_t,c,v)\leq W_t^0(k_t,c,v')$.
Then there exists a jointly continuous time aggregator $W_t:{\cal K}_t\times{\cal C}\times\Re\to\Re$ that agrees with $W_t^0$ on $D_t^W$. If $W_t^0$ satisfies compensated strict monotonicity in current consumption on $D_t^W$, then $W_t$ preserves that property on $D_t^W$.
\end{cor}

\begin{proof}
Set $\widehat{\cal K}_t\triangleq{\cal K}_t\times{\cal C}$. Since ${\cal K}_t$ and ${\cal C}$ are compact metric, $\widehat{\cal K}_t$ is compact metric, and the joint effective domain $D_t^W$ is a compact subset of $\widehat{\cal K}_t\times\Re$.

Under the assumptions on $W_t^0:D_t^W\to\Re$, Theorem~\ref{thm:monotone_extension}, applied with ${\cal K}=\widehat{\cal K}_t$, $P=\Re$, and $D=D_t^W$, gives a jointly continuous fiber-isotone extension $\overline{W}_t:
\widehat{\cal K}_t\times\Re\to\Re$ that agrees with $W_t^0$ on $D_t^W$. Define $W_t(k_t,c,v) = \overline{W}_t((k_t,c),v)$, then $W_t$ is jointly continuous, nondecreasing in $v$, and agrees with $W_t^0$ on $D_t^W$.
Finally, any compensated strict comparison in current consumption is preserved on $D_t^W$ where $W_t=W_t^0$.
\end{proof}

\subsection{Abstract Recursive Representation}

Our abstract representation result for the general state--transition system ${\cal G}$ is next.

\begin{thm}[Abstract recursive representation]
\label{thm:abstract_recursive}
Let ${\cal G}$ be a state--transition system, let ${\cal P}$ be a grand preference on ${\cal G}$, and let ${\bf U}$ be a compatible utility system for ${\cal P}$. Suppose that ${\cal P}$ satisfies Axioms~\ref{axiom:abs_weak_order}--\ref{axiom:abs_separability} and that ${\bf U}$ satisfies Assumption~\ref{assu:abs_CE_richness}. Then there exists a general recursive structure ${\cal V}$ on ${\cal G}$ whose generated utility system equals ${\bf U}$. Consequently, $({\cal G},{\cal V})$ represents ${\cal P}$.
\end{thm}

\begin{proof}
\emph{Step 1: Define recursive structure.}
By Proposition~\ref{prop:reduction}, each $U_t$ just depends on current consumption and the continuation-utility vector.
The induced risk comparison is well defined by Lemma~\ref{lem:induced}.
Assumption~\ref{assu:abs_CE_richness} supplies $M_t^0$.
Proposition~\ref{prop:separation} constructs $W_t^0$, and Lemma~\ref{lem:effective_time_continuity} gives joint continuity of $W_t^0$.
Corollaries~\ref{cor:risk_extension} and \ref{cor:time_extension} extend $M_t^0$ and $W_t^0$ to global continuous monotone aggregators $M_t$ and $W_t$.
We then set ${\cal V} = (\{W_t\}_{t=0}^{T-1}, \{M_t\}_{t=0}^{T-1}, V_{T+1}^{\cal V})$ where $V_{T+1}^{\cal V}\triangleq U_{T+1}$.

\emph{Step 2: Utility equivalence.}
We show that the utility system generated by $({\cal G},{\cal V})$ equals the fixed compatible system ${\bf U}$.

\indpart{Step 2A: Terminal step}
By construction, $U_{T+1}^{\cal V}=U_{T+1}$. Then,
\[
U_T^{\cal V}(k_T,c)=U_{T+1}(G_T(k_T,c,\star))=U_T(k_T,c)
\]
by compatibility of ${\bf U}$. Since $U_T$ represents $\succeq_{(T)}$ by Axiom~\ref{axiom:abs_terminal}, so does $U_T^{\cal V}$.

\indpart{Step 2B: Backward induction step}
For $t \in [T-1]$, suppose $U_{t+1}^{\cal V}=U_{t+1}$ on ${\cal D}_{t+1}$. Fix $(k_t,f)\in{\cal D}_t$ with $f=(c,f_+)$. By the inductive
hypothesis, we have
\[
[\tilde{u}_{t+1}^{\cal V}(k_t,f)](s)=U_{t+1}^{\cal V}(G_t(k_t,c,s),f_+(s))=U_{t+1}(G_t(k_t,c,s),f_+(s))=[\tilde{u}_{t+1}(k_t,f)](s),
\]
for all $s\in{\cal S}$. Then, $\tilde{u}_{t+1}^{\cal V}(k_t,f) \in \Lambda_t^{\bf U}(k_t,c)$ is realized by $f_+$ at consumption $c$.
Since $(k_t,\tilde u_{t+1}^{\cal V}(k_t,f))\in D_t^M$, we have $M_t\big(k_t,\tilde{u}_{t+1}^{\cal V}(k_t,f)\big)=M_t^0\big(k_t,\tilde{u}_{t+1}(k_t,f)\big) \triangleq v$ by Corollary~\ref{cor:risk_extension}.
Then, $v\in{\cal I}_t^{\bf U}(k_t,c)$ by definition, so $(k_t,c,v)\in D_t^W$. By Corollary~\ref{cor:time_extension}, $W_t(k_t,c,v)=W_t^0(k_t,c,v)$ on $D_t^W$. Combining these equalities with the decomposition of Proposition~\ref{prop:separation}, we obtain
\[
U_t^{\cal V}(k_t,f)=W_t\big(k_t,c,M_t(k_t,\tilde{u}_{t+1}^{\cal V}(k_t,f))\big)=W_t^0\big(k_t,c,M_t^0(k_t,\tilde{u}_{t+1}(k_t,f))\big)=U_t(k_t,f).
\]
It follows that $U_t^{\cal V}=U_t$ on ${\cal D}_t$. Since $U_t$ represents $\succeq_{(t)}$ by compatibility of ${\bf U}$, so does $U_t^{\cal V}$.

\indpart{Step 2C: Conclusion} By backward induction, $U_{T+1}^{\cal V}=U_{T+1}$ on ${\cal K}_{T+1}$ and $U_t^{\cal V}=U_t$ on ${\cal D}_t$ for all $t \in [T]$.
We conclude that the utility system generated by $({\cal G},{\cal V})$ equals ${\bf U}$, and so $({\cal G},{\cal V})$ represents ${\cal P}$.
\end{proof}

We conclude this section by verifying that the evaluation of every plan is independent of how the effective aggregators are extended outside their effective domains.

\begin{prop}[Extension invariance]
\label{prop:extension_invariance}
Fix a compatible utility system ${\bf U}$, and effective aggregators $\{W_t^0\}_{t=0}^{T-1}$ and $\{M_t^0\}_{t=0}^{T-1}$ satisfying the decomposition in Proposition~\ref{prop:separation}. Let
\[
{\cal V}=(\{W_t\}_{t=0}^{T-1},\{M_t\}_{t=0}^{T-1},V_{T+1}^{\cal V}),\, \overline{\cal V}=(\{\overline{W}_t\}_{t=0}^{T-1},\{\overline{M}_t\}_{t=0}^{T-1},V_{T+1}^{\overline{\cal V}})
\]
be two general recursive structures on ${\cal G}$ where $V_{T+1}^{\cal V} = V_{T+1}^{\overline{\cal V}} = U_{T+1}$ and $W_t=\overline{W}_t=W_t^0$ on $D_t^W$ and $M_t=\overline{M}_t=M_t^0$ on $D_t^M$ for every $t\in[T-1]$.
Then $U_{T+1}^{\cal V}=U_{T+1}^{\overline{\cal V}}=U_{T+1}$ on ${\cal K}_{T+1}$ and $U_t^{\cal V}=U_t^{\overline{\cal V}}=U_t$ on ${\cal D}_t$ for every $t\in[T]$.
\end{prop}

\begin{proof}
We proceed by backward induction.

\indpart{Terminal step} At $T+1$, both ${\cal V}$ and $\overline{\cal V}$ use the same terminal utility, so $U_{T+1}^{\cal V}=U_{T+1}^{\overline{\cal V}}=U_{T+1}$. For period $T$, by terminal compatibility we have
\[
U_T^{\cal V}(k_T,c)=U_T^{\overline{\cal V}}(k_T,c)=U_{T+1}(G_T(k_T,c,\star))=U_T(k_T,c).
\]

\indpart{Backward induction step} Fix $t\in[T-1]$, and suppose $U_{t+1}^{\cal V}=U_{t+1}^{\overline{\cal V}}=U_{t+1}$ on ${\cal D}_{t+1}$. For $k_t\in{\cal K}_t$ and $f=(c,f_+)\in{\cal F}_t$, we then have the equalities:
\[
\tilde{u}_{t+1}^{\cal V}(k_t,f)=\tilde{u}_{t+1}^{\overline{\cal V}}(k_t,f)=\tilde{u}_{t+1}(k_t,f)\in\Lambda_t^{\bf U}(k_t,c).
\]
Since $(k_t,\tilde u_{t+1}(k_t,f))\in D_t^M$, and both risk aggregators agree with $M_t^0$ on $D_t^M$, we have
\[
M_t(k_t,\tilde{u}_{t+1}^{\cal V}(k_t,f))=\overline{M}_t(k_t,\tilde{u}_{t+1}^{\overline{\cal V}}(k_t,f))=M_t^0(k_t,\tilde{u}_{t+1}(k_t,f))\triangleq v.
\]
Then, by definition we have $(k_t,c,v)\in D_t^W$. Both time aggregators agree with $W_t^0$ on $D_t^W$, and Proposition~\ref{prop:separation} gives
\[
U_t^{\cal V}(k_t,f)=W_t(k_t,c,v)=W_t^0(k_t,c,v)=U_t(k_t,f).
\]
Similarly, we obtain $U_t^{\overline{\cal V}}(k_t,f) = U_t(k_t,f)$.

\indpart{Conclusion} This reasoning shows that $U_{T+1}^{\cal V}=U_{T+1}^{\overline{\cal V}}=U_{T+1}$ and $U_t^{\cal V}=U_t^{\overline{\cal V}}=U_t$ for all $t \in [T]$.
\end{proof}

\begin{rem}[Extension invariance]
\label{rem:extension_invariance}
Corollaries~\ref{cor:risk_extension} and~\ref{cor:time_extension} provide existence of global extensions, not their uniqueness. Any jointly continuous, monotone extension of the effective aggregators is admissible, and the recursive structure selects one such pair. On the effective domains, the aggregators are uniquely determined by ${\bf U}$: normalization and the certainty-equivalent representation determine $M_t^0$, and Eq.~\eqref{eq:time_aggregator_construction} then determines $W_t^0$. The recursion of Theorem~\ref{thm:abstract_recursive} evaluates the global aggregators only at attainable arguments where all admissible extensions agree with the effective aggregators. Consequently, the represented preferences, attainable-state value functions, and the set of optimal policies for the original feasible problem are invariant to the selected global extensions. However, the belief and taste decomposition of Section~\ref{sec:separation} does depend on the choice of global extensions.
\end{rem}

\section{Proofs for Section~\ref{sec:history} (Full History)}

\subsection{Proof of Theorem~\ref{thm:recursive_utility}}

Let ${\cal G}^{\cal H}=(\{{\cal H}_t\}_{t=0}^{T+1},\{\iota_t\}_{t=0}^T)$ correspond to full-history preferences. Each space of histories ${\cal H}_t$ is a compact metric space and each history concatenation mapping $\iota_t$ is continuous, so ${\cal G}^{\cal H}$ is a state--transition system.
Axiom~\ref{axiom:weak_order}--Axiom~\ref{axiom:weak_separability} are exactly Axiom~\ref{axiom:abs_weak_order}--Axiom~\ref{axiom:abs_separability} stated for ${\cal G}^{\cal H}$, and Assumption~\ref{assu:CE_richness} is Assumption~\ref{assu:abs_CE_richness} stated for ${\cal G}^{\cal H}$.
Applying Theorem~\ref{thm:abstract_recursive} to ${\cal P}$ and the fixed compatible utility system ${\bf U}$ on ${\cal G}^{\cal H}$ yields a general recursive structure whose generated utility system equals ${\bf U}$. Together with the physical-state projections $\{\sigma_t^s\}$, this is a full-history recursive structure on $\mathfrak H$. We conclude that $(\mathfrak H,{\cal V})$ generates ${\bf U}$ and represents ${\cal P}$.

\section{Proofs for Section~\ref{sec:PA} (PA state)}

\subsection{Proof of Lemma~\ref{lem:odot_equivalence}}

We prove the claim by backward induction.

\indpart{Terminal step} For $t=T+1$, the relation $\odot_{T+1}$ is defined in terms of equality of terminal utility: $h_{T+1} \odot_{T+1} h_{T+1}'$ if and only if $U_{T+1}(h_{T+1}) = U_{T+1}(h_{T+1}')$. Equality is reflexive, symmetric, and transitive, so $\odot_{T+1}$ is an equivalence relation.

\indpart{Backward induction step} Fix $t\in[T]$, and suppose $\odot_{t+1}$ is an equivalence relation. For reflexivity of $\odot_t$, choose $h_t\in{\cal H}_t$. We have $U_t(h_t,f)=U_t(h_t,f)$ for all $f\in{\cal F}_t$, and the induction hypothesis gives $\iota_t(h_t,c,s) \odot_{t+1} \iota_t(h_t,c,s)$ for all $c\in{\cal C}$ and $s\in{\cal S}_{t+1}$. Then $h_t$ is equivalent to itself $h_t \odot_t h_t$.

For symmetry of $\odot_t$, suppose $h_t \odot_t h_t'$. Then $U_t(h_t,f)=U_t(h_t',f)$ and so $U_t(h_t',f)=U_t(h_t,f)$ for all $f\in{\cal F}_t$. We also have $\iota_t(h_t,c,s) \odot_{t+1} \iota_t(h_t',c,s)$ for all $c\in{\cal C}$ and $s\in{\cal S}_{t+1}$. By symmetry of $\odot_{t+1}$, this is equivalent to $\iota_t(h_t',c,s) \odot_{t+1} \iota_t(h_t,c,s)$ for all $c\in{\cal C}$ and $s\in{\cal S}_{t+1}$. We also have $\sigma_t^s(h_t)=\sigma_t^s(h_t')$ implies $\sigma_t^s(h_t')=\sigma_t^s(h_t)$, so $h_t' \odot_t h_t$.

For transitivity of $\odot_t$, suppose $h_t \odot_t h_t'$ and $h_t' \odot_t h_t''$. Then $U_t(h_t,f)=U_t(h_t',f)=U_t(h_t'',f)$ for all $f\in{\cal F}_t$. We also have $\iota_t(h_t,c,s) \odot_{t+1} \iota_t(h_t',c,s)$ and $\iota_t(h_t',c,s) \odot_{t+1} \iota_t(h_t'',c,s)$, for all $c\in{\cal C}$ and $s\in{\cal S}_{t+1}$.
By transitivity of $\odot_{t+1}$, we have $\iota_t(h_t,c,s) \odot_{t+1} \iota_t(h_t'',c,s)$.
Then, $\sigma_t^s(h_t)=\sigma_t^s(h_t')$ and $\sigma_t^s(h_t')=\sigma_t^s(h_t'')$ imply $\sigma_t^s(h_t)=\sigma_t^s(h_t'')$, and so $h_t \odot_t h_t''$.

\indpart{Conclusion} We conclude that $\odot_t$ is an equivalence relation. By backward induction, $\odot_t$ is an equivalence relation for all $t\in[T+1]$.

\subsection{Proof of Lemma~\ref{lem:quotient_topology}}

\emph{Step 1: Closedness of the canonical equivalence relations.}
We show that $\odot_t$ is closed for all $t\in[T+1]$ by backward induction.

\indpart{Step 1A: Terminal step} For the terminal period $t=T+1$, by definition we have
\[
\odot_{T+1} = \{(h_{T+1},h_{T+1}')\in{\cal H}_{T+1}\times{\cal H}_{T+1}: U_{T+1}(h_{T+1})=U_{T+1}(h_{T+1}')\}.
\]
Since $U_{T+1}$ is continuous, the map $(h_{T+1},h_{T+1}')\to U_{T+1}(h_{T+1})-U_{T+1}(h_{T+1}')$ is continuous. Then $\odot_{T+1}$ is the inverse image of the closed set $\{0\}$ under this continuous function, and so is closed by continuity.

\indpart{Step 1B: Backward induction step} Fix $t\in[T]$, and suppose $\odot_{t+1}$ is closed. Define
\[
B_t \triangleq \{(h_t,h_t')\in{\cal H}_t\times{\cal H}_t: U_t(h_t,f)=U_t(h_t',f),\, \forall f\in{\cal F}_t\}.
\]
For each fixed $f\in{\cal F}_t$, the map $(h_t,h_t')\to U_t(h_t,f)-U_t(h_t',f)$ is continuous by continuity of $U_t$. Each set $\{(h_t,h_t') : U_t(h_t,f)=U_t(h_t',f)\}$ is closed, and so
\[
B_t = \bigcap_{f\in{\cal F}_t} \{(h_t,h_t'): U_t(h_t,f)=U_t(h_t',f)\},
\]
is closed as an intersection of closed sets.
Next define
\[
C_t \triangleq \{(h_t,h_t')\in{\cal H}_t\times{\cal H}_t: (\iota_t(h_t,c,s),\iota_t(h_t',c,s)) \in \odot_{t+1},\, \forall c\in{\cal C},\, s\in{\cal S}_{t+1}\}.
\]
For fixed $(c,s)\in{\cal C}\times{\cal S}_{t+1}$, the map $\Psi_{t,c,s}:{\cal H}_t\times{\cal H}_t \to {\cal H}_{t+1}\times{\cal H}_{t+1}$ defined by $\Psi_{t,c,s}(h_t,h_t') = (\iota_t(h_t,c,s),\iota_t(h_t',c,s))$ is continuous by continuity of $\iota_t$. Since $\odot_{t+1}$ is closed by the induction hypothesis, the inverse image $\Psi_{t,c,s}^{-1}(\odot_{t+1})$ is closed and therefore $C_t = \bigcap_{(c,s) \in {\cal C} \times {\cal S}_{t+1}} \Psi_{t,c,s}^{-1}(\odot_{t+1})$ is closed as an intersection of closed sets.
Finally, let $P_t\triangleq\{(h_t,h_t')\in{\cal H}_t\times{\cal H}_t:\sigma_t^s(h_t)=\sigma_t^s(h_t')\}$ for the physical readout. Then $P_t=(\sigma_t^s\times\sigma_t^s)^{-1}({\rm diag}({\cal S}_t))$ is closed since the diagonal ${\rm diag}({\cal S}_t)\triangleq\{(s,s):s\in{\cal S}_t\}$ is closed and $\sigma_t^s$ is continuous.
We have $\odot_t=B_t\cap C_t \cap P_t$ by the definition of $\odot_t$, which is closed since $B_t$, $C_t$, and $P_t$ are all closed.

\indpart{Step 1C: Conclusion} By backward induction, $\odot_t$ is closed for all $t \in [T+1]$.

\emph{Step 2: Compactness and metrizability of the canonical quotient space.}
This claim is a direct application of Theorem~\ref{thm:closed_quotient}, since $\{\odot_t\}_{t=0}^{T+1}$ are closed.

\emph{Step 3: Quotient projection and readout map.}
The canonical projection $\sigma_t:{\cal H}_t\to{\cal X}_t$ is continuous and a quotient map by definition of the quotient topology. Further, it is closed by Theorem~\ref{thm:closed_quotient}. Since $\sigma_t^s$ is continuous and constant on $\odot_t$-classes, the universal property of quotient maps gives a unique continuous map $\varsigma_t:{\cal X}_t\to{\cal S}_t$ such that $\varsigma_t\circ\sigma_t=\sigma_t^s$.

\subsection{Proof of Lemma~\ref{lem:quotient_transition_continuity}}

We will construct unique maps $\Gamma_t:{\cal X}_t\times{\cal C}\times{\cal S}_{t+1}\to{\cal X}_{t+1}$ such that
\[
\Gamma_t(\sigma_t(h_t),c,s) = \sigma_{t+1}(\iota_t(h_t,c,s)), \forall h_t\in{\cal H}_t,\, c\in{\cal C},\, s\in{\cal S}_{t+1},\, t \in [T],   
\]
and
\[
\varsigma_{t+1}(\Gamma_t(x_t,c,s))=s, \forall x_t\in{\cal X}_t,\, c\in{\cal C},\, s\in{\cal S}_{t+1},\, t \in [T].
\]
Then, we will verify that they are continuous.

The one-step history extension map $\iota_t:{\cal H}_t\times{\cal C}\times{\cal S}_{t+1}\to{\cal H}_{t+1}$ is continuous since it is a concatenation. Next consider the map $\sigma_{t+1} \circ \iota_t : {\cal H}_t\times{\cal C}\times{\cal S}_{t+1} \to {\cal X}_{t+1}$.
By definition of the canonical projection and forward stability, $\sigma_t(h_t)=\sigma_t(h_t')$ implies $h_t\odot_t h_t'$ and $\iota_t(h_t,c,s)\odot_{t+1}\iota_t(h_t',c,s)$ for all $c\in{\cal C}$ and $s\in{\cal S}_{t+1}$. Hence $\sigma_{t+1}(\iota_t(h_t,c,s)) = \sigma_{t+1}(\iota_t(h_t',c,s))$ and so $\sigma_{t+1}\circ \iota_t$ is constant on the fibers of $\sigma_t\times{\rm id}_{\cal C}\times{\rm id}_{{\cal S}_{t+1}}$.
Then, there is a unique map $\Gamma_t:{\cal X}_t\times{\cal C}\times{\cal S}_{t+1}\to{\cal X}_{t+1}$ satisfying $\Gamma_t(\sigma_t(h_t),c,s) = \sigma_{t+1}(\iota_t(h_t,c,s))$.
For the terminal period $t=T$, we have $\iota_T(h_T,c,\star)=(h_T,c)$, and the terminal transition is $\Gamma_T(x_T,c,\star)$.

Next, we note that $\sigma_t$ is a quotient map and ${\cal C}\times{\cal S}_{t+1}$ is compact Hausdorff. By Theorem~\ref{thm:whitehead}, $\sigma_t\times{\rm id}_{\cal C}\times{\rm id}_{{\cal S}_{t+1}}$ is a quotient map. In addition, we have $\Gamma_t\circ (\sigma_t\times{\rm id}_{\cal C}\times{\rm id}_{{\cal S}_{t+1}}) = \sigma_{t+1}\circ \iota_t$.
The RHS of this equality is continuous, and so $\Gamma_t$ is continuous by Theorem~\ref{thm:quotient_universal_property}.

Finally, the physical readout identity follows because
\[
\varsigma_{t+1}(\Gamma_t(\sigma_t(h_t),c,s))=\varsigma_{t+1}(\sigma_{t+1}(\iota_t(h_t,c,s)))=\sigma_{t+1}^s(\iota_t(h_t,c,s))=s,
\]
and $\sigma_t$ is surjective.

\subsection{Proof of Lemma~\ref{lem:quotient_utilities}}

\emph{Step 1: Well-defined.}
For the terminal period $t = T+1$, if $[h_{T+1}]=[h_{T+1}']$ then $U_{T+1}(h_{T+1})=U_{T+1}(h_{T+1}')$ by definition, so $U_{T+1}^\circ(\sigma_{T+1}(h_{T+1})) \triangleq U_{T+1}(h_{T+1})$ is automatically well-defined.

We now show that $U_t^\circ$ is well-defined for $t\in[T]$.
Suppose $h_t,h_t'\in{\cal H}_t$ satisfy $h_t\odot_t h_t'$, so $(h_t,f)\sim_{(t)}(h_t',f)$ for all $f\in{\cal F}_t$ by definition of $\odot_t$.
Since $U_t$ represents $\succeq_{(t)}$ on the grand domain ${\cal D}_t$, this indifference is equivalent to $U_t(h_t,f)=U_t(h_t',f)$ for all $f\in{\cal F}_t$.
Therefore, the value $U_t^\circ(\sigma_t(h_t),f)\triangleq U_t(h_t,f)$ does not depend on the representative history $h_t$ of the canonical equivalence class.

\emph{Step 2: Continuity.}
The terminal factored utility $U_{T+1}^\circ$ is continuous because $U_{T+1}^\circ\circ\sigma_{T+1} = U_{T+1}$, $U_{T+1}$ is continuous, and $\sigma_{T+1}$ is a quotient map.
For periods $t\in[T]$, $\sigma_t\times{\rm id}_{{\cal F}_t}$ is a quotient map by Theorem~\ref{thm:whitehead} since ${\cal F}_t$ is compact Hausdorff.
We then have that $U_t^\circ$ is continuous by Theorem~\ref{thm:quotient_universal_property} because $U_t^\circ\circ(\sigma_t\times{\rm id}_{{\cal F}_t})=U_t$, $U_t$ is continuous, and $\sigma_t \times {\rm id}_{{\cal F}_t}$ is a quotient map.

\subsection{Proof of Theorem~\ref{thm:recursive_memory}}

The quotient utilities ${\bf U}^\circ = \{U_t^\circ\}_{t=0}^{T+1}$ induce preferences on the PA-state system, defined as follows.

\begin{defn}
\label{defn:induced_preference}
The \emph{induced grand preference} is a tuple ${\cal P}^\circ = (\{\succeq_{(t)}^\circ\}_{t=0}^T, V_{T+1}^\circ)$ where:

(i) For $t \in [T]$, $\succeq_{(t)}^\circ$ is a binary relation on ${\cal X}_t \times {\cal F}_t$ defined by $(x_t, f) \succeq_{(t)}^\circ (x_t', g)$ if and only if $U_t^\circ(x_t,f)\geq U_t^\circ(x_t',g)$.

(ii) $V_{T+1}^\circ\triangleq U_{T+1}^\circ$.
\end{defn}
\noindent
The conditional preference $\succeq_{x_t}^\circ$ is defined by $f \succeq_{x_t}^\circ g$ if and only if $(x_t, f) \succeq_{(t)}^\circ (x_t, g)$.

The canonical PA-state system is well-behaved: the canonical projections $\{\sigma_t\}_{t=0}^{T+1}$ are continuous by Lemma~\ref{lem:quotient_topology}, and the transitions are consistent by Lemma~\ref{lem:quotient_transition_continuity}.
Then, the quotient utilities induce a well-defined grand preference on the canonical PA-state system, which inherits all of our behavioral axioms from ${\cal P}$.

\begin{lem}[Axiom inheritance]
\label{lem:axiom_inheritance}
If ${\cal P}$ satisfies Axiom~\ref{axiom:weak_order}--Axiom~\ref{axiom:weak_separability}, then the canonical quotient preference ${\cal P}^\circ = (\{\succeq_{(t)}^\circ\}_{t=0}^T,V_{T+1}^\circ)$ satisfies Axiom~\ref{axiom:abs_weak_order}--Axiom~\ref{axiom:abs_separability} on the underlying state--transition system ${\cal G}^{\cal X}=(\{{\cal X}_t\}_{t=0}^{T+1},\{\Gamma_t\}_{t=0}^T)$.
\end{lem}
\begin{proof}
Each comparison of the induced preferences at $x_t$ is determined by the comparison at any $h_t\in\sigma_t^{-1}(x_t)$, independent of the representative by Lemma~\ref{lem:quotient_utilities}.
For example, if $x_t=\sigma_t(h_t)$, then every comparison at $x_t$ is represented by the corresponding comparison at $h_t$.
Weak order and continuity of the representing $U_t^\circ$ by Lemma~\ref{lem:quotient_utilities} then descend to the quotient.

For dynamic consistency, fix $x_t=\sigma_t(h_t)$ and suppose
\[
(\Gamma_t(x_t,c,s),f_+(s))\succeq_{(t+1)}^\circ(\Gamma_t(x_t,c,s),g_+(s)),\, \forall s\in{\cal S}.
\]
By transition consistency, $\Gamma_t(x_t,c,s) = \sigma_{t+1}(\iota_t(h_t,c,s))$, so the quotient comparisons are equivalent to
\[
(\iota_t(h_t,c,s),f_+(s))\succeq_{(t+1)}(\iota_t(h_t,c,s),g_+(s)),\, \forall s\in{\cal S}.
\]
Axiom~\ref{axiom:dynamic_consistency} gives $(h_t,(c,f_+))\succeq_{(t)}(h_t,(c,g_+))$,
which descends to $(x_t,(c,f_+))\succeq_{(t)}^\circ(x_t,(c,g_+))$.

The same representative-history argument proves compensated monotonicity and weak separability. Terminal compatibility descends to the quotient because $U_T^\circ(x_T,c)=U_{T+1}^\circ(\Gamma_T(x_T,c,\star))=V_{T+1}^\circ(\Gamma_T(x_T,c,\star))$.
\end{proof}

We now show that Assumption~\ref{assu:CE_richness} descends to the quotient, so we do not have to separately assume certainty equivalent richness on the quotient.

\begin{lem}[Descent of certainty-equivalent richness]
\label{lem:CE_descent}
If ${\bf U}$ is a compatible utility system satisfying Assumption~\ref{assu:CE_richness} on ${\cal G}^{\cal H}=(\{{\cal H}_t\}_{t=0}^{T+1},\{\iota_t\}_{t=0}^T)$, then the quotient utility system ${\bf U}^{\circ}$ satisfies Assumption~\ref{assu:abs_CE_richness} on ${\cal G}^{\cal X}=(\{{\cal X}_t\}_{t=0}^{T+1},\{\Gamma_t\}_{t=0}^T)$.
\end{lem}
\begin{proof}
We fix $t\in[T-1]$ for all of the following steps, and then conclude that the argument holds for any $t\in[T-1]$.

\emph{Step 1: Equality of effective domains.}
Fix $x_t\in{\cal X}_t$ and $h_t,h_t'\in\sigma_t^{-1}(x_t)$, then we have $\Lambda_t^{\bf U}(h_t,c)=\Lambda_t^{\bf U}(h_t',c)$ for all $c\in{\cal C}$. To establish this equality, let $\tilde{u}\in\Lambda_t^{\bf U}(h_t,c)$ and choose $f_+\in({\cal F}_{t+1})^{\cal S}$ such that $\tilde{u}(s)=U_{t+1}(\iota_t(h_t,c,s),f_+(s))$ for all $s\in{\cal S}$. Since $h_t\odot_t h_t'$, forward stability gives $\iota_t(h_t,c,s)\odot_{t+1}\iota_t(h_t',c,s)$ for all $s \in {\cal S}$. Lemma~\ref{lem:quotient_utilities} then implies that $U_{t+1}(\iota_t(h_t,c,s),f_+(s))=U_{t+1}(\iota_t(h_t',c,s),f_+(s))$ for all $s \in {\cal S}$, so $(c,f_+)$ also realizes $\tilde{u}$ at $h_t'$ and we have $\tilde{u} \in \Lambda_t^{\bf U}(h_t',c)$. The reverse inclusion follows symmetrically to give:
\[
\Lambda_t^{{\bf U}^{\circ}}(x_t,c)=\Lambda_t^{\bf U}(h_t,c),\, \Lambda_t^{{\bf U}^{\circ}}(x_t)=\Lambda_t^{\bf U}(h_t),\, \forall h_t\in\sigma_t^{-1}(x_t).
\]

\emph{Step 2: Constant effective certainty equivalent.}
Fix $\tilde{u}\in\Lambda_t^{\bf U}(h_t)$, choose $c\in{\cal C}$ and $f=(c,f_+)\in{\cal F}_t$ realizing $\tilde{u}$ at $h_t$, and let $v=M_t^0(h_t,\tilde{u})$. By deterministic solvability, there exists $g=(c,g_+)\in{\cal F}_t$ such that $(c,g_+)$ realizes ${\bf v}$ at $h_t$. By Step 1, the same plans $f$ and $g$ generate $\tilde{u}$ and ${\bf v}$, respectively, at any $h_t'$ with $h_t' \odot_t h_t$.

Normalization gives $M_t^0(h_t,{\bf v})=v=M_t^0(h_t,\tilde{u})$. Since both $\tilde{u}$ and ${\bf v}$ are jointly attainable at $c$, the certainty-equivalent representation implies $(h_t,f)\sim_{(t)}(h_t,g)$. In addition, $h_t \odot_t h_t'$ implies $(h_t,f)\sim_{(t)}(h_t',f)$ and $(h_t,g)\sim_{(t)}(h_t',g)$, so $(h_t',f) \sim_{(t)} (h_t',g)$ by transitivity. Applying the certainty-equivalent representation at $h_t'$ and using normalization of $M_t^0$ yields $M_t^0(h_t',\tilde{u})=M_t^0(h_t',{\bf v})=v=M_t^0(h_t,\tilde{u})$.
Then, $M_t^0(h_t,\tilde{u})$ is constant on $h_t \in \sigma_t^{-1}(x_t)$.

\emph{Step 3: Define certainty equivalent and prove continuity.}
Define the history-level and quotient-level risk aggregator effective domains by
\[
D_t^{M,H}\triangleq\{(h_t,\tilde{u}):\tilde{u}\in\Lambda_t^{\bf U}(h_t)\},\,
D_t^{M,X}\triangleq\{(x_t,\tilde{u}):\tilde{u}\in\Lambda_t^{{\bf U}^{\circ}}(x_t)\}.
\]
Define $p_t:D_t^{M,H}\to D_t^{M,X}$ by $p_t(h_t,\tilde{u})=(\sigma_t(h_t),\tilde{u})$. By Step 1, $p_t$ is onto. It is also continuous as the restriction of the continuous map $\sigma_t\times{\rm id}_{\cal L}$ to $D_t^{M,H}$. The set $D_t^{M,H}$ is compact by Lemma~\ref{lem:graph_compact}, while $D_t^{M,X}$ is Hausdorff as a subspace of ${\cal X}_t\times{\cal L}$. Then, the continuous surjection $p_t$ is a quotient map.

By Step 2, $M_t^0:D_t^{M,H}\to\Re$ is constant on the fibers of $p_t$. It follows that there is a unique function $M_t^{\circ,0}:D_t^{M,X}\to\Re$ satisfying $M_t^{\circ,0}(\sigma_t(h_t),\tilde{u})=M_t^0(h_t,\tilde{u})$.
Since $M_t^0$ is continuous and $p_t$ is a quotient map, $M_t^{\circ,0}$ is jointly continuous on $D_t^{M,X}$ by Theorem~\ref{thm:quotient_universal_property}.

\emph{Step 4: Certainty-equivalent richness on the quotient.}
Let $f=(c,f_+)$ and $g=(c,g_+)$ have the same current consumption at $x_t$, and choose $h_t\in\sigma_t^{-1}(x_t)$. Then we have
\begin{align*}
    f \succeq_{x_t}^{\circ} g \iff & f\succeq_{h_t}g \iff M_t^0(h_t,\tilde{u}_{t+1}(h_t,f))\geq M_t^0(h_t,\tilde{u}_{t+1}(h_t,g))\\
    \iff & M_t^{\circ,0}(x_t,\tilde{u}_{t+1}^{\circ}(x_t,f)) \geq M_t^{\circ,0}(x_t,\tilde{u}_{t+1}^{\circ}(x_t,g)),
\end{align*}
which shows that $M_t^{\circ,0}$ represents the quotient conditional risk ranking. Normalization and internality then follow from their history-level counterparts.

For deterministic solvability, fix $x_t$ and $f=(c,f_+)$, choose $h_t\in\sigma_t^{-1}(x_t)$, and let
\[
v=M_t^{\circ,0}(x_t,\tilde{u}_{t+1}^{\circ}(x_t,f)).
\]
By history-level deterministic solvability, there is a plan $g=(c,g_+)$ that realizes ${\bf v}$ at $h_t$. By Step 1, $g$ has quotient continuation-utility vector ${\bf v}$ at $x_t$.

Finally, suppose $\tilde{u},\tilde{v}\in\Lambda_t^{{\bf U}^{\circ}}(x_t)$ and $\tilde{u}\leq\tilde{v}$. By Step 1, we have $\tilde{u},\tilde{v} \in \Lambda_t^{\bf U}(h_t)$ for any $h_t\in\sigma_t^{-1}(x_t)$.
Then the history-level monotone completion gives $M_t^0(h_t,\tilde{u})\leq M_t^0(h_t,\tilde{v})$, and so $M_t^{\circ,0}(x_t,\tilde{u})\leq M_t^{\circ,0}(x_t,\tilde{v})$.
We conclude that $M_t^{\circ,0}$ satisfies joint continuity, certainty-equivalent representation, normalization, internality, deterministic solvability, and monotone completion. Since $t\in[T-1]$ was arbitrary, ${\bf U}^{\circ}$ satisfies Assumption~\ref{assu:abs_CE_richness} on ${\cal G}^{\cal X}$.
\end{proof}

\begin{proof}[Proof of Theorem~\ref{thm:recursive_memory}]
Let ${\cal G}^{\cal X}=(\{{\cal X}_t\}_{t=0}^{T+1},\{\Gamma_t\}_{t=0}^T)$ correspond to the canonical PA quotient. The spaces ${\cal X}_t$ are compact and metrizable by Lemma~\ref{lem:quotient_topology}, and all $\Gamma_t$ are continuous by Lemma~\ref{lem:quotient_transition_continuity}, so ${\cal G}^{\cal X}$ is a state--transition system.
By Lemma~\ref{lem:axiom_inheritance}, the quotient preference satisfies Axiom~\ref{axiom:abs_weak_order}--Axiom~\ref{axiom:abs_separability}.
By Lemma~\ref{lem:CE_descent}, the quotient utility system satisfies Assumption~\ref{assu:abs_CE_richness}. Then, we can apply Theorem~\ref{thm:abstract_recursive} to the quotient grand preference and its fixed compatible quotient utility system ${\bf U}^\circ=\{U_t^\circ\}_{t=0}^{T+1}$.
Theorem~\ref{thm:abstract_recursive} yields a recursive structure ${\cal V}^*$ whose generated utility system equals ${\bf U}^\circ$. Together with the physical readout maps, this is a recursive structure on the PA-state system $\mathfrak{X}^*$.
Taking the canonical projections $\{\sigma_t\}_{t=0}^{T+1}$ as summary maps, Lemma~\ref{lem:quotient_utilities} gives utility factorization, Lemma~\ref{lem:quotient_transition_continuity} gives $\sigma_{t+1}(\iota_t(h_t,c,s))=\Gamma_t(\sigma_t(h_t),c,s)$, and the definition of the canonical readout gives $\varsigma_t(\sigma_t(h_t))=\sigma_t^s(h_t)$.
All the clauses of Definition~\ref{defn:memory_representation} hold, and we conclude that $(\mathfrak{X}^*,{\cal V}^*)$ factorizes ${\bf U}$ in the sense of Definition~\ref{defn:memory_representation}.
\end{proof}

\subsection{Proof of Proposition~\ref{prop:minimality}}

\emph{Step 1: Physical-readout consistency.}
Because $(\mathfrak{X}',{\cal V}')$ factorizes ${\bf U}$, by definition its summary maps satisfy
\[
\varsigma_t'(\sigma_t'(h_t))=\sigma_t^s(h_t),\, \forall h_t\in{\cal H}_t,\, t\in[T+1].
\]

\emph{Step 2: Refinement.}
We show by induction that the summary maps $\{\sigma_t'\}_{t=0}^{T+1}$ are a refinement of the canonical projections $\{\sigma_t\}_{t=0}^{T+1}$.

\indpart{Step 2A: Terminal step} For $t=T+1$, suppose $\sigma_{T+1}'(h_{T+1})=\sigma_{T+1}'(h_{T+1}')$. By the terminal utility factorization in Eq.~\eqref{eq:factorization}, we have
\[
U_{T+1}(h_{T+1}) = U_{T+1}'(\sigma_{T+1}'(h_{T+1})) = U_{T+1}'(\sigma_{T+1}'(h_{T+1}')) = U_{T+1}(h_{T+1}'),
\]
so $(h_{T+1},h_{T+1}')\in \odot_{T+1}$.
 
\indpart{Step 2B: Backward induction step} Fix $t\in[T]$, and suppose the claim holds at $t+1$. Take $h_t, h_t' \in {\cal H}_t$ such that $\sigma_t'(h_t)=\sigma_t'(h_t')$, then we verify the three conditions of $\odot_t$.
First, since $\varsigma_t'\circ\sigma_t'=\sigma_t^s$ and $\sigma_t'(h_t)=\sigma_t'(h_t')$, we get $\sigma_t^s(h_t)=\sigma_t^s(h_t')$.
Second, for every $f\in{\cal F}_t$, the utility factorization gives
\[
U_t(h_t,f) = U_t'(\sigma_t'(h_t),f) = U_t'(\sigma_t'(h_t'),f) = U_t(h_t',f).
\]
Since ${\bf U}$ is compatible with ${\cal P}$, this implies $(h_t,f)\sim_{(t)}(h_t',f)$ for all $f\in{\cal F}_t$.
Third, for every $c\in{\cal C}$ and $s\in{\cal S}_{t+1}$, transition consistency Eq.~\eqref{eq:transition_consistency} gives
\[
\sigma_{t+1}'(\iota_t(h_t,c,s)) = \Gamma_t'(\sigma_t'(h_t),c,s) = \Gamma_t'(\sigma_t'(h_t'),c,s) = \sigma_{t+1}'(\iota_t(h_t',c,s)).
\]
By the induction hypothesis, $\iota_t(h_t,c,s)\odot_{t+1}\iota_t(h_t',c,s)$ for all $c\in{\cal C}$ and $s\in{\cal S}_{t+1}$.
All three clauses of Definition~\ref{defn:equivalence} hold, so $(h_t,h_t')\in \odot_t$.

\indpart{Step 2C: Conclusion} By backward induction, we conclude that $\sigma_t'(h_t)=\sigma_t'(h_t')$ implies $h_t\odot_t h_t'$ for all $t\in[T+1]$.
 
\emph{Step 3: Construction of $\theta_t$.}
Fix $t\in[T+1]$. For $x_t'\in\hat{\cal X}_t'$, choose $h_t\in{\cal H}_t$ with $\sigma_t'(h_t)=x_t'$ and define $\theta_t(x_t')\triangleq\sigma_t(h_t)$.
To show that this construction is well-defined, suppose $h_t'\in{\cal H}_t$ also satisfies $\sigma_t'(h_t')=x_t'$. By Step 2, we have $h_t\odot_t h_t'$ which gives $\sigma_t(h_t)=\sigma_t(h_t')$.
Further, we have $\theta_t\circ\sigma_t'=\sigma_t$ on ${\cal H}_t$ by construction.
Surjectivity of $\theta_t$ follows because $\sigma_t$ is onto ${\cal X}_t$.
Any map $\tilde{\theta}_t:\hat{\cal X}_t'\to{\cal X}_t$ with $\tilde{\theta}_t\circ\sigma_t'=\sigma_t$ agrees with $\theta_t$ everywhere on its domain, so it is unique.
 
\emph{Step 4: Continuity of $\theta_t$.}
The mapping $\sigma_t':{\cal H}_t\to\hat{\cal X}_t'$ is a continuous surjection from the compact set ${\cal H}_t$ onto $\hat{\cal X}_t'$, which is Hausdorff as a subspace of the Hausdorff space ${\cal X}_t'$. In addition, $\hat{\cal X}_t'$ is compact as the continuous image of a compact set.
By Theorem~\ref{thm:compact_to_hausdorff}, $\sigma_t'$ is a quotient map.
Then since $\theta_t\circ\sigma_t'=\sigma_t$ is continuous, $\theta_t$ is continuous by Theorem~\ref{thm:quotient_universal_property}.
 
\emph{Step 5: Dynamics.}
Fix $t\in[T]$, $x_t'\in\hat{\cal X}_t'$, $c\in{\cal C}$, and $s\in{\cal S}_{t+1}$, and choose $h_t\in{\cal H}_t$ with $\sigma_t'(h_t)=x_t'$.
Transition consistency Eq.~\eqref{eq:transition_consistency} for $\sigma_t'$ gives $\Gamma_t'(x_t',c,s) = \sigma_{t+1}'(\iota_t(h_t,c,s)) \in \hat{\cal X}_{t+1}'$.
Then,
\[
\theta_{t+1}(\Gamma_t'(x_t',c,s)) = \theta_{t+1}(\sigma_{t+1}'(\iota_t(h_t,c,s))) = \sigma_{t+1}(\iota_t(h_t,c,s)) = \Gamma_t(\sigma_t(h_t),c,s) = \Gamma_t(\theta_t(x_t'),c,s).
\]
 
\emph{Step 6: Factored utilities.}
Fix $t\in[T]$, $x_t'\in\hat{\cal X}_t'$, and $f\in{\cal F}_t$, and choose $h_t\in{\cal H}_t$ with $\sigma_t'(h_t)=x_t'$. Then
\[
U_t'(x_t',f) = U_t(h_t,f) = U_t^\circ(\sigma_t(h_t),f) = U_t^\circ(\theta_t(x_t'),f),
\]
where the first equality is by the utility factorization of ${\cal V}'$ and the second equality is by Lemma~\ref{lem:quotient_utilities}.
The terminal identity $U_{T+1}'(x_{T+1}')=U_{T+1}^\circ(\theta_{T+1}(x_{T+1}'))$ for $x_{T+1}'\in\hat{\cal X}_{T+1}'$ follows similarly.
 
\emph{Step 7: Homeomorphism.} 
Now suppose the additional hypothesis that $h_t\odot_t h_t'$ implies
$\sigma_t'(h_t)=\sigma_t'(h_t')$ holds. We show that $\theta_t$ is injective for each $t\in[T+1]$.
Suppose $\theta_t(x_t')=\theta_t(x_t'')$ for $x_t',x_t''\in\hat{\cal X}_t'$, and choose $h_t,h_t'\in{\cal H}_t$ with $\sigma_t'(h_t)=x_t'$ and $\sigma_t'(h_t')=x_t''$.
Then $\sigma_t(h_t)=\theta_t(x_t')=\theta_t(x_t'')=\sigma_t(h_t')$ since $\theta_t \circ \sigma_t' = \sigma_t$, and so $h_t\odot_t h_t'$.
By the additional hypothesis that $h_t\odot_t h_t'$ implies $\sigma_t'(h_t)=\sigma_t'(h_t')$, we have $x_t'=\sigma_t'(h_t)=\sigma_t'(h_t')=x_t''$, so $\theta_t$ is injective. It follows that $\theta_t$ is a continuous bijection.
Since $\hat{\cal X}_t'$ is compact and ${\cal X}_t$ is Hausdorff by Lemma~\ref{lem:quotient_topology}, $\theta_t$ is a homeomorphism by Theorem~\ref{thm:compact_to_hausdorff}.

\section{Proofs for Section~\ref{sec:DP} (Dynamic Programming)}

\subsection{Proof of Theorem~\ref{thm:DP_Bellman}}

Recall that the aggregators $\{M_t^*\}_{t=0}^{T-1}$ and $\{W_t^*\}_{t=0}^{T-1}$ on the PA state are jointly continuous by definition of ${\cal V}^*$.

\emph{Step 1: Candidate value functions.}
We will construct candidate optimal value functions $\{V_t\}_{t=0}^T$ by backward recursion. For $t = T$, let
\[
V_T(x_T)\triangleq\max_{c\in{\cal A}_T(x_T)}U_{T+1}^*(\Gamma_T(x_T,c,\star)),\, \forall x_T \in {\cal X}_T.
\]
For $t\in[T-1]$, given $V_{t+1}$, define $\tilde{v}_{t+1}(x_t,c) \in {\cal L}$ where $[\tilde{v}_{t+1}(x_t,c)](s)\triangleq V_{t+1}(\Gamma_t(x_t,c,s))$ for all $s \in {\cal S}$, and
\[
V_t(x_t)\triangleq\max_{c\in{\cal A}_t(x_t)}W_t^*(x_t,c,M_t^*(x_t,\tilde{v}_{t+1}(x_t,c))),\, \forall x_t \in {\cal X}_t.
\]

\emph{Step 2: Continuity and attainment of the maxima.}
For $t=T$, the objective
\[
(x_T,c)\to U_{T+1}^*(\Gamma_T(x_T,c,\star))
\]
is continuous. Since ${\cal A}_T$ is a continuous correspondence with nonempty compact values, Theorem~\ref{thm:berge} implies that $V_T$ is continuous and the maximum for each $x_T$ is attained. Next suppose $V_{t+1}$ is continuous. Since ${\cal S}$ is finite and $\Gamma_t$ is continuous, $(x_t,c)\to\tilde{v}_{t+1}(x_t,c)$ is continuous as a map into ${\cal L}\cong\Re^{\cal S}$.
The aggregators $M_t^*$ and $W_t^*$ are jointly continuous, so $V_t$ is continuous and its maximum is attained for each $x_t$ by Theorem~\ref{thm:berge}. By backward induction, these claims hold for every $t\in[T]$.

\emph{Step 3: Optimal PA policy.}
For each $x_T\in{\cal X}_T$, choose
\[
\bar\pi_T(x_T) \in \arg\max_{c\in{\cal A}_T(x_T)}U_{T+1}^*(\Gamma_T(x_T,c,\star)).
\]
For each $t\in[T-1]$ and $x_t\in{\cal X}_t$, choose
\[
\bar\pi_t(x_t)\in\arg\max_{c\in{\cal A}_t(x_t)}W_t^*(x_t,c,M_t^*(x_t,\tilde{v}_{t+1}(x_t,c))).
\]
The set $\Pi=\prod_{t=0}^T{\rm Sel}({\cal A}_t)$ contains all pointwise feasible selectors, so $\bar\pi\in\Pi$.

\emph{Step 4: Upper bound.}
We now upper bound the attainable utility of all Markov policies.

\indpart{Step 4A: Terminal step} At $t=T$, we have
\[
U_T^{*\pi}(x_T)=U_{T+1}^*(\Gamma_T(x_T,\pi_T(x_T),\star))\leq V_T(x_T),
\]
by definition of $V_T(x_T)$.

\indpart{Step 4B: Backward induction step} Fix $t \in [T-1]$, and suppose $U_{t+1}^{*\pi}(x_{t+1})\leq V_{t+1}(x_{t+1})$ for all $\pi\in\Pi$ and $x_{t+1}\in{\cal X}_{t+1}$.
Then fix $\pi\in\Pi$, $x_t\in{\cal X}_t$, and let $c=\pi_t(x_t)$. By the inductive hypothesis, we have
\[
[\tilde{u}_{t+1}^{*\pi}(x_t,c)](s)=U_{t+1}^{*\pi}(\Gamma_t(x_t,c,s))\leq V_{t+1}(\Gamma_t(x_t,c,s))=[\tilde{v}_{t+1}(x_t,c)](s),\, \forall s\in{\cal S},
\]
so $\tilde{u}_{t+1}^{*\pi}(x_t,c)\leq\tilde{v}_{t+1}(x_t,c)$. By monotonicity of $M_t^*(x_t,\cdot)$ and $W_t^*(x_t,c,\cdot)$, we have
\[
U_t^{*\pi}(x_t)\leq W_t^*(x_t,c,M_t^*(x_t,\tilde{v}_{t+1}(x_t,c)))\leq V_t(x_t).
\]

\indpart{Step 4C: Conclusion} By backward induction, we obtain the upper bound $U_t^{*\pi}(x_t)\leq V_t(x_t)$ for all $\pi\in\Pi$, $x_t\in{\cal X}_t$ and $t \in [T]$.

\emph{Step 5: Simultaneous attainment by $\bar\pi$.}
We show that $\bar{\pi}$ attains the optimal value for every state and period by backward induction.

\indpart{Step 5A: Terminal step} At $t=T$, the definition of $\bar\pi_T$ gives $U_T^{*\bar\pi}(x_T)=V_T(x_T)$ for every $x_T \in {\cal X}_T$.

\indpart{Step 5B: Backward induction step} Suppose $U_{t+1}^{*\bar\pi}(x_{t+1})=V_{t+1}(x_{t+1})$ for every $x_{t+1}\in{\cal X}_{t+1}$. For any $x_t\in{\cal X}_t$, let $c=\bar\pi_t(x_t)$. Then
\[
[\tilde{u}_{t+1}^{*\bar\pi}(x_t,c)](s)=V_{t+1}(\Gamma_t(x_t,c,s))=[\tilde{v}_{t+1}(x_t,c)](s),\, \forall s\in{\cal S}.
\]
Since $\bar\pi_t(x_t)$ attains the maximum for $V_t(x_t)$,
\[
U_t^{*\bar\pi}(x_t)=W_t^*(x_t,\bar\pi_t(x_t),M_t^*(x_t,\tilde{v}_{t+1}(x_t,\bar\pi_t(x_t))))=V_t(x_t).
\]

\indpart{Step 5C: Conclusion} By backward induction, $\bar\pi$ attains the maximum in $V_t$ for all states and periods.

\emph{Step 6: Bellman recursion.}
For period $t=T$, by the definition of $J_T^*$ we have
\[
J_T^*(x_T)=\max_{c \in {\cal A}_T(x_T)} U_{T+1}^*(\Gamma_T(x_T, c,\star))=V_T(x_T).
\]
For $t \in [T-1]$, the upper bound and attainment imply
\[
J_t^*(x_t)=\sup_{\pi\in\Pi}U_t^{*\pi}(x_t)=U_t^{*\bar\pi}(x_t)=V_t(x_t).
\]
Substituting $J_t^*=V_t$ into the recursive definition of $V_t$ gives Eq.~\eqref{eq:Bellman}, and $U_t^{*\bar\pi}(x_t)=J_t^*(x_t)$ for every state and period.

\subsection{Proof of Theorem~\ref{thm:DP_verification}}

\emph{Step 1: Policy-value identity.}
Fix $\varphi\in\Phi$ and $f^{\varphi}[h_t]\in{\cal F}_t$, so $U_t^\varphi(h_t)=U_t(h_t,f^{\varphi}[h_t])$. Because $(\mathfrak{X},{\cal V}^*)$ factorizes ${\bf U}$, we also have $U_t^\varphi(h_t)=U_t^*(\sigma_t(h_t),f^{\varphi}[h_t])$.
For period $T$, we have $U_T^\varphi(h_T)=U_{T+1}^*(\Gamma_T(x_T,\varphi_T(h_T),\star))$ for $x_T = \sigma_T(h_T)$.
For all $t\in[T-1]$, let $x_t=\sigma_t(h_t)$ and expand the memory recursion Eq.~\eqref{eq:memory_recursive} with transition consistency Eq.~\eqref{eq:transition_consistency} to obtain,
\[
U_t^\varphi(h_t)=W_t^*(x_t,\varphi_t(h_t),M_t^*(x_t,\tilde{u}_{t+1}^\varphi(h_t,\varphi_t(h_t)))),
\]
where $[\tilde{u}_{t+1}^\varphi(h_t,c)](s)=U_{t+1}^\varphi(\iota_t(h_t,c,s))$ for all
$s\in{\cal S}$.

\emph{Step 2: Feasibility of $\bar{\pi}$.}
Let $x_t=\sigma_t(h_t)$, then
\[
\bar{\varphi}_t(h_t)=\bar{\pi}_t(\sigma_t(h_t)) = \bar{\pi}_t(x_t) \in {\cal A}_t(x_t) = {\cal A}_t(h_t),
\]
by Assumption~\ref{assu:Markov_feasibility}.

\emph{Step 3: Upper bound.}
We show by backward induction that the attainable utility is upper bounded by $\{J_t^*\}_{t=0}^T$.

\indpart{Step 3A: Terminal step} For period $t=T$, fix $h_T\in{\cal H}_T$, let $x_T=\sigma_T(h_T)$, and set $c=\varphi_T(h_T) \in {\cal A}_T(x_T)$. By Step 1, we have $U_T^\varphi(h_T)=U_{T+1}^*(\Gamma_T(x_T, c, \star))$, therefore
\[
U_T^\varphi(h_T)\leq \max_{\hat{c}\in{\cal A}_T(x_T)} U_{T+1}^*(\Gamma_T(x_T,\hat{c},\star))=J_T^*(x_T).
\]

\indpart{Step 3B: Backward induction step} Fix $t \in [T-1]$, and suppose $U_{t+1}^\varphi(h_{t+1})\leq J_{t+1}^*(\sigma_{t+1}(h_{t+1}))$ for all $h_{t+1}\in{\cal H}_{t+1}$. Then fix $h_t\in{\cal H}_t$, let $x_t=\sigma_t(h_t)$, and set $c=\varphi_t(h_t)\in{\cal A}_t(x_t)$. We have
\begin{align*}
[\tilde{u}_{t+1}^\varphi(h_t,c)](s) \leq & J_{t+1}^*(\sigma_{t+1}(\iota_t(h_t,c,s))) \\
= & J_{t+1}^*(\Gamma_t(x_t,c,s)) \\
= & [\tilde{j}_{t+1}^*(x_t,c)](s),\, \forall s \in {\cal S},
\end{align*}
where the inequality follows from the induction hypothesis.
Then, $\tilde{u}_{t+1}^\varphi(h_t,c)\leq\tilde{j}_{t+1}^*(x_t,c)$.

By Step 1 and  monotonicity of $M_t^*$ and $W_t^*$, we obtain
\begin{align*}
U_t^\varphi(h_t)
&= W_t^*(x_t,c,M_t^*(x_t,\tilde{u}_{t+1}^\varphi(h_t,c)))\\
&\leq W_t^*(x_t,c,M_t^*(x_t,\tilde{j}_{t+1}^*(x_t,c)))\\
&\leq \max_{\hat{c}\in{\cal A}_t(x_t)} W_t^*(x_t,\hat{c},M_t^*(x_t,\tilde{j}_{t+1}^*(x_t,\hat{c})))\\
&= J_t^*(x_t).
\end{align*}
The choice of $\varphi\in\Phi$ was arbitrary, so $U_t^\varphi(h_t)\leq J_t^*(\sigma_t(h_t))$ for all $\varphi\in\Phi$.

\indpart{Step 3C: Conclusion} By backward induction, we have $U_t^\varphi(h_t)\leq J_t^*(\sigma_t(h_t))$ for all $\varphi\in\Phi$, $h_t\in{\cal H}_t$, and $t\in[T]$.

\emph{Step 4: Attainment by the Bellman selector.}
We prove that the Bellman selector attains the preceding upper bound by backward induction.

\indpart{Step 4A: Terminal step} For $t=T$, fix $h_T\in{\cal H}_T$ and let $x_T=\sigma_T(h_T)$. Since $\bar{\varphi}_T(h_T)=\bar{\pi}_T(x_T)$ attains the terminal maximum, we have $U_T^{\bar{\varphi}}(h_T)=U_{T+1}^*(\Gamma_T(x_T,\bar{\pi}_T(x_T),\star))=J_T^*(x_T)$.

\indpart{Step 4B: Backward induction step} Fix $t \in [T-1]$, and suppose $U_{t+1}^{\bar{\varphi}}(h_{t+1})=J_{t+1}^*(\sigma_{t+1}(h_{t+1}))$ for all $h_{t+1}\in{\cal H}_{t+1}$. Next choose $h_t\in{\cal H}_t$, let $x_t=\sigma_t(h_t)$, and set $c=\bar{\varphi}_t(h_t)=\bar{\pi}_t(x_t)$. By Step 1 and the induction hypothesis, we have $\tilde{u}_{t+1}^{\bar{\varphi}}(h_t,c)=\tilde{j}_{t+1}^*(x_t,c)$. Consequently,
\begin{align*}
U_t^{\bar{\varphi}}(h_t)
&= W_t^*(x_t,c,M_t^*(x_t,\tilde{u}_{t+1}^{\bar{\varphi}}(h_t,c)))\\
&= W_t^*(x_t,\bar{\pi}_t(x_t),M_t^*(x_t,\tilde{j}_{t+1}^*(x_t,\bar{\pi}_t(x_t))))\\
&= J_t^*(x_t),
\end{align*}
where the last equality follows because $\bar{\pi}_t(x_t)$ attains the maximum for the period $t$ value problem.

\indpart{Step 4C: Conclusion} By backward induction, $U_t^{\bar{\varphi}}(h_t)=J_t^*(\sigma_t(h_t))$ for all $h_t \in {\cal H}_t$ and $t \in [T]$. These equalities combined with the upper bound from Step 3 prove that $\bar{\varphi} \in \Phi$ is optimal.

\section{Proofs for Section~\ref{sec:separation} (Separated Beliefs and Tastes)}

\subsection{Proof of Lemma~\ref{lem:closedness_aggregator_equivalence}}

For $t\in[T-1]$, reflexivity, symmetry, and transitivity of $\sim_t^y$ and $\sim_t^z$ follow immediately because each relation is defined by equality of the corresponding aggregator functions, so they are equivalence relations.

For each $s\in{\cal S}_t$ and $\tilde u\in{\cal L}$, define $F_{s,\tilde u}:{\cal M}_t\times{\cal M}_t\to\Re$ by $F_{s,\tilde u}(m,m')=M_t^*(s,m,\tilde u)-M_t^*(s,m',\tilde u)$, which is continuous by joint continuity of $M_t^*$. Then
\[
\{(m,m'):m\sim_t^y m'\}=\bigcap_{s\in{\cal S}_t,\tilde u\in{\cal L}}F_{s,\tilde u}^{-1}(\{0\})
\]
is closed, and so $\sim_t^y$ is closed.

For each $s\in{\cal S}_t$, $c\in{\cal C}$, and $v\in\Re$, define $G_{s,c,v}:{\cal M}_t\times{\cal M}_t\to\Re$ by $G_{s,c,v}(m,m')=W_t^*(s,m,c,v)-W_t^*(s,m',c,v)$, which is continuous by joint continuity of $W_t^*$. Then
\[
\{(m,m'):m\sim_t^z m'\}=\bigcap_{s \in {\cal S}_t,c \in {\cal C},v \in \Re}G_{s,c,v}^{-1}(\{0\}).
\]
The taste-equivalence relation is also closed as the intersection of closed sets.

For $t\in\{T,T+1\}$, belief equivalence is the universal relation ${\cal M}_t\times{\cal M}_t$, which is closed. Taste equivalence is the diagonal ${\rm diag}({\cal M}_t)\triangleq\{(m_t,m_t):m_t\in{\cal M}_t\}$, which is closed because ${\cal M}_t$ is Hausdorff.

\subsection{Proof of Lemma~\ref{lem:E_identification}}

Fix $t\in[T+1]$. By Lemma~\ref{lem:closedness_aggregator_equivalence}, $\sim_t^y$ and $\sim_t^z$ are closed equivalence relations on ${\cal M}_t$. Since ${\cal M}_t$ is compact metrizable, Theorem~\ref{thm:closed_quotient} implies that ${\cal Y}_t={\cal M}_t/\sim_t^y$ and ${\cal Z}_t={\cal M}_t/\sim_t^z$ are compact metrizable, and that the quotient maps $p_t^y:{\cal M}_t\to{\cal Y}_t$ and $p_t^z:{\cal M}_t\to{\cal Z}_t$ are continuous.

The map $\Xi_t : {\cal S}_t\times{\cal M}_t \rightarrow {\cal S}_t\times{\cal Y}_t\times{\cal Z}_t$ defined by $\Xi_t(s,m)=(s,p_t^y(m),p_t^z(m))$ is continuous because each of its coordinate maps is continuous.
We show that $\Xi_t$ is injective. Suppose $\Xi_t(x_t)=\Xi_t(x_t')$, where $x_t=(s_t,m_t)$ and $x_t'=(s_t',m_t')$. Then $s_t=s_t'$, $p_t^y(m_t)=p_t^y(m_t')$ so $m_t\sim_t^y m_t'$, and $p_t^z(m_t)=p_t^z(m_t')$ so $m_t\sim_t^z m_t'$. By Assumption~\ref{assu:exhaustiveness}, this implies $m_t=m_t'$ and $x_t=x_t'$, so $\Xi_t$ is injective.

Since ${\cal X}_t$ is compact and ${\cal S}_t\times{\cal Y}_t\times{\cal Z}_t$ is Hausdorff, the continuous bijection $\Xi_t:{\cal X}_t\to{\cal X}_t^\dagger$ is a homeomorphism by Theorem~\ref{thm:compact_to_hausdorff}.
By the definition of the canonical belief and taste summary maps, $\Xi_t(\sigma_t(h_t))=(\sigma_t^s(h_t),\sigma_t^y(h_t),\sigma_t^z(h_t))$ for every $h_t\in{\cal H}_t$. Since $\sigma_t$ is onto ${\cal X}_t$, we conclude ${\cal X}_t^\dagger=\Xi_t({\cal X}_t)=\{(\sigma_t^s(h_t),\sigma_t^y(h_t),\sigma_t^z(h_t)):h_t\in{\cal H}_t\}$, with $\sigma_{T+1}^s(h_{T+1})=\star$ at the terminal period.

\subsection{Proof of Theorem~\ref{thm:memory_recursive_separated}}

Let $(\mathfrak X^*,{\cal V}^*)$ be the selected PA factorization on the canonical PA-state system supplied by Theorem~\ref{thm:recursive_memory}. The generated utility functions $\{U_t^*\}$ satisfy the recursion $U_t^*(x_t,f)=W_t^*(x_t,c,M_t^*(x_t,\tilde{u}_{t+1}^*(x_t,f)))$ of Eq.~\eqref{eq:memory_recursive} and the factorization $U_t^*(\sigma_t(h_t),f)=U_t(h_t,f)$.

\emph{Step 1: Construct belief and taste quotients.} By Lemma~\ref{lem:closedness_aggregator_equivalence}, $\sim_t^y$ and $\sim_t^z$ are closed equivalence relations on ${\cal M}_t$. Since ${\cal M}_t$ is compact metrizable, Theorem~\ref{thm:closed_quotient} implies that ${\cal Y}_t={\cal M}_t/\sim_t^y$ and ${\cal Z}_t={\cal M}_t/\sim_t^z$ are compact metrizable, with continuous quotient maps $p_t^y:{\cal M}_t\to{\cal Y}_t$ and $p_t^z:{\cal M}_t\to{\cal Z}_t$.

Fix $t\in[T-1]$. Recall $q_t^y:{\cal X}_t\to{\cal X}_t^{\dagger,y}$ and $q_t^z:{\cal X}_t\to{\cal X}_t^{\dagger,z}$ from Definition~\ref{defn:admissible_separated_states}. Since ${\cal X}_t$ is compact and ${\cal X}_t^{\dagger,y}$ and ${\cal X}_t^{\dagger,z}$ are Hausdorff subspaces, the continuous surjections $q_t^y$ and $q_t^z$ are quotient maps by Theorem~\ref{thm:compact_to_hausdorff}.

\emph{Step 2: Separated risk aggregators.}
We first descend the risk aggregator. 
Define $Q_t^y\triangleq q_t^y\times{\rm id}_{\cal L}:{\cal X}_t\times{\cal L}\to{\cal X}_t^{\dagger,y}\times{\cal L}$. Since ${\cal L}$ is locally compact Hausdorff, Theorem~\ref{thm:whitehead} implies that $Q_t^y$ is a quotient map.

The function $M_t^*:{\cal X}_t\times{\cal L}\to\Re$ is constant on every fiber of $Q_t^y$. Choose $x_t=(s_t,m_t)$ and $x_t'=(s_t',m_t')$, and suppose $Q_t^y(x_t,\tilde u)=Q_t^y(x_t',\tilde u')$. Then $\tilde u=\tilde u'$, $s_t=s_t'$, and $p_t^y(m_t)=p_t^y(m_t')$. It follows that $m_t\sim_t^y m_t'$ and so
\[
M_t^*(s_t,m_t,\tilde u)=M_t^*(s_t,m_t',\tilde u)=M_t^*(s_t',m_t',\tilde u').
\]
By the universal property of quotient maps, there is a unique continuous function $M_t^\dagger:{\cal X}_t^{\dagger,y}\times{\cal L}\to\Re$ satisfying $M_t^\dagger\circ Q_t^y=M_t^*$.
The descended function $M_t^\dagger$ is a risk aggregator. If $\tilde u\leq\tilde{v}$ pointwise and $(s_t,y_t)\in{\cal X}_t^{\dagger,y}$, choose $m_t$ such that $(s_t,m_t)\in{\cal X}_t$ and $p_t^y(m_t)=y_t$. Then
\[
M_t^\dagger(s_t,y_t,\tilde u)=M_t^*(s_t,m_t,\tilde u)\leq M_t^*(s_t,m_t,\tilde{v})=M_t^\dagger(s_t,y_t,\tilde{v}).
\]
For every $v\in\Re$, $M_t^\dagger(s_t,y_t,{\bf v})=M_t^*(s_t,m_t,{\bf v})=v$. Then $M_t^\dagger$ is normalized on constant vectors.

\emph{Step 3: Separated time aggregators.}
We next descend the time aggregator.
Now define $Q_t^z\triangleq q_t^z\times{\rm id}_{{\cal C}\times\Re}:{\cal X}_t\times{\cal C}\times\Re\to{\cal X}_t^{\dagger,z}\times{\cal C}\times\Re$. Since ${\cal C}\times\Re$ is locally compact Hausdorff, Theorem~\ref{thm:whitehead} implies that $Q_t^z$ is a quotient map.
The function $W_t^*:{\cal X}_t\times{\cal C}\times\Re\to\Re$ is constant on every fiber of $Q_t^z$. Choose $x_t=(s_t,m_t)$ and $x_t'=(s_t',m_t')$, and suppose $Q_t^z(x_t,c,v)=Q_t^z(x_t',c',v')$. Then $c=c'$, $v=v'$, $s_t=s_t'$, and $p_t^z(m_t)=p_t^z(m_t')$. It follows that $m_t\sim_t^z m_t'$, and so
\[
W_t^*(s_t,m_t,c,v)=W_t^*(s_t,m_t',c,v)=W_t^*(s_t',m_t',c',v').
\]
By the universal property of quotient maps, there is a unique continuous function $W_t^\dagger:{\cal X}_t^{\dagger,z}\times{\cal C}\times\Re\to\Re$ satisfying $W_t^\dagger\circ Q_t^z=W_t^*$.
The descended function $W_t^\dagger$ is a time aggregator. If $v\leq v'$ and $(s_t,z_t)\in{\cal X}_t^{\dagger,z}$, choose $m_t$ such that $(s_t,m_t)\in{\cal X}_t$ and $p_t^z(m_t)=z_t$. Then
\[
W_t^\dagger(s_t,z_t,c,v)=W_t^*(s_t,m_t,c,v)\leq W_t^*(s_t,m_t,c,v')=W_t^\dagger(s_t,z_t,c,v'),
\]
so $W_t^\dagger$ is nondecreasing in its continuation-value argument.

This procedure constructs the separated time and risk aggregators at every $t\in[T-1]$. Define the terminal primitive $V_{T+1}^\dagger(\xi_{T+1}) \triangleq U_{T+1}^*(\Xi_{T+1}^{-1}(\xi_{T+1}))$ so $U_{T+1}^\dagger=V_{T+1}^\dagger$.
By Lemma~\ref{lem:E_identification}, $\Xi_{T+1}$ is a homeomorphism onto ${\cal X}_{T+1}^\dagger$, so $U_{T+1}^\dagger$ is continuous.

\emph{Step 4: Separated structure and transitions.} By Assumption~\ref{assu:exhaustiveness} and Lemma~\ref{lem:E_identification}, $\Xi_t:{\cal X}_t\to{\cal X}_t^\dagger$ is a homeomorphism, so each $\xi_t\in{\cal X}_t^\dagger$ has a unique pre-image $x_t=\Xi_t^{-1}(\xi_t) \in {\cal X}_t$.
The canonical SPA transitions are defined by
\[
\Gamma_t^\dagger = \Xi_{t+1}\circ\Gamma_t\circ (\Xi_t^{-1}\times{\rm id}_{\cal C}\times{\rm id}_{{\cal S}_{t+1}}).
\]
Since $\Xi_t$ is a homeomorphism, $\Gamma_t^\dagger$ is continuous as the composition of continuous functions.
In addition, $\Gamma_t^\dagger$ satisfies $\Gamma_t^\dagger(\Xi_t(x_t),c,s)=\Xi_{t+1}(\Gamma_t(x_t,c,s))$.
If $\xi_t=\Xi_t(x_t)$, then
\[
\varsigma_{t+1}^\dagger\bigl(\Gamma_t^\dagger(\xi_t,c,s')\bigr) =\varsigma_{t+1}^\dagger\bigl(\Xi_{t+1}(\Gamma_t(x_t,c,s'))\bigr) = \varsigma_{t+1}(\Gamma_t(x_t,c,s')) = s',
\]
which shows the SPA transition preserves the physical-state coordinate.
Then $\mathfrak{X}^\dagger = (\{{\cal Y}_t\},\{{\cal Z}_t\},\{{\cal X}_t^\dagger\},\{\varsigma_t^\dagger\},\{\Gamma_t^\dagger\})$ is an SPA-state system, and the descended aggregators define an SPA recursive structure ${\cal V}^\dagger$ on $\mathfrak{X}^\dagger$. The separated aggregators generate $\{U_t^\dagger\}_{t=0}^{T+1}$ via Eq.~\eqref{eq:memory_recursive_separated}.

\emph{Step 5: Factorization on beliefs/tastes.} We show that the separated utilities $\{U_t^\dagger\}_{t=0}^{T+1}$ are a factorization of $\{U_t^*\}_{t=0}^{T+1}$ by backward induction.

\indpart{Step 5A: Terminal step} At $t=T$,
\begin{align*}
U_T^\dagger(\Xi_T(x_T),c) = & U_{T+1}^\dagger(\Gamma_T^\dagger(\Xi_T(x_T),c,\star))\\
= & U_{T+1}^\dagger(\Xi_{T+1}(\Gamma_T(x_T,c,\star)))\\
= & U_{T+1}^*(\Gamma_T(x_T,c,\star)) = U_T^*(x_T,c).
\end{align*}

\indpart{Step 5B: Backward induction step} For the period $t\in[T-1]$ induction step, suppose the desired factorization holds for period $t+1$. Fix $f=(c,f_+)\in{\cal F}_t$, choose $x_t=(s_t,m_t)$ and let $\xi_t=\Xi_t(x_t)=(s_t,y_t,z_t)$, $y_t=p_t^y(m_t)$, and $z_t=p_t^z(m_t)$. By transition consistency and the inductive hypothesis, for all $s \in {\cal S}$,
\[
[\tilde{u}_{t+1}^\dagger(\xi_t,f)](s)=U_{t+1}^\dagger(\Xi_{t+1}(\Gamma_t(x_t,c,s)),f_+(s))
=U_{t+1}^*(\Gamma_t(x_t,c,s),f_+(s))=[\tilde{u}_{t+1}^*(x_t,f)](s),
\]
so $\tilde{u}_{t+1}^\dagger(\xi_t,f)=\tilde{u}_{t+1}^*(x_t,f)$. Using $M_t^\dagger(s_t,y_t,\cdot)=M_t^*(x_t,\cdot)$, $W_t^\dagger(s_t,z_t,\cdot,\cdot)=W_t^*(x_t,\cdot,\cdot)$, and the PA recursion,
\[
U_t^\dagger(\xi_t,f)=W_t^\dagger(s_t,z_t,c,M_t^\dagger(s_t,y_t,\tilde{u}_{t+1}^*(x_t,f))) = W_t^*(x_t,c,M_t^*(x_t,\tilde{u}_{t+1}^*(x_t,f)))=U_t^*(x_t,f).
\]

\indpart{Step 5C: Conclusion} By backward induction, we conclude that $U_t^\dagger(\Xi_t(x_t),f)=U_t^*(x_t,f)$ for all $x_t \in {\cal X}_t$ and $t \in [T]$.

\emph{Step 6: Representation.} At the terminal period, for $x_{T+1}=\sigma_{T+1}(h_{T+1})$, we have
\[
U_{T+1}^\dagger(\sigma_{T+1}^\dagger(h_{T+1})) = U_{T+1}^\dagger(\Xi_{T+1}(x_{T+1})) = U_{T+1}^*(x_{T+1}) = U_{T+1}(h_{T+1}).
\]
For all $t \in [T]$, we have $U_t^\dagger(\Xi_t(x_t),f)=U_t^*(x_t,f)$ by Step 5, and $\Xi_t\circ\sigma_t = \sigma_t^\dagger = (\sigma_t^s,\sigma_t^y,\sigma_t^z)$.
It follows that
\[
U_t^\dagger(\sigma_t^\dagger(h_t),f)=U_t^\dagger(\Xi_t(\sigma_t(h_t)),f)=U_t^*(\sigma_t(h_t),f)=U_t(h_t,f),
\]
which is Eq.~\eqref{eq:factorization_separated}. Transition consistency Eq.~\eqref{eq:transition_consistency_separated} follows by Lemma~\ref{lem:quotient_transition_continuity} and the identity $\Gamma_t^\dagger\circ(\Xi_t\times{\rm id}_{\cal C}\times{\rm id}_{{\cal S}_{t+1}}) = \Xi_{t+1}\circ\Gamma_t$ obtained in Step 4. We conclude that $(\mathfrak{X}^\dagger,{\cal V}^\dagger)$ factorizes ${\bf U}$ and represents ${\cal P}$.

\subsection{Proof of Corollary~\ref{cor:DP_separated}}

\emph{Step 1: Define candidate separated value functions.}
Let
\[
V_T^\dagger(\xi_T)\triangleq\max_{c\in{\cal A}_T^\dagger(\xi_T)}U_{T+1}^\dagger(\Gamma_T^\dagger(\xi_T,c,\star)),\, \forall \xi_T \in {\cal X}_T^\dagger.
\]
For $t\in[T-1]$, given $V_{t+1}^\dagger$, define $\tilde{v}_{t+1}^\dagger(\xi_t,c) \in {\cal L}$ by $[\tilde{v}_{t+1}^\dagger(\xi_t,c)](s)\triangleq V_{t+1}^\dagger(\Gamma_t^\dagger(\xi_t,c,s))$ for all $s \in {\cal S}$, and
\[
V_t^\dagger(\xi_t)\triangleq\max_{c\in{\cal A}_t^\dagger(\xi_t)}W_t^\dagger(s_t,z_t,c,M_t^\dagger(s_t,y_t,\tilde{v}_{t+1}^\dagger(\xi_t,c))),\, \forall \xi_t \in {\cal X}_t^\dagger.
\]

\emph{Step 2: Establish continuity and attainment.}
The terminal objective is continuous as the composition of continuous functions. Since ${\cal A}_T^\dagger$ is a continuous correspondence with nonempty compact values, Theorem~\ref{thm:berge} implies that $V_T^\dagger$ is continuous and the maximum is attained in every state. Then for $t \in [T-1]$ suppose $V_{t+1}^\dagger$ is continuous. Finiteness of ${\cal S}$ and continuity of $\Gamma_t^\dagger$ imply that $(\xi_t,c)\to\tilde{v}_{t+1}^\dagger(\xi_t,c)$ is continuous. Joint continuity of $M_t^\dagger$ and $W_t^\dagger$, with Theorem~\ref{thm:berge}, implies that $V_t^\dagger$ is continuous and its maximum is attained in every state.

\emph{Step 3: Optimal SPA policy.}
For each $\xi_T \in {\cal X}_T^\dagger$, choose
\[
\bar\pi_T^\dagger(\xi_T) \in \arg\max_{c\in{\cal A}_T^\dagger(\xi_T)}U_{T+1}^\dagger(\Gamma_T^\dagger(\xi_T,c,\star)).
\] 
Then, for each $t\in[T-1]$ and $\xi_t=(s_t,y_t,z_t) \in {\cal X}_t^\dagger$, choose
\[
\bar\pi_t^\dagger(\xi_t)\in\arg\max_{c\in{\cal A}_t^\dagger(\xi_t)}W_t^\dagger(s_t,z_t,c,M_t^\dagger(s_t,y_t,\tilde{v}_{t+1}^\dagger(\xi_t,c))).
\]
The set $\Pi^\dagger=\prod_{t=0}^T{\rm Sel}({\cal A}_t^\dagger)$ contains all pointwise feasible selectors, so $\bar\pi^\dagger\in\Pi^\dagger$.

\emph{Step 4: Upper bound.}
We show that the attainable utility is upper bounded by the separated optimal value functions.

\indpart{Step 4A: Terminal step} For period $t=T$,
\[
U_T^{\dagger\pi^\dagger}(\xi_T)=U_{T+1}^\dagger(\Gamma_T^\dagger(\xi_T,\pi_T^\dagger(\xi_T),\star))\leq V_T^\dagger(\xi_T).
\]

\indpart{Step 4B: Backward induction step} Fix $t \in [T-1]$, and suppose the upper bound holds at $t+1$ for every state and policy. Fix $\xi_t=(s_t,y_t,z_t)$ and set $c=\pi_t^\dagger(\xi_t)$, then
\[
[\tilde{u}_{t+1}^{\dagger\pi^\dagger}(\xi_t,c)](s)=U_{t+1}^{\dagger\pi^\dagger}(\Gamma_t^\dagger(\xi_t,c,s))\leq V_{t+1}^\dagger(\Gamma_t^\dagger(\xi_t,c,s))=[\tilde{v}_{t+1}^\dagger(\xi_t,c)](s),\, \forall s\in{\cal S}.
\]
By monotonicity of $M_t^\dagger(s_t,y_t,\cdot)$ and $W_t^\dagger(s_t,z_t,c,\cdot)$, we then have
\[
U_t^{\dagger\pi^\dagger}(\xi_t)\leq W_t^\dagger(s_t,z_t,c,M_t^\dagger(s_t,y_t,\tilde{v}_{t+1}^\dagger(\xi_t,c)))\leq V_t^\dagger(\xi_t).
\]

\indpart{Step 4C: Conclusion} By backward induction, $U_t^{\dagger\pi^\dagger}(\xi_t)\leq V_t^\dagger(\xi_t)$ for every $\pi^\dagger\in\Pi^\dagger$, $\xi_t\in{\cal X}_t^\dagger$, and $t\in[T]$.

\emph{Step 5: Simultaneous attainment by $\bar\pi^\dagger$.}
The policy $\bar\pi^\dagger$ attains $V_t^\dagger$ simultaneously from every separated state.

\indpart{Step 5A: Terminal step} The terminal equality follows from the definition of $\bar\pi_T^\dagger$.

\indpart{Step 5B: Backward induction step} For $t \in [T-1]$, suppose $U_{t+1}^{\dagger\bar\pi^\dagger}(\xi_{t+1})=V_{t+1}^\dagger(\xi_{t+1})$ for every $\xi_{t+1}\in{\cal X}_{t+1}^\dagger$. For $\xi_t=(s_t,y_t,z_t)$, let $c=\bar\pi_t^\dagger(\xi_t)$, then $\tilde{u}_{t+1}^{\dagger\bar\pi^\dagger}(\xi_t,c)=\tilde{v}_{t+1}^\dagger(\xi_t,c)$. By choice of $\bar\pi_t^\dagger(\xi_t)$,
\[
U_t^{\dagger\bar\pi^\dagger}(\xi_t)=W_t^\dagger(s_t,z_t,\bar\pi_t^\dagger(\xi_t),M_t^\dagger(s_t,y_t,\tilde{v}_{t+1}^\dagger(\xi_t,\bar\pi_t^\dagger(\xi_t))))=V_t^\dagger(\xi_t).
\]

\indpart{Step 5C: Conclusion} By backward induction, $U_t^{\dagger\bar\pi^\dagger}(\xi_t)=V_t^\dagger(\xi_t)$ for every $\xi_t\in{\cal X}_t^\dagger$ and $t\in[T]$.

\emph{Step 6: Bellman recursion.}
To conclude, we have
\[
J_t^\dagger(\xi_t)=\sup_{\pi^\dagger\in\Pi^\dagger}U_t^{\dagger\pi^\dagger}(\xi_t)=U_t^{\dagger\bar\pi^\dagger}(\xi_t)=V_t^\dagger(\xi_t),\, \forall \xi_t \in {\cal X}_t^\dagger,\, t \in [T].
\]
Substituting $J_t^\dagger=V_t^\dagger$ into the recursion defining $V_t^\dagger$ yields Eq.~\eqref{eq:Bellman_separated}, and $\bar\pi^\dagger$ simultaneously attains the separated value function from every state.

\subsection{Proof of Corollary~\ref{cor:DP_verification_separated}}

\emph{Step 1: Policy-value identity.}
Fix $\varphi\in\Phi$ and let $f^{\varphi}[h_t]\in{\cal F}_t$ so $U_t^\varphi(h_t)=U_t(h_t,f^{\varphi}[h_t])$. Because $(\mathfrak{X}^\dagger,{\cal V}^\dagger)$ factorizes ${\bf U}$, Eq.~\eqref{eq:factorization_separated} gives $U_t^\varphi(h_t)=U_t^\dagger(\xi_t,f^{\varphi}[h_t])$ for $\xi_t=\sigma_t^\dagger(h_t)=(s_t,y_t,z_t)$.
For $t =T$, we have $U_T^\varphi(h_T)=U_{T+1}^\dagger(\Gamma_T^\dagger(\sigma_T^\dagger(h_T),\varphi_T(h_T),\star))$.
For all $t\in[T-1]$, expanding the separated recursion Eq.~\eqref{eq:memory_recursive_separated} and using transition consistency Eq.~\eqref{eq:transition_consistency_separated} yields
\[
U_t^\varphi(h_t)=W_t^\dagger(s_t,z_t,\varphi_t(h_t),M_t^\dagger(s_t,y_t,\tilde{u}_{t+1}^\varphi(h_t,\varphi_t(h_t)))).
\]

\emph{Step 2: Feasibility of $\bar\varphi$.}
For fixed $h_t\in{\cal H}_t$, we have $\bar\varphi_t(h_t)=\bar\pi_t^\dagger(\xi_t)\in {\cal A}_t^\dagger(\xi_t) = {\cal A}_t(h_t)$ by Assumption~\ref{assu:separated_Markov_feasibility}. Then $\bar\varphi\in\Phi$.

\emph{Step 3: Upper bound.}
We show that the separated optimal value functions upper bound the attainable utility.

\indpart{Step 3A: Terminal step} At $t=T$, set $c=\varphi_T(h_T)\in{\cal A}_T^\dagger(\xi_T)$. By Step 1, we have
\[
U_T^\varphi(h_T)=U_{T+1}^\dagger(\Gamma_T^\dagger(\xi_T,c,\star))\leq\max_{\hat{c}\in{\cal A}_T^\dagger(\xi_T)}U_{T+1}^\dagger(\Gamma_T^\dagger(\xi_T,\hat{c},\star))=J_T^\dagger(\xi_T),
\]
by Eq.~\eqref{eq:Bellman_separated}.

\indpart{Step 3B: Backward induction step.} Fix $t \in [T-1]$, suppose the bound holds for period $t+1$, and set $c=\varphi_t(h_t)\in{\cal A}_t^\dagger(\xi_t)$. For each $s\in{\cal S}$, transition consistency Eq.~\eqref{eq:transition_consistency_separated} gives $\sigma_{t+1}^\dagger(\iota_t(h_t,c,s))=\Gamma_t^\dagger(\xi_t,c,s)$,
so by the induction hypothesis
\[
[\tilde{u}_{t+1}^\varphi(h_t,c)](s)=U_{t+1}^\varphi(\iota_t(h_t,c,s))\leq J_{t+1}^\dagger(\Gamma_t^\dagger(\xi_t,c,s))=[\tilde{j}_{t+1}^\dagger(\xi_t,c)](s).
\]
It follows that $\tilde{u}_{t+1}^\varphi(h_t,c)\leq\tilde{j}_{t+1}^\dagger(\xi_t,c)$. By Step 1, and monotonicity of $M_t^\dagger(s_t,y_t,\cdot)$ and $W_t^\dagger(s_t,z_t,c,\cdot)$, we obtain
\begin{align*}
U_t^\varphi(h_t) & = W_t^\dagger\big(s_t,z_t,c,M_t^\dagger(s_t,y_t,\tilde{u}_{t+1}^\varphi(h_t,c))\big)\\
& \leq W_t^\dagger\big(s_t,z_t,c,M_t^\dagger(s_t,y_t,\tilde{j}_{t+1}^\dagger(\xi_t,c))\big)\\
& \leq \max_{\hat{c}\in{\cal A}_t^\dagger(\xi_t)} W_t^\dagger\big(s_t,z_t,\hat{c},M_t^\dagger(s_t,y_t,\tilde{j}_{t+1}^\dagger(\xi_t,\hat{c}))\big)= J_t^\dagger(\xi_t),
\end{align*}
where the last equality is by Eq.~\eqref{eq:Bellman_separated}. Since $\varphi$ was arbitrary, we have $U_t^\varphi(h_t)\leq J_t^\dagger(\xi_t)$ for all $\varphi\in\Phi$, $h_t\in{\cal H}_t$, and $t\in[T]$.

\indpart{Step 3C: Conclusion} By backward induction, $U_t^\varphi(h_t)\leq J_t^\dagger(\sigma_t^\dagger(h_t))$ for all $\varphi\in\Phi$, $h_t\in{\cal H}_t$, and $t\in[T]$.
 
\emph{Step 4: Attainment by the Bellman selector.}
We now show that the optimal value is attained by backward induction.

\indpart{Step 4A: Terminal step} For $t=T$ and $h_T \in {\cal H}_T$, $\bar\varphi_T(h_T)=\bar\pi_T^\dagger(\xi_T)$ attains the terminal maximum, so
\[
U_T^{\bar\varphi}(h_T)=U_{T+1}^\dagger(\Gamma_T^\dagger(\xi_T,\bar\pi_T^\dagger(\xi_T),\star))=J_T^\dagger(\xi_T).
\]

\indpart{Step 4B: Backward induction step} Fix $t \in [T-1]$, suppose the equality holds for period $t+1$, and set $c=\bar\varphi_t(h_t)=\bar\pi_t^\dagger(\xi_t)$. Transition consistency and the induction hypothesis give $\tilde{u}_{t+1}^{\bar\varphi}(h_t,c)=\tilde{j}_{t+1}^\dagger(\xi_t,c)$. By Step 1,
\[
U_t^{\bar\varphi}(h_t) = W_t^\dagger(s_t,z_t,\bar\pi_t^\dagger(\xi_t),M_t^\dagger(s_t,y_t,\tilde{j}_{t+1}^\dagger(\xi_t,\bar\pi_t^\dagger(\xi_t)))) = J_t^\dagger(\xi_t),
\]
where the second equality holds because $\bar\pi_t^\dagger(\xi_t)$ attains the maximum in Eq.~\eqref{eq:Bellman_separated}.

\indpart{Step 4C: Conclusion} By backward induction, $U_t^{\bar\varphi}(h_t)=J_t^\dagger(\sigma_t^\dagger(h_t))$ for all $h_t\in{\cal H}_t$ and $t\in[T]$.
These equalities combined with the upper bound from Step 3 prove that $\bar{\varphi} \in \Phi$ is optimal.

\section{Glossary of Notation}
\label{sec:glossary}

This section collects notation for ease of reference.

\begingroup
\small
\renewcommand{\arraystretch}{1.16}
\setlength{\tabcolsep}{5pt}

\begin{center}
\captionof{table}{Time, physical states, and primitive spaces}
\label{tab:notation_primitives}
\begin{tabular}{p{0.27\textwidth} p{0.66\textwidth}}
\textbf{Symbol} & \textbf{Description} \\
\hline
$[T]$ & Time index set $\{0,1,\ldots,T\}$. \\
$[n_1:n_2]$ & Integer index set $\{n_1,n_1+1,\ldots,n_2\}$. \\
$\bar s_0$ & Fixed initial physical state. \\
${\cal S}$ & Finite nonterminal physical state space. \\
${\cal S}_t$ & Physical state space at period $t$: ${\cal S}_0=\{\bar s_0\}$, ${\cal S}_t={\cal S}$ for $t=1,\ldots,T$, and ${\cal S}_{T+1}=\{\star\}$. \\
$\star$ & Dummy terminal physical state. \\
$s_t\in{\cal S}_t$ & Physical Markov state observed at the beginning of period $t$. \\
$q_t(\cdot\vert s_t)$ & Reference transition kernel for the physical Markov process, for $t\in[T-1]$. \\
${\cal C}=[0,\bar c]$ & Compact set of admissible nonnegative consumption levels. \\
$c_t\in{\cal C}$ & Consumption chosen during period $t$. \\
${\cal L}\cong\Re^{\cal S}$ & Space of nonterminal state-contingent continuation-utility vectors, equipped with the supremum norm and pointwise order. \\
$\tilde u,\tilde{v}\in{\cal L}$ & Generic continuation-utility vectors. \\
${\bf v}\in{\cal L}$ & Constant continuation-utility vector with value $v$ in every nonterminal state. \\
$\Delta({\cal S})$ & Probability simplex over ${\cal S}$. \\
${\rm id}_{\cal X}$ & Identity map on a set ${\cal X}$. \\
${\rm Sel}({\cal A})$ & Set of selectors $g$ of a correspondence ${\cal A}$, satisfying $g(e)\in{\cal A}(e)$. \\
\end{tabular}
\end{center}

\begin{center}
\captionof{table}{Histories, plans, and full-history objects}
\label{tab:notation_histories_plans}
\begin{tabular}{p{0.27\textwidth} p{0.66\textwidth}}
\textbf{Symbol} & \textbf{Description} \\
\hline
${\cal H}_t$ & Period-$t$ history space. \\
$h_t\in{\cal H}_t$ & History $h_t=(\bar s_0,c_0,\ldots,s_{t-1},c_{t-1},s_t)$, for $t\in[1,T]$. \\
${\cal H}_{T+1}$ & Terminal history space ${\cal H}_T\times{\cal C}$. \\
$\sigma_t^s:{\cal H}_t\to{\cal S}_t$ & Current physical-state projection of a history. \\
$\iota_t$ & One-step history-extension map. For $t\in[T-1]$, $\iota_t(h_t,c,s')=(h_t,c,s')$; for $t=T$, $\iota_T(h_T,c,\star)=(h_T,c)$. \\
$\mathfrak H$ & Full-history state system $(\{{\cal H}_t\},\{\sigma_t^s\},\{\iota_t\})$. \\
${\cal F}_t$ & Space of continuation plans from period $t$. \\
$f=(c,f_+)\in{\cal F}_t$ & Plan decomposed into current consumption $c$ and next-period state-contingent tail $f_+$. \\
$f_+:{\cal S}\to{\cal F}_{t+1}$ & Tail plan assigning a continuation plan to each next nonterminal physical state. \\
$c_{t:T}\in{\cal F}_t$ & Deterministic consumption plan independent of future states. \\
${\cal D}_t={\cal H}_t\times{\cal F}_t$ & Full-history grand domain of life paths. \\
$f_t^\varphi[h_t]$ & Continuation plan induced from history $h_t$ by following a history-dependent policy $\varphi$. \\
\end{tabular}
\end{center}

\begin{center}
\captionof{table}{Preferences and full-history recursive representation}
\label{tab:notation_preferences}
\begin{tabular}{p{0.27\textwidth} p{0.66\textwidth}}
\textbf{Symbol} & \textbf{Description} \\
\hline
${\cal P}$ & Grand preference system $(\{\succeq_{(t)}\}_{t=0}^T,V_{T+1})$. \\
$\succeq_{(t)}$ & Grand preference relation on ${\cal D}_t$. \\
$\succ_{(t)},\sim_{(t)}$ & Strict preference and indifference induced by $\succeq_{(t)}$. \\
$V_{T+1}$ & Primitive terminal value function on ${\cal H}_{T+1}$. \\
$\succeq_{h_t}$ & Conditional preference on ${\cal F}_t$ at history $h_t$. \\
${\bf U}=\{U_t\}_{t=0}^{T+1}$ & Compatible utility system. \\
$U_t:{\cal H}_t\times{\cal F}_t\to\Re$ & Utility representation of $\succeq_{(t)}$, for $t\in[T]$. \\
$U_{T+1}:{\cal H}_{T+1}\to\Re$ & Terminal utility, equal to $V_{T+1}$ under compatibility. \\
$U_{h_t}(\cdot)$ & Conditional utility slice $U_t(h_t,\cdot)$. \\
$W_t(h_t,c,v)$ & Full-history time aggregator. \\
$M_t(h_t,\tilde u)$ & Full-history risk aggregator. \\
${\cal V}$ & Full-history recursive structure on $\mathfrak H$. \\
$V_{T+1}^{\cal V}$ & Terminal primitive in a full-history recursive structure. \\
${\bf U}^{\cal V}=\{U_t^{\cal V}\}$ & Utility system generated by $(\mathfrak H,{\cal V})$. \\
$\tilde u_{t+1}^{\cal V}(h_t,f)$ & Next-period continuation-utility vector generated by $(\mathfrak H,{\cal V})$. \\
$\tilde u_{t+1}(h_t,f)$ & One-step continuation-utility vector generated by a fixed compatible utility system ${\bf U}$. \\
$\Lambda_t^{\bf U}(h_t,c)$ & Attainable one-step continuation-utility vectors at history $h_t$ and current consumption $c$. \\
$\Lambda_t^{\bf U}(h_t)$ & Union $\bigcup_{c\in{\cal C}}\Lambda_t^{\bf U}(h_t,c)$. \\
$\Lambda_t^{\bf U}$ & Union $\bigcup_{h_t\in{\cal H}_t}\Lambda_t^{\bf U}(h_t)$. \\
$D_t^M$ & Effective domain of history-dependent risk evaluation, $\{(h_t,\tilde u):\tilde u\in\Lambda_t^{\bf U}(h_t)\}$. \\
$M_t^0:D_t^M\to\Re$ & Effective certainty-equivalent risk aggregator on $D_t^M$. \\
\end{tabular}
\end{center}

\begin{center}
\captionof{table}{Preference-augmented states and canonical quotients}
\label{tab:notation_pa_states}
\begin{tabular}{p{0.27\textwidth} p{0.66\textwidth}}
\textbf{Symbol} & \textbf{Description} \\
\hline
$\mathfrak{X}$ & Generic PA-state system $(\{{\cal X}_t\},\{\varsigma_t\},\{\Gamma_t\})$. \\
${\cal X}_t$ & PA-state space; in the canonical construction, ${\cal X}_t={\cal H}_t/\odot_t$. \\
$x_t\in{\cal X}_t$ & PA state. \\
$\varsigma_t:{\cal X}_t\to{\cal S}_t$ & PA physical-state readout map. \\
$\Gamma_t:{\cal X}_t\times{\cal C}\times{\cal S}_{t+1}\to{\cal X}_{t+1}$ & PA transition map. \\
$\sigma_t:{\cal H}_t\to{\cal X}_t$ & Summary map from histories to PA states; in the canonical construction, the quotient projection. \\
${\cal V}^*$ & Recursive structure on a PA-state system. \\
$V_{T+1}^*$ & Terminal primitive in a PA recursive structure. \\
$W_t^*(x_t,c,v)$ & PA time aggregator. \\
$M_t^*(x_t,\tilde u)$ & PA risk aggregator. \\
$U_t^*:{\cal X}_t\times{\cal F}_t\to\Re$ & Utility generated by $(\mathfrak{X},{\cal V}^*)$, for $t\in[T]$. \\
$U_{T+1}^*:{\cal X}_{T+1}\to\Re$ & Terminal PA utility generated by $(\mathfrak{X},{\cal V}^*)$. \\
$\odot_t$ & Canonical PA-state equivalence relation on ${\cal H}_t$. \\
$\mathfrak{X}^*$ & Canonical PA-state system. \\
$U_t^\circ:{\cal X}_t\times{\cal F}_t\to\Re$ & Canonical quotient utility induced by $U_t$, for $t\in[T]$. \\
$U_{T+1}^\circ:{\cal X}_{T+1}\to\Re$ & Terminal canonical quotient utility. \\
${\cal X}_t(s)$ & Fiber of the canonical PA state over physical state $s$, $\varsigma_t^{-1}(s)$. \\
${\cal X}_t'$ & PA-state space of an alternative recursive factorization. \\
$\sigma_t'$ & Summary map of an alternative recursive factorization. \\
$\hat{\cal X}_t'=\sigma_t'({\cal H}_t)$ & Reachable part of an alternative PA-state space. \\
$\theta_t:\hat{\cal X}_t'\to{\cal X}_t$ & Memory correspondence map from an alternative factorization to the canonical PA state. \\
\end{tabular}
\end{center}

\begin{center}
\captionof{table}{Rectangular PA states and separated PA states}
\label{tab:notation_spa_states}
\begin{tabular}{p{0.27\textwidth} p{0.66\textwidth}}
\textbf{Symbol} & \textbf{Description} \\
\hline
${\cal M}_t$ & Preference-memory state space under rectangularity. \\
$m_t\in{\cal M}_t$ & Preference-memory coordinate. \\
$\kappa_t:{\cal X}_t\to{\cal S}_t\times{\cal M}_t$ & Rectangularization homeomorphism of the canonical PA state. \\
$x_t=(s_t,m_t)$ & Rectangular PA state under the chosen rectangularization. \\
$\Gamma_t^m$ & Preference-memory component of the PA transition under rectangular coordinates. \\
$\sim_t^y$ & Aggregator-induced belief equivalence relation on ${\cal M}_t$. \\
$\sim_t^z$ & Aggregator-induced taste equivalence relation on ${\cal M}_t$. \\
${\cal Y}_t={\cal M}_t/\sim_t^y$ & Belief memory-coordinate space. \\
${\cal Z}_t={\cal M}_t/\sim_t^z$ & Taste memory-coordinate space. \\
$y_t\in{\cal Y}_t$ & Belief, risk-evaluation, or ambiguity coordinate. \\
$z_t\in{\cal Z}_t$ & Taste, habit, wealth, or reference coordinate. \\
$p_t^y:{\cal M}_t\to{\cal Y}_t$ & Belief quotient map. \\
$p_t^z:{\cal M}_t\to{\cal Z}_t$ & Taste quotient map. \\
$\Xi_t:{\cal X}_t\to{\cal S}_t\times{\cal Y}_t\times{\cal Z}_t$ & SPA-state map, $\Xi_t(s_t,m_t)=(s_t,p_t^y(m_t),p_t^z(m_t))$. \\
${\cal X}_t^\dagger=\Xi_t({\cal X}_t)$ & Separated PA-state space, or SPA-state space. \\
$\xi_t=(s_t,y_t,z_t)\in{\cal X}_t^\dagger$ & SPA state. \\
$\mathfrak{X}^\dagger$ & SPA-state system. \\
$\varsigma_t^\dagger$ & SPA physical-state readout map, $\varsigma_t^\dagger(s_t,y_t,z_t)=s_t$. \\
$r_t^y(\xi_t)=(s_t,y_t)$ & Projection from an SPA state to physical-belief coordinates. \\
$r_t^z(\xi_t)=(s_t,z_t)$ & Projection from an SPA state to physical-taste coordinates. \\
${\cal X}_t^{\dagger,y}$ & Physical-belief domain $r_t^y({\cal X}_t^\dagger)$. \\
${\cal X}_t^{\dagger,z}$ & Physical-taste domain $r_t^z({\cal X}_t^\dagger)$. \\
$q_t^y:{\cal X}_t\to{\cal X}_t^{\dagger,y}$ & Map $q_t^y(s_t,m_t)=(s_t,p_t^y(m_t))$. \\
$q_t^z:{\cal X}_t\to{\cal X}_t^{\dagger,z}$ & Map $q_t^z(s_t,m_t)=(s_t,p_t^z(m_t))$. \\
$\sigma_t^y,\sigma_t^z$ & Canonical belief and taste summary maps from histories. \\
$\sigma_t^\dagger$ & Canonical SPA summary map, $\sigma_t^\dagger=\Xi_t\circ\sigma_t$. \\
$\ast$ & Singleton inactive belief, taste, or memory coordinate. \\
\end{tabular}
\end{center}

\begin{center}
\captionof{table}{SPA recursive structures}
\label{tab:notation_spa_recursive}
\begin{tabular}{p{0.27\textwidth} p{0.66\textwidth}}
\textbf{Symbol} & \textbf{Description} \\
\hline
$\Gamma_t^\dagger:{\cal X}_t^\dagger\times{\cal C}\times{\cal S}_{t+1}\to{\cal X}_{t+1}^\dagger$ & SPA-state transition map. \\
$\Gamma_t^y,\Gamma_t^z$ & Belief and taste components of the SPA transition. \\
${\cal V}^\dagger$ & SPA recursive structure on $\mathfrak{X}^\dagger$. \\
$V_{T+1}^\dagger$ & Terminal primitive in an SPA recursive structure. \\
$W_t^\dagger:{\cal X}_t^{\dagger,z}\times{\cal C}\times\Re\to\Re$ & SPA time aggregator. \\
$M_t^\dagger:{\cal X}_t^{\dagger,y}\times{\cal L}\to\Re$ & SPA risk aggregator. \\
$U_t^\dagger:{\cal X}_t^\dagger\times{\cal F}_t\to\Re$ & Utility generated by $(\mathfrak{X}^\dagger,{\cal V}^\dagger)$, for $t\in[T]$. \\
$U_{T+1}^\dagger:{\cal X}_{T+1}^\dagger\to\Re$ & Terminal separated utility. \\
$\tilde u_{t+1}^\dagger(\xi_t,f)$ & Next-period continuation-utility vector generated by an SPA recursive structure. \\
\end{tabular}
\end{center}

\begin{center}
\captionof{table}{Policies and dynamic programming}
\label{tab:notation_dp}
\begin{tabular}{p{0.27\textwidth} p{0.66\textwidth}}
\textbf{Symbol} & \textbf{Description} \\
\hline
${\cal A}_t:{\cal H}_t\rightrightarrows{\cal C}$ & History-dependent feasible-consumption correspondence. \\
${\cal A}_t(x_t)$ & PA-state feasible-consumption correspondence under Markov feasibility. \\
${\cal A}_t^\dagger(\xi_t)$ & SPA-state feasible-consumption correspondence. \\
$\varphi=(\varphi_t)_{t=0}^T$ & History-dependent policy. \\
$\Phi=\prod_{t=0}^T{\rm Sel}({\cal A}_t)$ & Set of feasible history-dependent policies. \\
$U_t^\varphi(h_t)$ & Utility from following history-dependent policy $\varphi$ starting at $h_t$. \\
$J_t(h_t)$ & Full-history optimal value function. \\
$\pi=(\pi_t)_{t=0}^T$ & PA-state Markov policy. \\
$\Pi=\prod_{t=0}^T{\rm Sel}({\cal A}_t)$ & Set of feasible PA-state Markov policies. \\
$\varphi^\pi$ & History-dependent policy induced by a PA-state Markov policy $\pi$. \\
$U_t^{*\pi}(x_t)$ & Utility from following PA-state policy $\pi$ starting at $x_t$. \\
$\tilde u_{t+1}^{*\pi}(x_t,c)$ & Continuation-utility vector under PA-state policy $\pi$. \\
$J_t^*(x_t)$ & PA-state Bellman value function. \\
$\tilde j_{t+1}^*(x_t,c)$ & Bellman continuation-value vector on the PA state. \\
$\bar\pi$ & PA-state Bellman selector. \\
$\bar\varphi$ & History-dependent policy induced by a PA or SPA Bellman selector. \\
$\pi^\dagger=(\pi_t^\dagger)_{t=0}^T$ & SPA-state Markov policy. \\
$\Pi^\dagger=\prod_{t=0}^T{\rm Sel}({\cal A}_t^\dagger)$ & Set of feasible SPA-state Markov policies. \\
$U_t^{\dagger\pi^\dagger}(\xi_t)$ & Utility from following SPA-state policy $\pi^\dagger$. \\
$\tilde u_{t+1}^{\dagger\pi^\dagger}(\xi_t,c)$ & Continuation-utility vector under SPA-state policy $\pi^\dagger$. \\
$J_t^\dagger(\xi_t)$ & SPA-state Bellman value function. \\
$\tilde j_{t+1}^\dagger(\xi_t,c)$ & Bellman continuation-value vector on the SPA state. \\
$\bar\pi^\dagger$ & SPA-state Bellman selector. \\
\end{tabular}
\end{center}

\begin{center}
\captionof{table}{Common notation in examples}
\label{tab:notation_examples}
\begin{tabular}{p{0.27\textwidth} p{0.66\textwidth}}
\textbf{Symbol} & \textbf{Description} \\
\hline
$\mathbb E_{s_t}[\tilde u]$ & Expectation of $\tilde u$ with respect to the reference kernel $q_t(\cdot\vert s_t)$. \\
$u(c)$ & Instantaneous consumption utility. \\
$\beta\in(0,1]$ & Discount factor. \\
$\alpha,\alpha_1,\alpha_2$ & Persistence or updating parameters, depending on the example. \\
$\rho,\gamma,\theta$ & Curvature, risk, ambiguity, or share parameters, depending on the example. \\
$\phi,\psi$ & Transformations used in smooth ambiguity and related examples. \\
$\Theta$ & Set of candidate transition models in Bayesian or smooth-ambiguity examples. \\
$p_\theta(\cdot\vert s)$ & Candidate transition kernel indexed by model $\theta$. \\
$\nu_t$ & Prior over transition models in the fixed-prior smooth-ambiguity example. \\
$l_t$ & Density perturbation of the reference kernel in Markov risk-measure examples. \\
$\Delta_t(s_t)$ & Set of admissible density perturbations at physical state $s_t$. \\
${\cal U}_t(s_t)$ & Uncertainty set of density perturbations at physical state $s_t$. \\
$\zeta_t$ & Latent regime in hidden-regime examples. \\
$P_{ij}$ & Regime transition probability ${\rm Pr}(\zeta_{t+1}=j\vert\zeta_t=i)$. \\
$\widehat y_{t+1}$ & One-step forecast probability of the good regime. \\
$\widehat p(\cdot\vert s_t,y_t)$ & Predictive distribution over next physical states under hidden-regime beliefs. \\
$R(\cdot),Y(\cdot)$ & Return and income functions in wealth examples. \\
$g(\cdot)$ & Aggregate-growth function in the Campbell--Cochrane example. \\
$m^\star$ & Steady-state memory or surplus value in entangled PA examples. \\
$\bar m$ & Wealth target in the wealth-target example. \\
$\tau(\cdot)$ & Risk-tolerance function in the Campbell--Cochrane example. \\
$\lambda$ & Penalty or sensitivity parameter, depending on the example. \\
$\Pi_I$ &
Clipping map onto a compact interval $I$.\\
\end{tabular}
\end{center}

\endgroup

\end{document}